\documentclass[english]{article}
\usepackage[T1]{fontenc}
\usepackage[latin9]{inputenc}
\usepackage{color}
\usepackage{babel}
\usepackage{amsthm}
\usepackage{amsmath}
\usepackage{amssymb}
\usepackage{graphicx}
\usepackage[authoryear]{natbib}
\usepackage[unicode=true,pdfusetitle,
 bookmarks=true,bookmarksnumbered=false,bookmarksopen=false,
 breaklinks=true,pdfborder={0 0 0},backref=false,colorlinks=true]
 {hyperref}

\makeatletter

\newcommand{\lyxdot}{.}

  \theoremstyle{definition}
  \newtheorem{defn}{\protect\definitionname}
  \theoremstyle{remark}
  \newtheorem{rem}{\protect\remarkname}
  \theoremstyle{plain}
  \newtheorem{prop}{\protect\propositionname}
  \theoremstyle{plain}
  \newtheorem{cor}{\protect\corollaryname}
\theoremstyle{plain}
\newtheorem{thm}{\protect\theoremname}

\@ifundefined{date}{}{\date{}}
\usepackage{amsthm}\usepackage{mathrsfs}\usepackage{bm}\usepackage{fullpage}\usepackage{setspace}\usepackage{subfigure}\usepackage{color}

\bibpunct{(}{)}{;}{a}{,}{,}

\usepackage{setspace}\usepackage[stable]{footmisc}
 \usepackage{subfigure}

\theoremstyle{plain}
\hypersetup{citecolor=blue}

 \DeclareRobustCommand*\textsubscript[1]{%
          \@textsubscript{\selectfont#1}}
        \def\@textsubscript#1{%
          {\m@th\ensuremath{_{\mbox{\fontsize\sf@size\z@#1}}}}}

\theoremstyle{remark}

\title{Parsimonious shooting heuristic for trajectory control of connected automated traffic part I: Theoretical analysis with generalized time geography}
\author{Fang Zhou$^{\mbox{a}}$, Xiaopeng Li$^{\mbox{a}}$\footnote{Corresponding author. Tel: 662-325-7196, E-mail: xli@cee.msstate.edu.}, Jiaqi Ma$^{\mbox{b}}$\\ 
a. Department of Civil and Environmental Engineering,\\
\quad Mississippi State University, MS 39762, USA\\
b. Transportation Solutions and Technology Applications Division,\\
\quad  Leidos, Inc., Reston, VA 20190, USA
}

\makeatother

  \providecommand{\definitionname}{Definition}
  \providecommand{\propositionname}{Proposition}
  \providecommand{\remarkname}{Remark}
\providecommand{\corollaryname}{Corollary}
\providecommand{\theoremname}{Theorem}

\begin{document}
\maketitle
\begin{abstract}
This paper studies a problem of controlling trajectories of a platoon
of vehicles on a highway section with advanced connected and automated
vehicle technologies. This problem is very complex because each vehicle
trajectory is essentially an infinite-dimensional object and neighboring
trajectories have complex interactions (e.g., car-following behavior).
A parsimonious shooting heuristic algorithm is proposed to construct
vehicle trajectories on a signalized highway section that comply with
the boundary condition for vehicle arrivals, vehicle mechanical limits,
traffic lights and vehicle following safety. This algorithm breaks
each vehicle trajectory into a few segments that each is analytically
solvable. This essentially decomposes the original hard trajectory
control problem to a simple constructive heuristic. Then we slightly
adapt this shooting heuristic algorithm to one that can efficiently
solve the leading vehicle problem on an uninterrupted freeway. To
study theoretical properties of the proposed algorithms, the time
geography theory is generalized by considering finite accelerations.
With this generalized theory, it is found that under mild conditions,
these algorithms can always obtain a feasible solution to the original
complex trajectory control problem. Further, we discover that the
shooting heuristic solution is a generalization of the solution to
the classic kinematic wave theory by incorporating finite accelerations.
We identify the theoretical bounds to the difference between the shooting
heuristic solution and the kinematic wave solution. Numerical experiments
are conducted to verify the theoretical results and to draw additional
managerial insights into the potential of trajectory control in improving
traffic performance. In summary, this paper provides a methodological
and theoretical foundation for advanced traffic control by optimizing
the trajectories of connected and automated vehicles. Built upon this
foundation, an optimization framework will be presented in a following
paper as Part II of this study. 
\end{abstract}

\section{Introduction}

\subsection{Background}

As illustrated by the trajectories in the time-space diagram in Figure
\ref{fig:motivating_example}(a), traffic on a signalized arterial
is usually forced to decelerate and accelerate abruptly as a result
of alternating green and red lights. When traffic density is relatively
high, stop-and-go traffic patterns will be formed and propagated backwards
along so-called shock waves. Similar stop-and-go traffic also occurs
frequently on freeways even without explicit signal interruptions.
Such stop-and-go traffic imposes a number an adverse impacts to highway
performance. Obviously, vehicles engaged in abrupt stop-and-go movements
are exposed to a high crash risk \citep{hoffmann1994drivers}, not
to mention extra discomfort to drivers \citep{beard2013discomfort}.
Also, frequent decelerations and accelerations cause excessive fuel
consumption and emissions \citep{li2014stop}, which pose a severe
threat to the urban environment. Further, when vehicles slow down
or stop, the corresponding traffic throughput decreases and the highway
capacity drops \citep{Cassidy1999}, which can cause excessive travel
delay.

\begin{figure}
\begin{centering}
\includegraphics{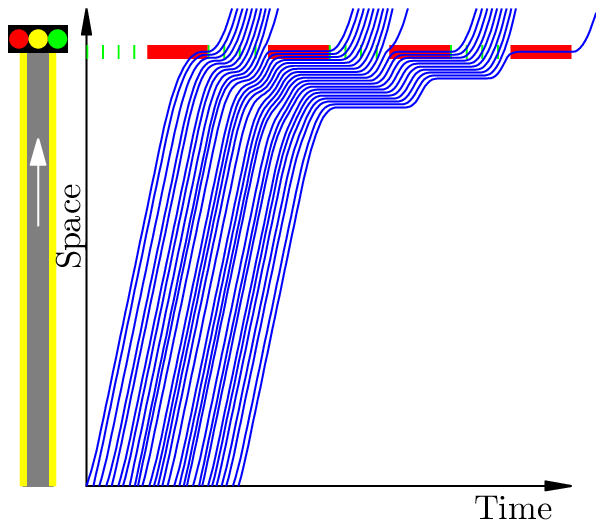}\includegraphics{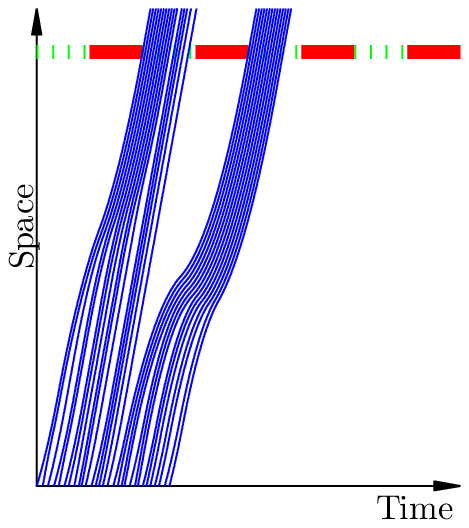}
\par\end{centering}

\begin{centering}
(a)\quad{}\quad{}\quad{}\quad{}\quad{}\quad{}\quad{}\quad{}\quad{}\quad{}\quad{}\quad{}\quad{}\quad{}(b)
\par\end{centering}

\protect\caption{Vehicle trajectories along a highway section upstream of an signalized
intersection: (a) benchmark manual vehicle trajectories; (b) smoothed
automated vehicle trajectories. \label{fig:motivating_example}}
\end{figure}

Although stop-and-go traffic has been intensively studied in the context
of freeway traffic with either theoretical models (e.g., \citet{Herman1958,Bando1995,Li09c})
or empirical observations (e.g., \citet{Kuhne1987,Kerner1996,Mauch2002,ahn2005,Li09c,Laval11}),
few studies had investigated how to smooth traffic and alleviate corresponding
adverse consequences on signalized highways until the advent of vehicle-based
communication (e.g., connected vehicles or CV) and control (e.g.,
automated vehicles or AV) technologies. CV basically enables real-time
information sharing and communications among individual vehicles and
infrastructure control units\footnote{http://www.its.dot.gov/connected\_vehicle/connected\_vehicle.htm.}.
AV aims to replacing a human driver with a robot that constantly receives
environmental information via various sensor technologies (as compared
to human eyes and ears) and consequently determines vehicle control
decisions (e.g., acceleration and braking) with proper computer algorithms
(as compared to human brains) and vehicle control mechanics (as compared
to human limbs)\footnote{http://en.wikipedia.org/wiki/Autonomous\_car}.
The combination of these two technologies, which is referred as connected
and automated vehicles (CAV), essentially enables disaggregated control
(or coordination) of individual vehicles with real-time vehicle-to-vehicle
and vehicle-to-infrastructure communications. Before these technologies,
highway vehicle dynamics was essentially determined by microscopic
human driving behavior. However, there is not even a universally accepted
formulation of human driving behavior \citep{treiber2010three} due
to the unpredictable nature of humans \citep{Kerner1996} and limited
empirical data to comprehensively describe such behavior \citep{daganzo1999possible}.
Therefore, it was very challenging, if not completely impossible,
to perfectly smooth vehicle trajectories with traditional infrastructure-based
controls (e.g., traffic signals) that are designed to accommodate
human behavior. Whereas CAV enables replacing (at least partially)
human drivers with programmable robots whose driving algorithms can
be flexibly customized and accurately executed. This opens up a range
of opportunities to control individual vehicle trajectories in ways
that cooperate with aggregated infrastructure-based controls so as
to optimize both individual drivers' experience and overall traffic
performance. These opportunities inspired several pioneering studies
to explore how to utilize CAV to improve mobility and safety at intersections
\citep{dresner2008,Lee2012} and reduce environmental impacts along
highway segments \citep{ahn2013ecodrive,yang2014control}. However,
these limited studies mostly focus on controlling one or very few
vehicles at a particular highway facility (e.g., either an intersection
or a segment) to achieve a certain specific objective (e.g., stability,
safety or fuel consumption) rather than smoothing a stream of vehicles
to improve its overall traffic performance. Mos of the developed control
algorithms require sophisticated numerical computations and their
real-time applications might be hindered by excessive computational
complexities. 

This study aims to propose a new CAV-based traffic control framework
that controls detailed trajectory shapes of a stream of vehicles on
a stretch of highway combining a one-lane section and a signalized
intersection. As illustrated in Figure \ref{fig:motivating_example}(b),
the very basic idea of this study is smoothing vehicle trajectories
and clustering them to platoons that can just properly occupy the
green light windows and pass the intersection at a high speed. Note
that a higher passing speed indicates a larger intersection capacity,
and thus we see that the CAVs in Figure \ref{fig:motivating_example}(b)
not only have much smoother trajectories but also spend much less
travel times compared with the benchmark manual vehicles in Figure
\ref{fig:motivating_example}(a). Further, smoother trajectories imply
safer traffic, less fuel consumption, milder emissions, and better
driver experience. While the research idea is intuitive, the technical
development is quite sophisticated, because this study needs to manipulate
continuous trajectories that not only individually have infinite control
points but also have complex interactions between one another due
to the shared rights of way. In order to overcome these modeling challenges,
we first partition each trajectory into a few parabolic segments that
each is analytically solvable. This essentially reduces an infinite-dimensional
trajectory into a few set of parabolic function parameters. Further,
we only use four acceleration and deceleration variables that are
nonetheless able to control the overall smoothness of the whole stream
of vehicle trajectories while assuring the exceptional parsimony and
simplicity of the proposed algorithm. With these treatments, we propose
an efficient shooting heuristic algorithm that can generate a stream
of smooth and properly platooned trajectories that can pass the intersection
efficiently and safely yielding minimum environmental impacts. Also,
note that this algorithm can be easily adapted to freeway speed harmonization
as well because the freeway trajectory control problem is essentially
a special case of the investigated problem with an infinite green
time. We investigate a lead vehicle problem to study this freeway
adaption. After generalizing the concept of time geography \citep{miller2005measurement}
by allowing finite acceleration and deceleration, we discover a number
of elegant properties of these algorithms in the feasibility to the
original trajectory control problem and the connection with classic
traffic flow models. 

This Part I paper focuses construction of parsimonious feasible algorithms
and analysis of related theoretical properties. We want to note that
the ultimate goal of this whole study is to establish a methodology
framework that determines the best trajectory vectors under several
traffic performance measures, such as travel time, fuel consumption,
emission and safety surrogate measures. While this paper also qualitatively
discusses optimality issues with visual patterns in trajectory plots,
we leave the detailed computational issues and the overall optimization
framework to the Part II paper \citep{Ma2015}.

\subsection{Literature Review}

Freeway traffic smoothing has drawn numerous attentions from both
academia and industry in the past several decades. Numerous studies
have been conducted in attempts to characterize stop-and-go traffic
on freeway \citep{Herman1958,Chandler1958,Kuhne1987,Bando1995,Kerner1997,Bando1998,Kerner1998,Mauch2002,ahn2005,Li09c,Laval11}
\citet{Herman1958,Bando1995,Li09c}. However, probably due to the
lack of high resolution trajectory data \citep{daganzo1999possible},
no consensus has been formed on fundamental mechanisms of stop-and-go
traffic formation and propagation, particularly at the microscopic
level \citep{treiber2010three}. To harmonize freeway traffic speed,
scholars and practitioners have proposed and tested a number of infrastructure-based
control methods mostly targeting at aggregated traffic (rather than
individual vehicles), including variable speed limits \citep{lu2014review},
ramp metering\citep{hegyi2005model}, and merging traffic control
\citep{spiliopoulou2009toll}. While theoretical results show that
these speed harmonization methods can drastically improve traffic
performance in all major performance measures, e.g., safety, mobility
and environmental impacts \citep{islam2013assessing,yang2014control},
field studies show the performances of these methods exhibit quite
some discrepancies \citep{bham2010evaluation}. Probably due to limited
understandings of microscopic behavior of highway traffic, these field
practices of speed harmonization are mostly based on empirical experience
and trial-and-error approaches without taking full advantage of theoretical
models. Also, drivers may not fully comply with the speed harmonization
control and their individual responses may be highly stochastic, which
further comprises the actual performance of these control strategies.
Therefore, these aggregated infrastructure-based traffic smoothing
measures may not perform as ideally as theoretical model predictions. 

With the advance of vehicle-based communication (i.e., CV) and control
(i.e., AV) technologies, researcher started exploring ways of freeway
traffic smoothing by controlling individual vehicles. \citet{schwarzkopf1977control}
analytically solved the optimal trajectory of a single vehicle on
certain grade profiles with simple assumptions of vehicle characteristics
based on Pontryagin's minimum principle. \citet{hooker1988optimal}
instead proposed a simulation approach that capture more realistic
vehicle characteristics. \citet{Van_Arem2006} finds that traffic
flow stability and efficiency at a merge point can be improved by
cooperative adapted cruise control (a longitudinal control strategy
of CAV) that smooths car-following movements. \citet{liu2012reducing}
solved an optimal trajectory for one single vehicle and used this
trajectory as a template to control multiple vehicles with variable
speed limits. \citet{ahn2013ecodrive} proposed an rolling-horizon
individual CAV control strategy that minimizes fuel consumption and
emission considering roadway geometries (e.g., grades). \citet{yang2014control}
proposed a vehicle speed control strategy to mitigate traffic oscillation
and reduce vehicle fuel consumption and emission based on connected
vehicle technologies. They found with only a 5 percent compliance
rate, this control strategy can reduce traffic fuel consumption by
up to 15 percent. Wang et. al. \citeyearpar{wang2014rolling-non-coop,wang2014rolling-coop}
proposed optimal control models based on dynamic programming and Pontryagin's
minimum principle that determine accelerations of a platoon of AVs
or CAVs to minimize a variety of objective cost functions. \citet{li2014stop}
revised a classic manual car-following model into one for CAV following
by incorporating CAV features such as faster responding time and shared
information. They found the CAV following rules can significantly
reduce magnitudes of traffic oscillation, emissions and travel time.
Despite relatively homogenous settings and complex algorithms, these
adventurous developments have demonstrated a great potential of these
advanced technologies in improving freeway mobility, safety and environment.

Despite these fruitful developments on the freeway side, traffic smoothing
on interrupted highways (i.e., with at-grade intersections) is a relative
recent concept. This concept is probably motivated by recent CAV technologies
that allow vehicles paths to be coordinated with signal controls.
The existing traffic smooth studies for interrupted highways can be
in general categorized into two types. The first type assumes that
CAVs can communicate with each other to pass an intersection in a
self-organized manner (e.g., like a school of fish) even without conventional
traffic signals. For example, \citet{dresner2008} proposes a heuristic
control algorithm that process vehicles as a queuing system. While
this development probably performs excellently when traffic is light
or moderate, its performance under dense traffic is yet to be investigated.
Further, \citet{lee2012development} proposes a nonlinear optimization
model to test the limits of a non-stop intersection control scheme.
They show that ideally, the optimal non-stop intersection control
can significantly outperform classic signalized control in both mobility
and environmental impacts at different congestion levels. \citet{zohdy2014intersection}
integrates an embedded car-following rule and an intersection communication
protocol into an nonlinear optimization model that manages a non-stop
intersection. This model considers different weather conditions, heterogeneous
vehicle characteristics and varying market penetrations. Overall,
these developments on non-stop unsignalized control usually only focus
on the operations of vehicles in the vicinity of an intersection and
require complicated control algorithms and simulation. How to implement
this complex mathematical programming model in real-time application
might need further investigations. 

The second type of studies for interrupted traffic smoothing consider
how to design vehicles trajectories in compliance with existing traffic
signal controls at intersections. The basic ideal is that a vehicle
shall slow down from a distance when it is approaching to a red light
so that this vehicle might be able to pass the next green light following
a relatively smooth trajectory without an abrupt stop. Trayford et
al. \citeyearpar{trayford1984fuel2,trayford1984fuel1} tested using
speed advice to vehicles approaching to an intersection so as to reduce
fuel consumption with computer simulation. Later studies extend the
speed advice approach to car-following dynamics \citep{sanchez2006predicting},
in-vehicle traffic light assistance \citep{iglesias2008i2v,wu2010energy}
multi-intersection corridors \citep{mandava2009arterial,guan2013predictive,de2013eco},
scaled-up simulation \citep{tielert2010impact}, and electric vehicles
\citep{wu2015Energy}. These approaches mainly focused on the bulky
part of a vehicle's trajectory with constant cruise speeds without
much tuning its microscopic acceleration. However, acceleration detail
actually largely affects a vehicle's fuel consumption and emissions
\citep{rakha2011eco}. To address this issue, \citet{kamalanathsharma2013multi}
proposes an optimization model that considers a more realistic yet
more sophisticated fuel-consumption objective function in smoothing
a single vehicle trajectory at an signalized intersection. While such
a model captures the advantage from microscopically tuning vehicle
acceleration, it requires a complicated numerical solution algorithm
that takes quite some computation resources even for a single trajectory. 

In summary, there have been increasing interests in vehicle smoothing
using advanced vehicle-based technologies in recent years. However,
most relevant studies only focus on controlling one or a few individual
vehicles. Most studies either ignore acceleration detail and allow
speed jumps to assure the model computational tractability or capture
acceleration in very sophisticated algorithms that are difficult to
be simply implemented in real time. Further, few studies investigated
theoretical properties of the proposed controls and their relationships
with classic traffic flow theories. Without such theoretical insights,
we would miss the great opportunity of transferring the vast elegant
developments on existing manual traffic in the past few decades to
future CAV traffic. 

This proposed trajectory optimization framework aims to fill these
research gaps. We investigate a general trajectory control problem
that optimizes individual trajectories of a long stream of interactive
CAVs on a signalized highway section. This problem is general such
that the lead vehicle problem on a freeway can be represented as its
special case (e.g., by setting the red light time to zero). This problem
is a very challenging infinite-dimension nonlinear optimization problem,
and thus it is very hard to solve its exact optimal solution. We instead
propose a heuristic shooting algorithm to solve an near-optimum solution
to this problem. This algorithm can be easily extended to the freeway
lead vehicle problem. While the proposed algorithms can flexibly control
trajectory shapes by tuning acceleration across a broad range, it
is extremely parsimonious: it compresses a trajectory into a very
few number of analytical segments, and includes only a few acceleration
levels as control variables. With such parsimony and simplicity, these
algorithms expect to be quite suitable for real-time applications
and further adaptations. The simple structure of these algorithms
also allows us to analyze its theoretical properties. By extending
the traditional time-geography theory to a second-order version that
considers finite acceleration, we are able to analytically investigate
the feasibility of the proposed algorithms and the implication to
the feasibility of the original problem. This novel extension also
allows us to relate the trajectories solved by the lead vehicle problem
algorithm to those generated by classic traffic flow models. This
helps us reveal the fundamental commonalities between these two highway
traffic management paradigms based on completely different technologies. 

This paper is organized as follows. Section \ref{sec:Problem-Statement}
states the studied CAV trajectory optimization problem on a signalized
highway section and its variant for the lead vehicle problem on a
freeway segment. Section \ref{sec:Shooting-Heuristic-Algorithms}
describes the proposed shooting heuristic algorithms for the original
problem and its variant, respectively. Section \ref{sec:Theoretical-Properties-of}
analyzes the theoretical properties of the proposed algorithms based
on an extended time-geography theory. Section \ref{sec:Numerical-Examples}
demonstrates the proposed algorithms and their properties with a few
illustrative examples. Section \ref{sec:Conclusion} concludes this
paper and briefly discusses future research directions.

\section{Problem Statement\label{sec:Problem-Statement}}

\subsection{Primary Problem}

This section describes the primary problem of trajectory construction
on a signalized one-lane highway segment, as illustrated in Figure
\ref{fig:problem_statement}. The problem setting is described below.

\begin{figure}
\begin{centering}
\includegraphics[width=0.6\textwidth]{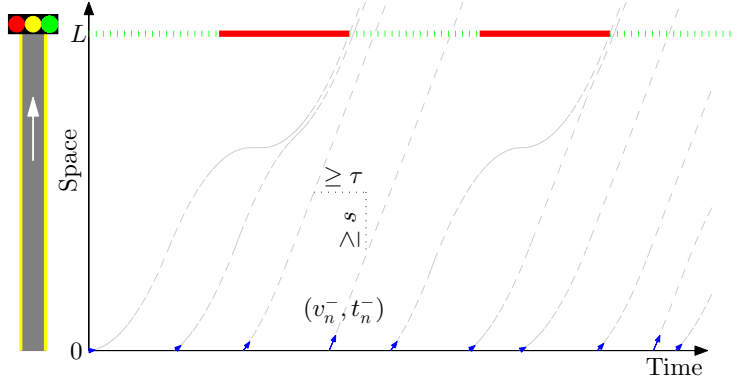}
\par\end{centering}

\protect\caption{Illustration of the studied problem. \label{fig:problem_statement}}
\end{figure}

\begin{description}
\item [{Roadway}] \textbf{Geometry}: We consider a single-lane highway
section of length $L$. Location on this segment starts from 0 at
the upstream and ends at $L$ at the downstream. We use location set
$[0,L]$ to denote this highway section. Traffic goes from location
$0$ to $L$ on this section. A fixed-time traffic light is installed
at location $L.$ The effective green phase starts at time $0$ with
an duration of $G$, followed by an effective red phase of duration
$R$, and this pattern continues all the way. This indicates the signal
cycle time is always $C:=R+G$. We denote the set of green time intervals
as $\mathcal{G}:=\left\{ \left[g_{m}:=mC,r_{m}:=mC+G\right)\right\} _{m\in\mathbb{Z}^{+}}$
where $\mathbb{Z}^{+}$ is the non-negative integer set. We define
function 
\[
G(t)=\min\{t'\in\mathcal{G},t'>t\},\forall t\in[-\infty,\infty],
\]
 which identifies the next closest green time to $t$. Note that $G(t)=t$
if $t\in\mathcal{G}$ or $G(t)>t$ if $t\notin\mathcal{G}$. 
\item [{Vehicle}] \textbf{Characteristics}: We consider a stream of $N$
identical automated vehicles indexed as $n\in\mathcal{N}:=\left\{ 1,2,\cdots,N\right\} $.
Each vehicle's acceleration at any time is no less than deceleration
limit $\underline{a}<0$ and no greater than acceleration limit $\overline{a}>0$.
The speed limit on this segment is $\bar{v}$, and we don't allow
a vehicle to back up, thus a vehicle's speed range is $[0,\overline{v}]$. 
\end{description}
We want to design a set of trajectories for these vehicles to follow
on this highway section. A trajectory is formally defined below.
\begin{defn}
A \emph{trajectory} is defined as a second-order semi-differentiable
function $p(t),\forall t\in(-\infty,\infty)$ such that its first
order differential (or \emph{velocity}) $\dot{p}(t)$ is absolutely
continuous and its second-order right-differential $\ddot{p}(t)$
(or \emph{acceleration}) is Riemann integrable over any $t\in(-\infty,\infty)$.
We denote the set of all trajectories by $\bar{\mathcal{T}}$. We
call the subsection of $p$ between times $t^{-}$ and $t^{+}$ ($-\infty\le t^{-}<t^{+}\le\infty$)
a\emph{ trajectory section}, denoted by $p\left(t^{-}:t^{+}\right)$. 
\end{defn}
We let function $p_{n}$ denote the trajectory of vehicle $n,\forall n\in\mathcal{N}$.
At any time $t$, $p_{n}(t)$ essentially denotes the location of
vehicle $n$'s front bumper. Collectively we denote all vehicle trajectories
by \emph{trajectory vector} $\mathbf{p}:=\left[p_{n}(t)\right]_{n\in\mathcal{N}}$.
The trajectories in vector $\mathbf{p}$ shall satisfy  the following
constraints.
\begin{description}
\item [{Kinematic}] \textbf{Constraints}: Trajectory $p$ is\emph{ kinetically
feasible} if $\dot{p}(t)\in[0,\bar{v}]$ and $\ddot{p}(t)\in\left[\underline{a},\bar{a}\right],\forall t\in(-\infty,\infty)$.
We denote the set of all kinetically feasible trajectories as 
\begin{equation}
\mathcal{T}:=\left\{ p\in\mathcal{\bar{T}}\left|0\le\dot{p}(t)\le\bar{v},\underline{a}\le\ddot{p}(t)\le\bar{a},\forall t\in\left(-\infty,\infty\right)\right.\right\} .\label{eq:set-kinematic-constraints}
\end{equation}

\item [{Entry}] \textbf{Boundary Condition}: Let $t_{n}^{-}$ and $v_{n}^{-}$
denote the time and the speed when vehicle $n$ arrives at the entry
of this segment (or location 0), $\forall n\in\mathcal{N}$. We require
$t_{1}^{-}<t_{2}^{-}<\cdots<t_{N}^{-}$ and separation between $t_{n-1}$
and $t_{n}$ is sufficient for the safety requirement, $\forall n\in\mathcal{N}\backslash\{1\}$.
Define the subset of trajectories in $\mathcal{T}$ that are consistent
with vehicle $n$'s entry boundary condition as 
\begin{equation}
\mathcal{T}_{n}^{-}:=\left\{ p\in\mathcal{T}\left|p(t_{n}^{-})=0,\dot{p}(t_{n}^{-})=v_{n}^{-}\right.\right\} ,\forall n\in\mathcal{N}.\label{eq: set-entry-boundary}
\end{equation}

\item [{Exit}] \textbf{Boundary Condition}: Due to the traffic signal at
location $L$, vehicles can only exit this section during a green
light. Denote the set of trajectories in $\mathcal{T}$ that pass
location $L$ during a green signal phase by 
\begin{equation}
\mathcal{T}^{+}:=\left\{ p\in\mathcal{T}\left|p^{-1}(L)\in\mathcal{G}\right.\right\} ,\forall n\in\mathcal{N},\label{eq: set-exit-boundary}
\end{equation}
where the generalized inverse function is defined as $p^{-1}(l):=\inf\left\{ t\left|p(t)\ge l\right.\right\} ,\forall l\in\left[0,\infty\right),p\in\mathcal{T}$.
Note that function $p^{-1}(\cdot)$ shall satisfy the following properties:
\end{description}
\vspace{0.01\baselineskip} 

\textbf{P1}: Function $p^{-1}(l)$ is increasing with $l\in\left(-\infty,\infty\right)$;

\textbf{P2}: Due to speed limit $\bar{v}$, $p^{-1}(l+\delta)\ge p^{-1}(l)+\delta/\bar{v},\forall l\in\left(-\infty,\infty\right),\delta\in[0,\infty)$.

Further, let $\mathcal{T}_{n}$ denote the set of trajectories that
satisfy both vehicle $n$'s entry and exit boundary conditions, i.e.,
$\mathcal{T}_{n}=\mathcal{T}_{n}^{-}\bigcap\mathcal{T}^{+}$. 
\begin{description}
\item [{Car-following}] \textbf{Safety}: We require that the separation
between vehicle $n$'s location at any time and its preceding vehicle
$(n-1$)'s location a communication delay $\tau$ ago is no less than
a jam spacing $s$ (which usually includes the vehicle length and
a safety buffer), $\forall n\in\mathcal{N}\backslash\{1\}$. For a
genetic trajectory $p\in\mathcal{T}$ , a \emph{safely following trajectory
}of $p$ is a trajectory $p'\in\mathcal{T}$ such that $p'(t-\tau)-p(t)\ge s,\forall t\in(-\infty,\infty)$.
We denote the set of all safely following trajectories of $p$ as
$\mathcal{F}(p)$, i.e., 
\begin{multline}
\mathcal{F}(p):=\left\{ p'\left|p'(t-\tau)-p(t)\ge s,\forall t\in(-\infty,\infty)\right.\right\} ,\forall p\in\mathcal{T}.\label{eq: set-safety-constraints}
\end{multline}
With this, we define $\mathcal{T}_{n}(p_{n-1}):=\mathcal{F}(p_{n-1})\bigcap\mathcal{T}_{n}$
that denotes the set of feasible trajectories for vehicle $n$ that
are safely following given vehicle $(n-1)$'s trajectory $p_{n-1}$. 
\end{description}
In summary, a feasible lead vehicle's trajectory $p_{1}$ has to fall
in $\mathcal{T}_{1}$, and any feasible following vehicle trajectory
$p_{n}$ has to belong to $\mathcal{T}_{n}(p_{n-1}),\forall n\in\mathcal{N}\backslash\{1\}$
. We say a trajectory vector $\mathbf{p}$ is \emph{feasible} if it
satisfies all above-defined constraints. Let $\mathcal{P}$ denote
the set of all\emph{ }feasible trajectory vectors, i.e., 
\begin{equation}
\mathcal{P}:=\left\{ \mathbf{p}:=\left[p_{n}\right]{}_{n\in\mathcal{N}}\left|p_{1}\in\mathcal{T}_{1},p_{n}\in\mathcal{T}_{n}\left(p_{n-1}\right),\forall n\in\mathcal{N}\backslash{1}\right.\right\} .\label{eq:P_feasible_platoon}
\end{equation}
The primary problem (PP) investigated in this paper is finding and
analyzing feasible solutions to $\mathcal{P}$. 
\begin{rem}
Although a realistic vehicle trajectory only has a limited length,
we set a trajectory's time horizon to $\left(-\infty,\infty\right)$
to make the mathematical presentation convenient without loss of generality.
In our study, we are only interested in trajectory sections between
locations $0$ and $L$, i.e., $p_{n}\left(t_{n}^{-},p_{n}^{-1}(L)\right).$
Therefore, we can just view $p_{n}(-\infty,t_{n}^{-})$ as the given
trajectory history that leads to the entry boundary condition, and
$p_{n}\left(p_{n}^{-1}(L),\infty\right)$ as some feasible yet trivial
projection above location $L$ (e.g., accelerating to $\bar{v}$ with
rate $\bar{a}$ and then cruising at $\bar{v}$). Further, with this
extension, safety constraint \eqref{eq: set-safety-constraints} also
ensures that vehicles did not collide before arriving location $0$
and will not collide after exiting location $L$.
\end{rem}

\begin{rem}
This study can be trivially extended to static yet time-variant signal
timing; i.e., the signal timing plan is pre-determined, yet different
cycles could have different green and red durations, e.g., alternating
like $G_{1},R_{1},G_{2},R_{2},\cdots$. In this case, define signal
timing switch points $g_{m}=\sum_{i=1}^{m}\left(G_{i}+R_{i}\right),$
$r_{m}=\sum_{i=1}^{m}\left(G_{i}+R_{i}\right)+G_{i},$ $\forall m\in\mathbb{Z}^{+}$
and $\mathcal{G}$ becomes $\left\{ \left[g_{m},r_{m}\right)\right\} _{\forall m=1,2}$
where $r_{0}:=0$, and all the following results shall remain valid. 
\end{rem}

\subsection{Problem Variation: Lead-Vehicle Problem}

In the classic traffic flow theory, the lead-vehicle problem (LVP)
is a well-known fundamental problem that predicts traffic flow dynamics
on one-lane freeway given the lead vehicle's trajectory and the following
vehicles' initial states \citep{Daganzo2006}. We notice that PP \eqref{eq:P_feasible_platoon}
can be easily adapted to LVP by relaxing exit boundary condition \eqref{eq: set-exit-boundary}
yet fixing trajectory $p_{1}$. The LVP is officially formulated as
follows. Given lead vehicle's trajectory $p_{1}\in\mathcal{T}$, the
set of feasible trajectories for LVP is 

\begin{equation}
\mathcal{P}^{\mbox{LVP}}\left(p_{1}\right):=\left\{ \mathbf{p}:=\left[p_{n}\right]{}_{n\in\mathcal{N}}\left|p_{n}\in\mathcal{F}(p_{n-1})\bigcap\mathcal{T}_{n}^{-},\forall n\in\mathcal{N}\backslash{1}\right.\right\} .\label{eq:P_feasible_platoon_LVP}
\end{equation}
The LVP investigated in this paper is finding and analyzing feasible
solutions to $\mathcal{P}^{\mbox{LVP}}$.

\section{Shooting Heuristic Algorithms \label{sec:Shooting-Heuristic-Algorithms}}

This section proposes customized heuristic algorithms to solve feasible
trajectory vectors to PP and LVP. Although a trajectory is defined
over the entire time horizon $(-\infty,\infty)$, these algorithms
only focus on the trajectory sections from the entry time $t_{n}^{-}$
for each vehicle $n\in\mathcal{N},$ because the trajectory sections
before $t_{n}^{-}$ should be trivial given history and do not affect
the algorithm results. Therefore, in the following presentation, we
view a trajectory and the corresponding trajectory section over time
$[t_{n}^{-},\infty)$ the same.

\subsection{Shooting Heuristic for PP}

This section presents a shooting heuristic (SH) algorithm that is
able to construct a smooth and feasible trajectory vector to solve
$\mathcal{P}$ in PP very efficiently. Traditional methods for trajectory
optimization include analytical approaches that can only solve simple
problems with special structures and numerical approaches that can
accommodate more complex settings yet may demand enormous computation
resources \citep{von1992direct}. Since a vehicle trajectory is essentially
an infinite-dimensional object along which the state (e.g. location,
speed, acceleration) at every point can be varied, it is challenging
to even construct one single trajectory, particularly under nonlinear
constraints. Note that our problem deals with a large number of trajectories
for vehicles in a traffic stream that constantly interact with each
other and are subject to complex nonlinear constraints \eqref{eq:set-kinematic-constraints}-\eqref{eq:P_feasible_platoon}.
Therefore we deem that it is very complex and time-consuming to tackle
this problem with a traditional approach. Therefore, we opt to devise
a new approach that circumvents the need for formulating high-dimensional
objects or complex system constraints. This leads to the development
of a shooting heuristic (SH) algorithm that can efficiently construct
a smooth feasible trajectory vector with only a few control parameters.

Figure \ref{fig:shooting-process} illustrates the components in the
proposed SH algorithm. Basically, for each vehicle $n\in\mathcal{N}$,
SH first constructs a trajectory, denoted by $p_{n}^{\mbox{f}}$ ,
with a forward shooting process that conforms with kinematic constraint
\eqref{eq:set-kinematic-constraints}, entry boundary constraint \eqref{eq: set-entry-boundary}
and safety constraint \eqref{eq: set-safety-constraints} (if $n>1$).
As illustrated in Figure \ref{fig:shooting-process}(a), if trajectory
$p_{n}^{\mbox{f}}$ (dashed blue curve) turns out far enough from
preceding trajectory $p_{n-1}$ (red solid curve) such that safety
constraint \eqref{eq: set-safety-constraints} is even not activated
(or if $n=1$ and $p_{n}^{\mbox{f}}$ is already the lead trajectory),
$p_{n}^{\mbox{f}}$ basically accelerates from its entry boundary
condition $\left(t_{n}^{-},v_{n}^{-}\right)$ at location 0 with a
forward acceleration rate of $\bar{a}^{\mbox{f}}\in(0,\bar{a}]$ until
reaching speed limit $\bar{v}$ and then cruises at constant speed
$\bar{v}$. Otherwise, as illustrated in Figure \ref{fig:shooting-process}(b),
if trajectory $p_{n}^{\mbox{f}}$ is blocked by $p_{n-1}$ due to
safety constraint \eqref{eq: set-safety-constraints}, we just let
$p_{n}^{\mbox{f}}$ smoothly merge into a safety bound (the red dotted
curve) translated from $p_{n-1}$ that just keeps spatial separation
$s$ and temporal separation $\tau$ from $p_{n}^{\mbox{f}}$. The
transitional segment connecting $p_{n}^{\mbox{f}}$ with the safety
bound decelerates at a forward deceleration rate of $\underline{a}^{\mbox{f}}\in[\underline{a},0)$.
If trajectory $p_{n}^{\mbox{f}}$ from the forward shooting process
is found to violate exit boundary constraint \eqref{eq: set-exit-boundary}
(or run into the red light), as illustrated in Figure \ref{fig:shooting-process}(c),
a backward shooting process is activated to revise $p_{n}^{\mbox{f}}$
to comply with constraint \eqref{eq: set-exit-boundary}. The backward
shooting process first shifts the section of $p_{n}^{\mbox{f}}$ above
location $L$ rightwards to the start of the next green phase to be
a backward shooting trajectory $p_{n}^{\mbox{b}}$ . Then $p_{n}^{\mbox{b}}$
shoots backwards from this start point at a backward acceleration
rate $\bar{a}^{\mbox{b}}\in(0,\bar{a}]$ until getting close enough
to merge into $p_{n}^{\mbox{f}}$ , which may require $p_{n}^{\mbox{b}}$
stops for some time if the separation between the backward shooting
start point and $p_{n}^{\mbox{f}}$ is long relative to acceleration
rate $\bar{a}^{\mbox{b}}$ . Then, $p_{n}^{\mbox{b}}$ shoots backwards
a merging segment at a backward deceleration rate of $\underline{a}^{\mbox{b}}\in[\underline{a},0)$
until getting tangent to $p_{n}^{\mbox{f}}$. Finally, merging $p_{n}^{\mbox{f}}$
and $p_{n}^{\mbox{b}}$ yields a feasible trajectory $p_{n}$ for
vehicle $n$. Such forward and backward shooting processes are executed
from vehicle $1$ through vehicle $N$ consecutively, and then SH
concludes with a feasible trajectory vector $\mathbf{p}=\left[p_{n}\right]{}_{n\in\mathcal{N}}\in\mathcal{P}.$
Note that SH only uses four control variables $\left\{ \bar{a}^{\mbox{f}},\underline{a}^{\mbox{f}},\bar{a}^{\mbox{b}},\underline{a}^{\mbox{b}}\right\} $
that are yet able to control the overall smoothness of all trajectories
(so as to achieve certain desired objectives). Further, the constructed
trajectories are composed of only a few quadratic (or linear) segments
that are all analytically solvable. Therefore, the proposed SH algorithm
is parsimonious and simple to implement.

\begin{figure}
\begin{centering}
\includegraphics[width=0.33\textwidth]{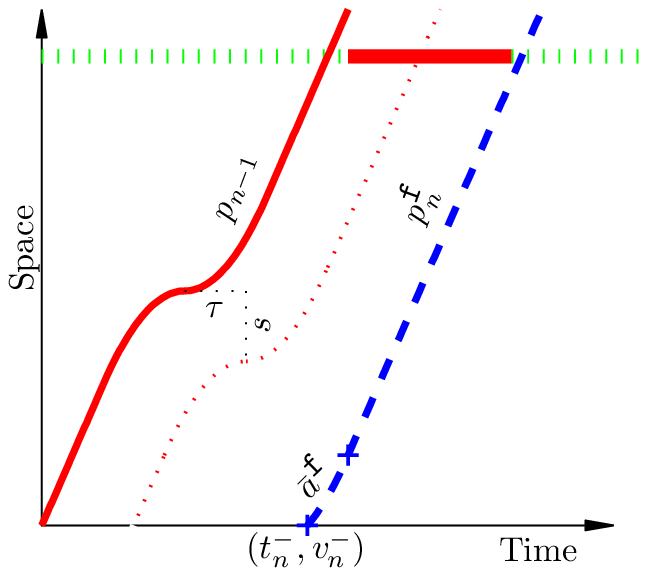}\includegraphics[width=0.33\textwidth]{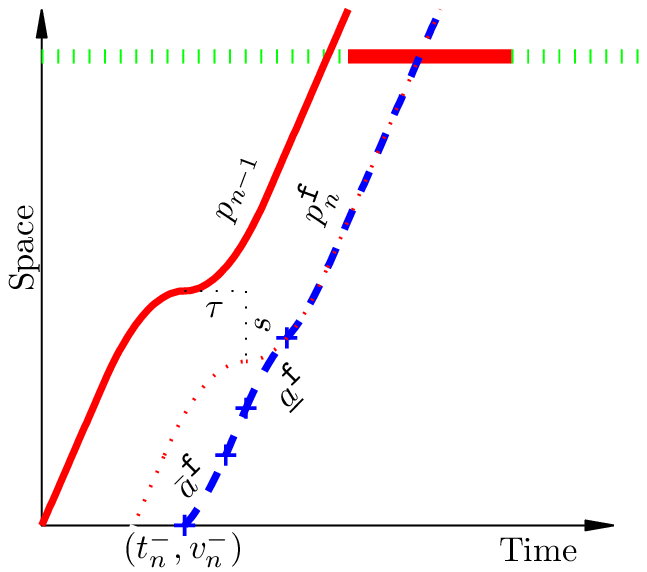}\includegraphics[width=0.33\textwidth]{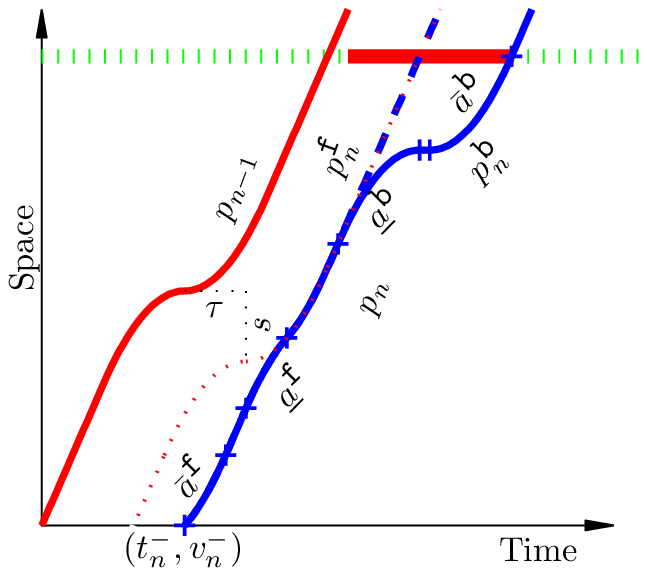}
\par\end{centering}

\begin{centering}
(a)\quad{}\quad{}\quad{}\quad{}\quad{}\quad{}\quad{}\quad{}\quad{}\quad{}\quad{}\quad{}\quad{}\quad{}(b)\quad{}\quad{}\quad{}\quad{}\quad{}\quad{}\quad{}\quad{}\quad{}\quad{}\quad{}\quad{}\quad{}\quad{}(c)
\par\end{centering}

\protect\caption{Components in the proposed shooting heuristic: (a) forward shooting
process without activating safety constraint \eqref{eq: set-safety-constraints};
(b) forward shooting process with activating safety constraint \eqref{eq: set-safety-constraints};
and (c) backward shooting process. \label{fig:shooting-process}. }
\end{figure}

To formally state SH, we first define the following terminologies
in Definitions \ref{def: state_point}-\ref{def: elemental_segment}
with respect to a single trajectory, as illustrated in Figure \ref{fig:seg_def}(a). 
\begin{defn}
We define a \emph{state point} by a three-element tuple $(l,v,t')$
, which represents that at time $t'$, the vehicle is at location
$l$ and operates at speed $v$. A \emph{feasible state point} $(l,v,t')$
should satisfy $v\in[0,\bar{v}]$.\label{def: state_point}
\end{defn}

\begin{defn}
We use a four-element tuple $(l,v,a,t')$ to denote the \emph{quadratic
function} that passes location $l$ at time $t'$ with velocity $v$
and acceleration $a$; i.e., $0.5a(t-t')^{2}+v(t-t')+l$ with respect
to time variable $t\in(-\infty,\infty)$. Note that this quadratic
function definition also includes a linear function (i.e., $a=0$).
For simplicity, we can use a boldface letter to denote a quadratic
function, e.g., $\mathbf{f}:=(l,v,a,t')$ and $\mathbf{f}(t):=0.5a(t-t')^{2}+v(t-t')+l$.
\end{defn}

\begin{defn}
We use a five-element tuple $\mathbf{s}:=(l,v,a,t',t'')$ to denote
a segment of quadratic function $\mathbf{f}=(l,v,a,t')$ between time
$\min\left\{ t',t''\right\} $ and $\max\left\{ t',t''\right\} $.
We call this tuple a \emph{quadratic segment}. Define $\mathbf{s}(t):=\mathbf{f}(t),\forall t\in\left[\min\left\{ t',t''\right\} ,\max\left\{ t',t''\right\} \right]$.
If the speed and the acceleration on every point along this segment
satisfy constraint \eqref{eq:set-kinematic-constraints}, we call
it a \emph{feasible quadratic segment}.  \label{def: elemental_segment}\end{defn}
\begin{rem}
In Definitions \ref{def: state_point}-\ref{def: elemental_segment},
if one or more elements in a tuple are unknown or variable, we use
``$\cdot$'' to hold their places (e.g., $(l,\cdot,t')$, $(l,v,\cdot,t',\cdot)$
). 
\end{rem}
\begin{figure}
\begin{centering}
\includegraphics[width=0.33\textwidth]{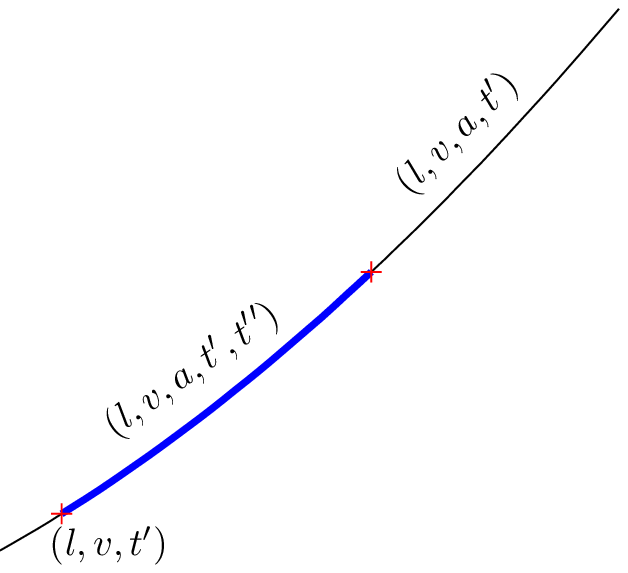}\includegraphics[width=0.33\textwidth]{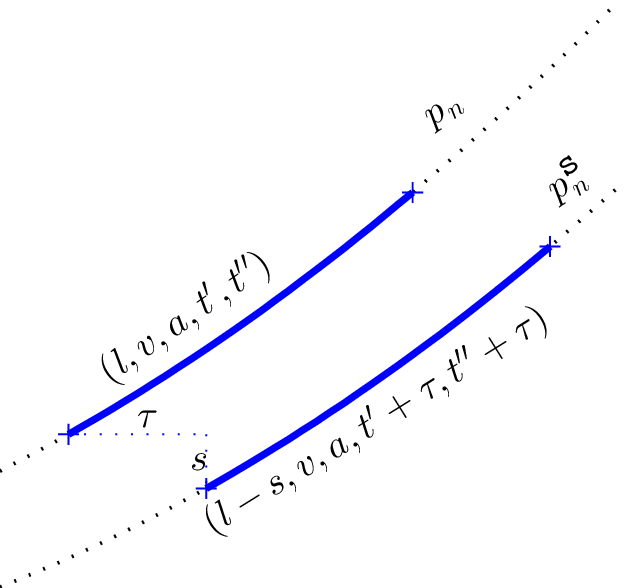}\includegraphics[width=0.33\textwidth]{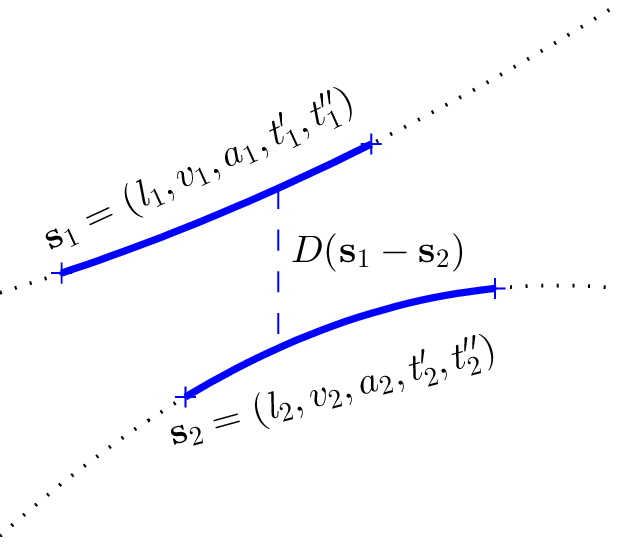}
\par\end{centering}

\begin{centering}
(a)\quad{}\quad{}\quad{}\quad{}\quad{}\quad{}\quad{}\quad{}\quad{}\quad{}\quad{}\quad{}\quad{}\quad{}(b)\quad{}\quad{}\quad{}\quad{}\quad{}\quad{}\quad{}\quad{}\quad{}\quad{}\quad{}\quad{}\quad{}\quad{}(c)
\par\end{centering}

\protect\caption{Illustrations of definitions ($x$-axis for time elapsing rightwards
and $y$-axis for location increasing upwards): (a) state point, quadratic
function and segment; (b) shadow trajectory and shadow segment; and
(c) segment distance. \label{fig:seg_def}. }
\end{figure}

\begin{defn}
For a trajectory $p\in\bar{\mathcal{T}}$ composed by a vector of
consecutive quadratic segments $\mathbf{s}_{k}:=$ $[(l_{k},v_{k},a_{k},t_{k},t_{k+1})]_{k=1,2,\cdots,\bar{k}}$
with $-\infty=t_{1}<t_{2}<\cdots<t_{\bar{k}+1}=\infty$, we denote
$p=\left[\mathbf{s}_{k}\right]_{k=1,2,\cdots,\bar{k}}$. Without loss
of generality, we assume that any trajectory in this study can be
decomposed into a vector of quadratic segments.
\end{defn}
During the forward shooting process, we need to check safety constraints
\eqref{eq: set-safety-constraints} at every move for any vehicle
$n\ge2$. As illustrated in Figure \ref{fig:seg_def}(b), we basically
create a shadow (or safety bound) of the preceding trajectory $p_{n-1}$
by shifting $p_{n-1}$ downwards by $s$ and rightwards by $\tau$.
Then safety constraint \eqref{eq: set-safety-constraints} is simply
equivalent to that $p_{n}$ does not exceed this shadow trajectory
at any time. The following definition specifies the shadow trajectory
and its elemental segments. 
\begin{defn}
For a trajectory $p_{n}(t),\forall t\in(-\infty,\infty)$, we define
its \emph{shadow trajectory} $p_{n}^{\mbox{s}}$ as $p_{n}^{\mbox{s}}(t):=p_{n}(t-\tau)-s,\forall t\in(-\infty,\infty)$.
It is obvious that $t^{-}\left(p_{n}^{\mbox{s}}\right)=t_{n}^{-}+\tau$
and $t^{+}\left(p_{n}^{\mbox{s}}\right)=\infty$. Note that if the
following trajectory $p_{n+1}(t)$ initiated at time $t_{n+1}^{-}$
satisfies $p_{n+1}(t)\le p_{n}^{\mbox{s}},\forall t\in(-\infty,\infty)$,
then $p_{n}$ and $p_{n+1}$ satisfies safety constraint \eqref{eq: set-safety-constraints}.
Further, a \emph{shadow segment} of $\mathbf{s}:=(l,v,a,t',t'')$
is simply $\mathbf{s}^{\mbox{s}}:=(l-s,v,a,t'+\tau,t''+\tau)$. We
also generalize this definition to the $m^{\mbox{th}}$\emph{-order
shadow trajectory} of $p_{n}$ as $p_{n}^{\mbox{s}^{m}}(t):=p_{n}(t-m\tau)-sm,\forall t\in(-\infty,\infty)$
and $m^{\mbox{th}}$\emph{-order shadow segment} $\mathbf{s}^{\mbox{s}^{m}}:=(l-ms,v,a,t'+m\tau,t''+m\tau)$,
which are essentially the results of repeating the shadow operation
by $m$ times.
\end{defn}

The following definitions specify an analytical function that checks
the distance between two segments (e.g., the current segment to be
constructed and a reference shadow segment in forward shooting), as
illustrated in Figure \ref{fig:seg_def}(c).

\begin{defn}
\label{def:trajectory_distance} Given two segments $\mathbf{s}_{1}:=\left(l_{1},v_{1},a_{1},t'_{1},t''_{1}\right)$,
$\mathbf{s}_{2}:=\left(l_{2},v_{2},a_{2},t'_{2},t''_{2}\right)$,
define \emph{segment distance} from $\mathbf{s}_{1}$ to $\mathbf{s}_{2}$
as 

\[
D\left(\mathbf{s}_{1}-\mathbf{s}_{2}\right):=\begin{cases}
\min_{t\in\left[t^{-},t^{+}\right]}\mathbf{s}_{1}(t)-\mathbf{s}_{2}(t), & \mbox{if }t^{-}\le t^{+};\\
\infty, & \mbox{otherwise,}
\end{cases}
\]
where $t^{-}:=\max\left\{ \min\{t'_{1},t''_{1}\},\min\{t'_{2},t''_{2}\}\right\} $
and $t^{+}:=\min\left\{ \max\{t'_{1},t''_{1}\},\max\{t'_{2},t''_{2}\}\right\} $.
If $t^{-}\le t^{+}$, $D\left(\mathbf{s}_{1}-\mathbf{s}_{2}\right)$
can be solved analytically as 

\[
D\left(\mathbf{s}_{1}-\mathbf{s}_{2}\right):=\begin{cases}
\mathbf{s}_{1}\left(t^{*}\left(\mathbf{s}_{1}-\mathbf{s}_{2}\right)\right)-\mathbf{s}_{2}\left(t^{*}\left(\mathbf{s}_{1}-\mathbf{s}_{2}\right)\right),\mbox{ if }a_{1}-a_{2}>0\mbox{ and }t^{*}\left(\mathbf{s}_{1}-\mathbf{s}_{2}\right)\in(t^{-},t^{+})\\
\min\left\{ \mathbf{s}_{1}\left(t^{-}\right)-\mathbf{s}_{2}\left(t^{-}\right),\mathbf{s}_{1}\left(t^{+}\right)-\mathbf{s}_{2}\left(t^{+}\right)\right\} , & \mbox{otherwise}.
\end{cases}
\]
where 
\begin{equation}
t^{*}\left(\mathbf{s}_{1}-\mathbf{s}_{2}\right):=\frac{v_{1}-a_{1}t'_{1}-v_{2}+a_{2}t'_{2}}{a_{2}-a_{1}}.\label{eq:t_min_dist}
\end{equation}
We also extend the distance definition to trajectory sections and
trajectories below. Given two trajectory sections $p(t^{-}:t^{+})$
and $p'(t'^{-}:t'^{+}),$ the\emph{ }distance from $p(t^{-}:t^{+})$
to $p'(t'^{-}:t'^{+})$ is defined as 
\[
D(p(t^{-}:t^{+})-p'(t'^{-}:t'^{+})):=\min_{t\in[\max(t^{-},t'^{-}),\min(t^{+},t'^{+})]}p(t)-p'(t)
\]
if $\max(t^{-},t'^{-})\le\min(t^{+},t'^{+})$ or $D(p(t^{-}:t^{+})-p'(t'^{-}:t'^{+})):=\infty$.
Suppose these two sections can be partitioned into quadratic segments,
i.e., $p(t^{-}:t^{+})=\left[\mathbf{s}_{k}\right]_{k=1,2,\cdots,\bar{k}}$
and $p'(t'^{-}:t'^{+})=\left[\mathbf{s}'_{k'}\right]_{k'=1,2,\cdots,\bar{k}'}$,
then 

\[
D\left(p(t^{-}:t^{+})-p'(t'^{-}:t'^{+})\right)=\min_{k=1,\cdots,\bar{k},k'=1,\cdots,\bar{k}',}D(\mathbf{s}_{k}-\mathbf{s}'_{k'}).
\]
Note that function $D(\cdot)$ has a transitive relationship; i.e.,
$D(A-B)\ge D_{AB}$ and $D(B-C)\ge D_{BC}$ indicates $D(A-C)\ge D_{AB}+D_{BC}$. 
\end{defn}

\begin{defn}
\label{def: forward_shooting } Next, we define an analytical operation
that determines how we take a move in the forward shooting process,
as illustrated in Figure \ref{fig:shooting_operation}(a). Given a
quadratic segment $\mathbf{s}':=\left(l',v',a',t'^{-},t'^{+}\right)$
with $t'^{-}<t'^{+}$, a feasible state point $\left(l,v,t^{-}\right)$
with $t^{-}<t'^{+}$, acceleration rate $a^{+}\ge0$ and deceleration
rate $a^{-}<0$ (and $a^{-}\le a'$), we want to construct a \emph{forward
shooting segment} $\mathbf{s}:=\left(l,v,a^{+},t^{-},t^{\mbox{m}}\right)$
followed by a \emph{forward merging segment} $\mathbf{s}^{\mbox{m}}:=\left(l^{\mbox{m}},v^{\mbox{m}},a^{-},t^{\mbox{m}},t^{+}\right)$
where $v^{\mbox{m}}:=v+a^{+}(t^{\mbox{m}}-t^{-})$, $l^{\mbox{m}}:=p+v(t^{\mbox{m}}-t^{-})+0.5a^{+}(t^{\mbox{m}}-t^{-})^{2}$
with $t^{+}\ge t^{\mbox{m}}\ge t^{-}$ in the following way. We basically
select $t^{-}$ and $t^{m}$ values to make $\mbox{\ensuremath{\mathbf{s}}}$
and $\mathbf{s}^{\mbox{m}}$ satisfy the following conditions. First,
we want to keep $\mathbf{s}'$ above $\mathbf{s}$ and the segment
extended from $\mathbf{s}^{\mbox{m}}$ to time $\infty$, i.e., 
\begin{equation}
D\left[\mathbf{s}'-\mathbf{s}\right]\ge0,\mbox{ and }D\left[\mathbf{s}'-\left(p^{\mbox{m}},v^{\mbox{m}},a^{-},t^{\mbox{m}},\infty\right)\right]\ge0.\label{eq:non_cross_condition}
\end{equation}
Further, the exact values of $t^{\mbox{m}}$ and $t^{+}$ shall be
determined in the following three cases: (I) if no $t^{\mbox{m}}\in[t^{-},\infty)$
can be found to satisfy constraint \eqref{eq:non_cross_condition},
this shooting operation is infeasible and return $t^{m}=t^{+}=-\infty$;
(II) otherwise, we try to find $t^{+}\in\left[\max\left\{ t'^{-},t^{-}\right\} ,t'^{+}\right]$
and $t^{\mbox{m}}\in[t^{-},t^{+}]$ such that $\mathbf{s}'$ and $\mathbf{s}^{\mbox{m}}$
get tangent at time $t^{+}$; and (III) if this trial fails, set $t^{\mbox{m}}=t^{+}=\infty$.
Fortunately, since these segments are all simple quadratic segments,
$t^{\mbox{m}}$ and $t^{+}$ can be solved analytically in the following
\emph{forward shooting operation (FSO) }algorithm, where the final
solutions to $t^{\mbox{m}}$ and $t^{+}$ are denoted by $t^{\mbox{mf}}\left(\mathbf{s}',\left(l,v,t^{-},a^{+}\right),a^{-}\right)$
and $t^{\mbox{+f}}\left(\mathbf{s}',\left(l,v,t^{-},a^{+}\right),a^{-}\right)$,
respectively.\end{defn}
\begin{description}
\item [{FSO-1:}] If $D\left(\mathbf{s}'-\bar{\mathbf{s}}\right)<0$ where
$\bar{\mathbf{s}}:=\left(l,v,a^{-},t^{-},\infty\right)$, there is
no feasible solution, and we just set $t^{\mbox{m}}=t^{+}=-\infty$
(Case I). Go to Step FSO-3.
\item [{FSO-2:}] Shift the origin to time $t^{-}$ and denote $\hat{t}'^{-}:=t'^{-}-t^{-}$,
$\hat{t}'^{+}:=t'^{+}-t^{-}$, $\hat{t}^{\mbox{m}}:=t^{\mbox{m}}-t^{-}$,
$\hat{t}^{\mbox{+}}:=t^{\mbox{+}}-t^{-}$ and $\hat{t}^{-}:=\max\{\hat{t}'^{-},0\}$.
Then get a quadratic function $\mathbf{q}$ by subtracting $\left(l^{\mbox{m}},v^{\mbox{m}},a^{-},t^{\mbox{m}}\right)$
from $(l',v',a',t'^{-})$, i.e., $\mathbf{q}:=\left(\hat{l},\hat{v},a'-a^{-},0\right)$
where 
\[
\hat{l}:=0.5a'\left(\hat{t}'^{-}\right)^{2}+0.5\left(a^{+}-a^{-}\right)\left(\hat{t}^{\mbox{m}}\right)^{2}-a'\hat{t}'^{-}+l'-l,
\]
and 
\[
\hat{v}=v'-v-a'\hat{t}'^{-}-\left(a^{+}-a^{-}\right)\hat{t}^{\mbox{m}}.
\]

\item [{FSO-2-1:}] If $a'=a^{-}$, test whether we can make $\mathbf{q}(t)=0,\forall t\in[\hat{t}^{-},\infty)$,
i.e., whether $\hat{l}=0$ with $\hat{t}^{\mbox{m}}=\left(v'-v-a'\hat{t}'^{-}\right)/\left(a^{+}-a^{-}\right)$.
If yes and $\hat{t}^{\mbox{m}}\in\left[\hat{t}^{-},\hat{t}'^{+}\right]$
(Case II), set $t^{\mbox{m}}=t^{-}+\hat{t}^{\mbox{m}}$ and $t^{\mbox{+}}=t'^{+}$.
Otherwise (Case III), set $t^{\mbox{+}}=t^{\mbox{m}}=\infty$. Go
to Step FSO-3.
\item [{FSO-2-2:}] If $a^{'}>a^{-}$,then $\mathbf{q}$ is a convex quadratic
function, and we need to solve $\alpha\left(\hat{t}^{\mbox{m}}\right)^{2}+\beta\hat{t}^{\mbox{m}}+\gamma=0$
where $\alpha:=\left(a^{+}-a^{-}\right)\left(a^{+}-a'\right)$, $\beta:=-2\left(a^{+}-a^{-}\right)\left(v'-v-a'\hat{t}'^{-}\right)$
and $\gamma:=\left(v'-v-a'\hat{t}'^{-}\right)^{2}-(a'-a^{-})\left(a'\left(\hat{t}'^{-}\right)^{2}-2v'\hat{t}'^{-}+2l'-2l\right)$
. In case of $\alpha=\beta=0$ (Case III), set $t^{\mbox{+}}=t^{\mbox{m}}=\infty$,
and go to Step FSO-3. Otherwise, we need to try candidate solutions
to $\hat{t}^{\mbox{m}}$ and $\hat{t}^{\mbox{+}}$, denoted by $\hat{t}^{\mbox{mc}}$
and $\hat{t}^{+\mbox{c}}$, respectively. In case of $\alpha=0$ but
$\beta\neq0$, solve $\hat{t}^{\mbox{mc}}=-\gamma/\beta$ and 
\begin{equation}
\hat{t}^{+\mbox{c}}=\frac{v'-v-a'\hat{t}'^{-}-\left(a^{+}-a^{-}\right)\hat{t}^{\mbox{mc}}}{a^{-}-a'}.\label{eq:t_hat_plus}
\end{equation}
Otherwise, we should have $\alpha\neq0$, and then solve both candidate
solutions
\begin{equation}
\hat{t}^{\mbox{mc}}=\frac{-\beta\pm\sqrt{\beta^{2}-4\alpha\gamma}}{2\alpha}\label{eq:t_hat_m}
\end{equation}
 and the corresponding $\hat{t}^{+\mbox{c}}$ with equation\eqref{eq:t_hat_plus}.
If the candidate solutions are not real numbers, then we just set
$\hat{t}^{\mbox{+c}}=\hat{t}^{\mbox{mc}}=\infty$. Otherwise, try
both sets of solutions and select the set satisfying $\hat{t}^{+\mbox{c}}\ge\hat{t}^{\mbox{mc}}\ge0$.
For either of these two cases, if $\hat{t}^{+\mbox{c}}\in\left[\hat{t}^{-},\hat{t}'^{+}\right]$
(Case II) , we set $t^{\mbox{m}}=t^{-}+\hat{t}^{\mbox{mc}}$ and $\hat{t}^{\mbox{+}}=\hat{t}^{\mbox{+c}}$.
Otherwise, set $t^{\mbox{+}}=t^{\mbox{m}}=\infty$ (Case III). Go
to Step FSO-3.
\item [{FSO-3:}] Finally, we return $t^{\mbox{mf}}\left(\mathbf{s}',\left(l,v,t^{-},a^{+}\right),a^{-}\right)=t^{\mbox{m}}$
and $t^{\mbox{+f}}\left(\mathbf{s}',\left(l,v,t^{-},a^{+}\right),a^{-}\right)=t^{\mbox{+}}$.
\end{description}
\begin{figure}
\begin{centering}
\includegraphics[width=0.33\textwidth]{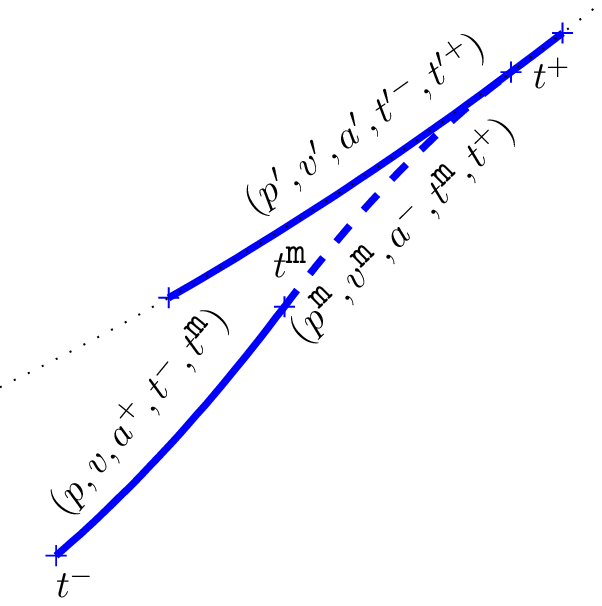}\quad{}\quad{}\quad{}\includegraphics[width=0.33\textwidth]{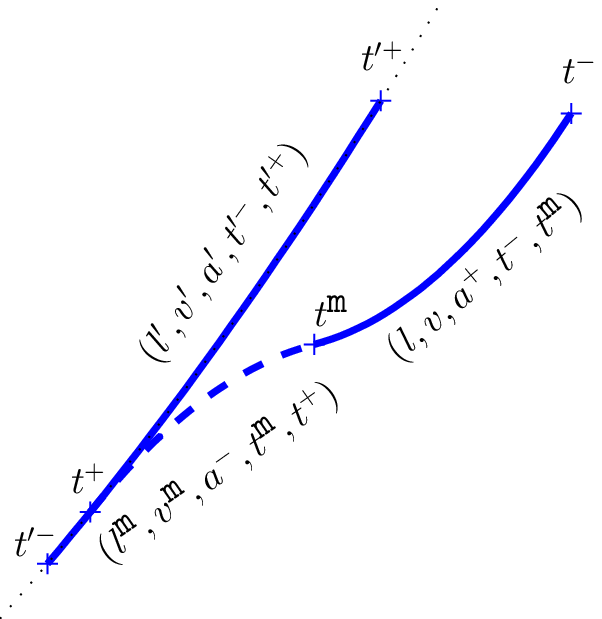}
\par\end{centering}

\begin{centering}
(a)\quad{}\quad{}\quad{}\quad{}\quad{}\quad{}\quad{}\quad{}\quad{}\quad{}\quad{}\quad{}\quad{}\quad{}\quad{}\quad{}(b)
\par\end{centering}

\protect\caption{Illustrations of shooting operations: (a) a forward shooting operation,
and (b) a backward shooting operation. \label{fig:shooting_operation}. }
\end{figure}

\begin{defn}
\label{def: FSP}We extend one forward shooting operation to the following
\emph{forward shooting process (FSP)} that generates a whole trajectory.
Given a shadow trajectory $p^{\mbox{s}}:=\left[\mathbf{s}_{h}^{\mbox{s}}:=\left(l_{h}^{\mbox{s}},v_{h}^{\mbox{s}},a_{h}^{\mbox{s}},t_{h}^{\mbox{s}},t_{h+1}^{\mbox{s}}\right)\right]{}_{h=1,\cdots,\bar{h}}.$
(with $t_{1}^{\mbox{s}}=-\infty$ and $t_{\bar{h}+1}^{\mbox{s}}=\infty$)
and a feasible entry state point $(l,v,t^{-})$, we basically want
to construct a \emph{forward shooting trajectory} $p^{\mbox{f}}((l,v,t^{-}),p^{\mbox{s}})$
that starts from $(l,v,t^{-})$ and maintains acceleration $\bar{a}^{\mbox{f}}$
or speed $\bar{v}$ until being bounded by $p^{\mbox{s}}$. We consider
a template trajectory starting at $(l,v,t^{-})$ and composed by these
two candidate segments $\mathbf{s}^{\mbox{a}}:=(l,v,\bar{a}^{\mbox{f}},t^{-},t^{\mbox{a}}:=t^{-}+(\bar{v}-v)/\bar{a}^{\mbox{f}})$
(which accelerates from $v$ to $\bar{v}$ given $v<\bar{v}$) and
$\mathbf{s}^{\mbox{\ensuremath{\infty}}}:=\left(l^{\mbox{a}}:=l+0.5\left(\bar{v}^{2}-v^{2}\right)/\bar{a}^{\mbox{f}},\bar{v},0,t^{\mbox{a}},\infty\right)$
(which maintains maximum speed $\bar{v}$ all the way), i.e.,

\[
p^{\mbox{t}}:=\begin{cases}
\left[\mathbf{s}^{\mbox{a}},\mathbf{s}^{\infty}\right], & \mbox{if }v<\bar{v};\\
\left[\mathbf{s}^{\infty}\right], & \mbox{if }v=\bar{v}.
\end{cases}
\]
Then $p^{\mbox{f}}((l,v,t^{-}),p^{\mbox{s}})$, if bounded by $p^{\mbox{s}}$,
shall first follow $p^{\mbox{t}}$ and then merge into $p^{\mbox{s}}$
with a merging segment $\left(p^{\mbox{t}}(t^{\mbox{m}}),\dot{p}^{\mbox{t}}(t^{\mbox{m}}),\underline{a}^{\mbox{f}},t^{\mbox{m}},t^{\mbox{+}}\right)$.
This can be solved analytically with the following FSP algorithm.\end{defn}
\begin{description}
\item [{FSP-1:}] Initiate $h=1$, $t^{+}=t^{\mbox{m}}=\infty$, $p^{\mbox{f}}=\emptyset$
and iterate through the segments in $p^{\mbox{s}}$ as follows.
\item [{FSP-2:}] If $v<\bar{v}$, apply the FSO algorithm to solve candidate
time points $t^{\mbox{mc}}:=t^{\mbox{mf}}\left(\mathbf{s}_{h}^{\mbox{s}},(l,v,\bar{a}^{\mbox{f}},t^{-}),\underline{a}^{\mbox{f}}\right)$
and $t^{+\mbox{c}}:=t^{\mbox{+f}}\left(\mathbf{s}_{h}^{\mbox{s}},(l,v,\bar{a}^{\mbox{f}},t^{-}),\underline{a}^{\mbox{f}}\right)$.
If $t^{\mbox{mc}}>t^{\mbox{a}}$, revise $t^{\mbox{mc}}:=(t^{\mbox{mf}}\mathbf{s}_{h}^{\mbox{s}},(l^{\mbox{a}},\bar{v},0,t^{\mbox{a}}),\underline{a}^{\mbox{f}})$
and $t^{\mbox{+c}}:=t^{\mbox{+f}}(\mathbf{s}_{h}^{\mbox{s}},(l^{\mbox{a}},\bar{v},0,t^{\mbox{a}}),\underline{a}^{\mbox{f}})$.
If $v=\bar{v}$, solve $t^{\mbox{mc}}:=t^{\mbox{mf}}(\mathbf{s}_{h}^{\mbox{s}},(l,\bar{v},0,t^{-}),\underline{a}^{\mbox{f}})$
and $t^{\mbox{+c}}:=(t^{\mbox{+f}}\mathbf{s}_{h}^{\mbox{s}},(l,\bar{v},0,t^{-}),\underline{a}^{\mbox{f}})$.
If $t^{\mbox{mc}}=-\infty$, the algorithm cannot find a feasible
solution and return $p^{\mbox{f}}((l,v,t^{-}),p^{\mbox{s}})=\emptyset$,
$t^{\mbox{mf}}((l,v,t^{-}),p^{\mbox{s}})=-\infty$ and $t^{+\mbox{f}}((p,v,t^{-}),p^{\mbox{s}})=-\infty$.
If $t^{\mbox{+c}}\in[t_{h}^{\mbox{s}},t_{h+1}^{\mbox{s}}]$, set $t^{\mbox{m}}=t^{\mbox{mc}}$
and $t^{+}=t^{+\mbox{c}}$, and go to the Step FSP-3. Otherwise, $t^{\mbox{+c}}$
shall be $\infty$. Then if $h<\bar{h}$, set $h=h+1$ and repeat
this step. 
\item [{FSP-3:}] If $v<\bar{v}$ and $t^{\mbox{a}}>t^{-}$, append segment
$(l,v,\bar{a}^{\mbox{f}},t^{-},\min(t^{\mbox{m}},t^{\mbox{a}}))$
to $p^{\mbox{f}}$ (appending means adding this segment as the last
element of $p^{\mbox{f}}$). If $t^{\mbox{m}}>t^{\mbox{a}},$ append
$(l^{\mbox{a}},\bar{v},0,t^{\mbox{a}},t^{\mbox{m}})$ to $p^{\mbox{f}}$.
If $t^{+}<t_{\bar{h}+1}^{\mbox{s}}$, which indicates $t^{+}\in[t_{h}^{\mbox{s}},t_{h+1}^{\mbox{s}}]$,
we first append segment $(l_{h}^{\mbox{s}}+v_{h}^{\mbox{s}}(t^{+}-t_{h}^{\mbox{s}})+0.5a_{h}^{\mbox{s}}(t^{+}-t_{h}^{\mbox{s}})^{2},v_{h}^{\mbox{s}}+a_{h}^{\mbox{s}}(t^{+}-t_{h}^{\mbox{s}}),a_{h}^{\mbox{s}},t^{+},t_{h+1}^{\mbox{s}})$
to $p^{\mbox{f}}$, and then append all segments $[\mathbf{s}_{h'}^{\mbox{s}}]_{h'=h+1,\cdots,\bar{h}}$
to $p^{\mbox{f}}$. 
\item [{FSP-4:}] Finally, return $p^{\mbox{f}}((l,v,t^{-}),p^{\mbox{s}})=p^{\mbox{f}}$,
$t^{\mbox{mf}}((l,v,t^{-}),p^{\mbox{s}})=t^{\mbox{m}}$ and $t^{+\mbox{f}}((p,v,t^{-}),p^{\mbox{s}})=t^{+}$.
\end{description}
\begin{defn}
\label{def:backward_shooting} This definition further specifies how
we take a symmetric move in the backward shooting, as illustrated
in Figure \ref{fig:shooting_operation}(b). Given a quadratic segment
$\mathbf{s}':=\left(l',v',a',t'^{-},t'^{+}\right)$ with $t'^{-}<t'^{+}$
(e.g., a segment generated from the forward shooting process), a feasible
state point $\left(l,v,t^{-}\right)$ with $t^{-}>t'^{-}$, acceleration
rate $a^{+}\ge0$ and deceleration rate $a^{-}<0$ (and $a^{-}\le a'$),
we construct a \emph{backward shooting segment} $\mathbf{s}:=\left(l,v,a^{+},t^{-},t^{\mbox{m}}\right)$
preceded by a \emph{backward merging segment} $\mathbf{s}^{\mbox{m}}:=\left(l^{\mbox{m}},v^{\mbox{m}},a^{-},t^{\mbox{m}},t^{+}\right)$
where again $v^{\mbox{m}}:=v+a^{+}(t^{\mbox{m}}-t^{-})$, $l^{\mbox{m}}:=l+v(t^{\mbox{m}}-t^{-})+0.5a^{+}(t^{\mbox{m}}-t^{-})^{2}$,
and $t^{+}\le t^{\mbox{m}}\le t^{-}$, such that again condition \eqref{eq:non_cross_condition}
is satisfied (and thus $\mathbf{s}'$ is above $\mathbf{s}$ and $\mathbf{s}{}^{\mbox{m}}$).
And again there are three cases in determining $t^{-}$ and $t^{\mbox{m}}$
values: (I) if no $t^{\mbox{m}}\in(-\infty,t^{-}]$ can be found to
satisfy constraint \eqref{eq:non_cross_condition}, this shooting
operation is infeasible and return $t^{m}=t^{+}=\infty$; (II ) otherwise
we try to find $t^{+}\in\left[t'^{+},\min\left\{ t^{-},t'^{-}\right\} \right]$
and $t^{\mbox{m}}\in[t^{+},t^{-}]$ such that $\mathbf{s}'$ and $\mathbf{s}^{\mbox{m}}$
get tangent at time $t^{+}$ (as Figure \ref{fig:shooting_operation}(b)
indicates); and (III) if no such $t^{+}$ is found, set $t^{m}=t^{+}=-\infty$.
Solutions $t^{\mbox{m}}$ and $t^{-}$ are denoted as functions $t^{\mbox{mb}}\left(\mathbf{s}',\left(l,v,t^{-}\right),a^{+},a^{-}\right)$
and $t^{\mbox{-b}}\left(\mathbf{s}',\left(l,v,t^{-}\right),a^{+},a^{-}\right)$,
respectively, and they can be solved analytically in the following
\emph{backward shooting operation} ($BSO$) algorithm. \end{defn}
\begin{description}
\item [{BSO-1:}] If $D[\mathbf{s}'-\left(l,v,a^{-},t^{-},-\infty\right)]<0$,
there is no feasible solution, and we just return $t^{\mbox{m}}=t^{+}=-\infty$
(Case I). Go to Step BSO-3.
\item [{BSO-2:}] Again shift the origin point to $t^{-}$ and denote $\hat{t}'^{-}:=t'^{-}-t^{-}$,
$\hat{t}'^{+}:=t'^{+}-t^{-}$, $\hat{t}^{\mbox{m}}:=t^{\mbox{m}}-t^{-}$,
$\hat{t}^{\mbox{+}}:=t^{\mbox{+}}-t^{-}$and $\hat{t}^{-}:=\min\{\hat{t}'^{+},0\}$.
And obtain $\mathbf{q}$ by subtracting $\left(l^{\mbox{m}},v^{\mbox{m}},a^{-},t^{\mbox{m}}\right)$
from $(l',v',a',t'^{-})$, which is formulated the same as that in
Step FSO-2. 
\item [{BSO-2-1:}] If $a^{'}=a^{-}$, test whether $\hat{l}=0$ with $\hat{t}^{\mbox{m}}=\left(v'-v-a'\hat{t}'^{-}\right)/\left(a^{+}-a^{-}\right)$.
If yes and $\hat{t}^{\mbox{m}}\in\left[\hat{t}^{-},\hat{t}'^{-}\right]$
(Case II),
\item [{\textmd{set}}] $t^{\mbox{m}}=t^{-}+\hat{t}^{\mbox{m}}$ and $t^{\mbox{+}}=t'^{-}$;
Otherwise, return $\hat{t}^{\mbox{m}}=\hat{t}{}^{+}=-\infty$ (Case
III). Go to BSO-3.
\item [{BSO-2-2:}] If $a^{'}>a^{-}$, we need to again solve $\alpha\left(\hat{t}^{\mbox{m}}\right)^{2}+\beta\hat{t}^{\mbox{m}}+\gamma=0$
formulated in Step FSO-2-2. In case of $\alpha=\beta=0$ (Case III),
set $t^{\mbox{+}}=t^{\mbox{m}}=-\infty$, and go to Step FSO-3. Otherwise,
we again try candidate solutions $\hat{t}^{\mbox{mc}}$ and $\hat{t}^{+\mbox{c}}$.
In case of $\alpha=0$ but $\beta\neq0$, solve $\hat{t}^{\mbox{mc}}=-\gamma/\beta$
and $\hat{t}^{+\mbox{c}}$ with equation \eqref{eq:t_hat_plus}. In
case of $\alpha\neq0$, then we solve two sets of solutions $\hat{t}^{\mbox{mc}}$
and $\hat{t}^{\mbox{+c}}$ with equations \eqref{eq:t_hat_m} and
\eqref{eq:t_hat_plus} respectively. if the candidate solutions are
not real numbers, then we just set $\hat{t}^{\mbox{+c}}=\hat{t}^{\mbox{mc}}=\infty$.
Otherwise, try both sets of solutions and select the set satisfying
$\hat{t}^{+\mbox{c}}\le\hat{t}^{\mbox{mc}}\le0$. With this, if we
obtain $\hat{t}^{+\mbox{c}}\in\left[\hat{t}^{-},\hat{t}'^{-}\right]$
(Case II), we set $t^{\mbox{m}}=t^{-}+\hat{t}^{\mbox{mc}}$ and $\hat{t}^{\mbox{+}}=\hat{t}^{\mbox{+c}}$.
Otherwise, set $t^{\mbox{+}}=t^{\mbox{m}}=\infty$. Go to Step BSO-3.
\item [{BSO-3:}] Finally, we return $t^{\mbox{mb}}\left(\mathbf{s}',\left(l,v,t^{-},a^{+}\right),a^{-}\right)=t^{m}$
and $t^{\mbox{+b}}\left(\mathbf{s}',\left(l,v,t^{-},a^{+}\right),a^{-}\right)=t^{+}$.\end{description}
\begin{defn}
Symmetric to Definition \ref{def: FSP}, we extend one backward move
BSO to the following \emph{backward shooting process} \emph{(BSP)}.
We consider \emph{a backward shooting template trajectory }starting
at a feasible entry state point $(l,v,t^{-})$ composed by one or
both of $\mathbf{s}^{-\mbox{a}}:=(l,v,\bar{a}^{\mbox{f}},t^{-},t^{-\mbox{a}}:=t^{-}-v/\bar{a}^{\mbox{b}})$
(which decelerates backward from $v$ to $0$ given $v>0$) and $\mathbf{s}^{-\mbox{\ensuremath{\infty}}}:=\left(l^{\mbox{-a}}:=l-0.5v^{2}/\bar{a}^{\mbox{b}},0,0,t^{\mbox{-a}},-\infty\right)$
(which denotes the vehicle is stopped prior to time $t^{\mbox{-a}}$),
i.e., 
\[
p^{\mbox{t}}:=\begin{cases}
\left[\mathbf{s}^{-\infty},\mathbf{s}^{\mbox{-a}}\right], & \mbox{if }v>0;\\
\left[\mathbf{s}^{-\infty}\right], & \mbox{if }v=0.
\end{cases}
\]
Further, we are given the original trajectory (e.g., those generated
from FSP) 
\[
p^{\mbox{f}}:=\left\{ \mathbf{s}_{h}:=\{l_{h},v_{h},a_{h},t_{h},t_{h+1}\}\right\} {}_{h=1,\cdots,\bar{h}}.
\]
We will find a \emph{backward shooting trajectory section} $p^{\mbox{b}}((l,v,t^{-}),p)$
that is to merge into $p^{\mbox{f}}$ with a merging segment $\left(p^{\mbox{t}}(t^{\mbox{m}}),\dot{p}(t^{\mbox{m}}),\underline{a}^{\mbox{b}},t^{\mbox{m}},t^{\mbox{+}}\right)$
and does not exceed (i.e., going left of) $p^{\mbox{f}}$ at any time.
This can be solved analytically with the following BSP algorithm. \end{defn}
\begin{description}
\item [{BSP-1:}] Initiate $h$ being the largest segment index of $p^{\mbox{f}}$
such that $t_{h}<t^{-}$, set $t^{+}=t^{\mbox{m}}=-\infty$, $p^{\mbox{b}}=\emptyset$
and iterate through the segments in $p^{\mbox{f}}$.
\item [{BSP-2:}] If $v>0$, apply BSO to solve $t^{\mbox{mc}}:=t^{\mbox{mb}}\left(\mathbf{s}_{h},(l,v,\bar{a}^{\mbox{b}},t^{-}),\underline{a}^{\mbox{b}}\right)$
and $t^{+\mbox{c}}:=t^{\mbox{mf}}\left(\mathbf{s}_{h},(l,v,\bar{a}^{\mbox{b}},t^{-}),\underline{a}^{\mbox{b}}\right)$.
If $t^{\mbox{mc}}<t^{-\mbox{a}}$, revise $t^{\mbox{mc}}:=t^{\mbox{mb}}\left(\mathbf{s}_{h},(l^{\mbox{-a}},0,0,t^{\mbox{-a}}),\underline{a}^{\mbox{b}}\right)$
and $t^{\mbox{+c}}:=t^{\mbox{+b}}\left(\mathbf{s}_{h},(l^{\mbox{-a}},0,0,t^{\mbox{-a}}),\underline{a}^{\mbox{b}}\right)$
with BSO. If $v=0$, directly solve $t^{\mbox{mc}}:=t^{\mbox{mb}}\left(\mathbf{s}_{h},(l,0,0,t^{-}),\underline{a}^{\mbox{b}}\right)$
and $t^{\mbox{+c}}:=t^{\mbox{+b}}\left(\mathbf{s}_{h},(l,0,0,t^{-}),\underline{a}^{\mbox{b}}\right)$
with BSO. If $t^{\mbox{mc}}=\infty$, the algorithm cannot find a
feasible solution (because any trajectory through $(l,v,t^{-})$ will
go above $p^{\mbox{f}}$) and thus returns $p^{\mbox{b}}((l,v,t^{-}),p)=p^{\mbox{eb}}\left((l,v,t^{-}),p^{\mbox{f}}\right)=\emptyset$.
If $t^{\mbox{+c}}=-\infty$ and $h=1$, the algorithm cannot find
a feasible backward trajectory that can touch $p^{\mbox{f}}$ and
thus return $p^{\mbox{b}}((l,v,t^{-}),p)=p^{\mbox{eb}}\left((l,v,t^{-}),p^{\mbox{f}}\right)=\emptyset$.
If $t^{\mbox{+c}}\in[t_{h}^{\mbox{s}},t_{h+1}^{\mbox{s}}]$, set $t^{\mbox{m}}=t^{\mbox{mc}}$
and $t^{+}=t^{+\mbox{c}}$ and go to Step BSP-3. Otherwise, if $h>1$,
set $h=h-1$ and repeat this step. 
\item [{BSP-3:}] If $v>0$ and $t^{\mbox{-a}}<t^{-}$, insert segment $\left(l^{-\mbox{ma}},v^{-\mbox{ma}},\bar{a}^{\mbox{b}},t^{-\mbox{ma}},t^{-}\right)$
to $p^{\mbox{b}}$ (inserting means adding this segment before the
head of $p^{\mbox{b}}$), where $t^{-\mbox{ma}}:=\max(t^{\mbox{m}},t^{\mbox{-a}})$,
$v^{-\mbox{ma}}:=v-\bar{a}^{\mbox{b}}(t^{-}-t^{-\mbox{ma}})$ and
$l^{-\mbox{ma}}:=l-v(t^{-}-t^{-\mbox{ma}})+0.5\bar{a}^{\mbox{b}}(t^{-}-t^{-\mbox{ma}})^{2}$.
If $t^{\mbox{m}}<t^{-\mbox{a}},$ insert $(l^{\mbox{-a}},0,0,t^{\mbox{m}},t^{\mbox{-a}})$
to $p^{\mbox{b}}$. Then we insert merging segment $(l^{\mbox{m}},v^{\mbox{m}}\bar{a}^{\mbox{b}},t^{+},t^{\mbox{m}})$
to $p^{\mbox{b}}$ where $l^{\mbox{m}}:=l_{h}+v_{h}(t^{+}-t_{h}^{\mbox{s}})+0.5a_{h}(t^{+}-t_{h}^{\mbox{s}})^{2},$
and $v^{\mbox{m}}:=v_{h}+a_{h}(t^{+}-t_{h})$. 
\item [{BSP-4:}] Now we have obtained a \emph{backward shooting trajectory
section} $p^{\mbox{b}}\left((l,v,t^{-}),p^{\mbox{f}}\right)=p^{\mbox{b}}\left(t^{\mbox{m}}:t^{-}\right)$
(or $p^{\mbox{b}}$ for simplicity). We further extend $p^{\mbox{b}}\left(t^{\mbox{m}}:t^{-}\right)$
by inserting $p^{\mbox{f}}(t_{1}:t^{\mbox{m}})$ and appending $p^{\mbox{f}}\left(\left(l,v,t^{-}\right),p^{\mbox{f}}\right)$
generated from an \emph{auxiliary FSP}, and construct the \emph{extended
backward shooting trajectory} $p^{\mbox{eb}}\left((l,v,t^{-}),p^{\mbox{f}}\right)$
$:=\left[p^{\mbox{f}}(t_{1}:t^{\mbox{m}}),p^{\mbox{b}}\left(t^{\mbox{m}}:t^{-}\right),p^{\mbox{f}}\left(\left(l,v,t^{-}\right),p^{\mbox{f}}\right)\right]$.
Return $p^{\mbox{b}}\left((l,v,t^{-}),p^{\mbox{f}}\right)$ and $p^{\mbox{eb}}\left((l,v,t^{-}),p^{\mbox{f}}\right)$. 
\item [{}]~
\end{description}
Now we are ready to present the proposed shooting algorithm that yields
a trajectory vector $P\left(\bar{a}^{\mbox{f}},\underline{a}^{\mbox{f}},\bar{a}^{\mbox{b}},\bar{a}^{b}\right)$
as a functional of these four acceleration rates. 
\begin{description}
\item [{SH-1:}] Initialize control variables $\bar{a}^{\mbox{f}}$, $\underline{a}^{\mbox{f}}$,
$\bar{a}^{\mbox{b}}$ and $\underline{a}^{\mbox{b}}$. Set $n=1$,
trajectory vector $P=\emptyset$.
\item [{SH-2:}] Apply the FSP to obtain $p_{n}^{\mbox{f}}=\left[\mathbf{s}_{nk}=\left(l_{nk},v_{nk},a_{nk},t_{nk},t_{n(k+1)}\right)\right]{}_{k=1,\cdots,\bar{k}_{n}}=p^{\mbox{f}}((0,v_{n}^{-},t_{n}^{-}),p_{n-1}^{\mbox{s}})$
(define $p_{0}^{\mbox{s}}:=\emptyset$). We call this process the
\emph{primary FSP} (to differentiate from the auxiliary FSP in the
BSP). If $p_{n}^{\mbox{f}}=\emptyset$, which means that this algorithm
cannot find a feasible solution for trajectory platoon $P$, set $P\left(\bar{a}^{\mbox{f}},\underline{a}^{\mbox{f}},\bar{a}^{\mbox{b}},\bar{a}^{b}\right):=P=\emptyset$
and return. 
\item [{SH-3:}] This steps checks the need for the BSP. Find the segment
index $k_{n}^{L}$ such that $L\in\left[l_{nk^{L}},l_{n\left(k^{L}+1\right)}\right)$
and solve the time $t_{n}^{L}$ when vehicle $n$ passes location
$L$ as follows 
\[
t_{n}^{L}:=t_{nk^{L}}+\begin{cases}
\frac{-v_{nk^{L}}+\sqrt{v_{nk^{L}}^{2}+2a_{nk^{L}}(L-l_{nk^{L}})}}{a_{nk^{L}}}, & \mbox{if }a_{nk^{L}}\neq0;\\
\frac{L-l_{nk^{L}}}{v_{nk^{L}}}, & \mbox{if }a_{nk^{L}}=0.
\end{cases}
\]
 If $t_{n}^{L}=G\left(t_{n}^{L}\right)$, which means that $p_{n}^{\mbox{f}}$
does not violate the exit boundary constraint \eqref{eq: set-exit-boundary}
(or does not run into the red light), we set $p_{n}=p_{n}^{\mbox{f}}$
and go to SH4. Otherwise, $p_{n}^{\mbox{f}}$ violates constraint
\eqref{eq: set-exit-boundary} and we need to apply BSP to revise
it in the following step. Set $v_{n}^{L}:=v_{nk^{L}}+a_{nk^{L}}\left(t_{n}^{L}-t_{nk^{L}}\right),$
apply the BSP to solve $p_{n}^{\mbox{b}}=p^{\mbox{b}}\left(\left(L,v_{n}^{L},G\left(t_{n}^{L}\right)\right),p_{n}^{\mbox{f}}\right)$.
If the start location of obtain $p_{n}^{\mbox{b}}$ is no less than
0, set $p_{n}=p^{\mbox{eb}}\left(\left(L,v_{n}^{L},G\left(t_{n}^{L}\right)\right),p_{n}^{\mbox{f}}\right).$
Otherwise, $p_{n}^{\mbox{b}}$ can not meet $p_{n}^{\mbox{f}}$ on
this highway section, and return $P\left(\bar{a}^{\mbox{f}},\underline{a}^{\mbox{f}},\bar{a}^{\mbox{b}},\bar{a}^{b}\right):=\emptyset$
.
\item [{SH-4:}] Append $p_{n}$ to $P$. Return $P^{\mbox{SH}}\left(\underline{a}^{\mbox{f}},\bar{a}^{\mbox{f}},\underline{a}^{\mbox{b}},\bar{a}^{\mbox{b}}\right):=P$
if $n=N$, or otherwise set $n=n+1$ and go to SH-2.
\end{description}
Although FSO (Definition \ref{def: forward_shooting }) and BSO (Definition
\ref{def:backward_shooting}) do not explicitly impose speed limits
to the generated trajectory segments, as long as a non-empty trajectory
vector $P$ is returned by the SH algorithm, $P$ shall satisfy all
constraints defined in Section \ref{sec:Problem-Statement} (or $P\in\mathcal{P}$),
as proven in the following proposition. 
\begin{prop}
\label{prop: SH_feas_necessity}If the SH algorithm successfully generates
a vector of trajectories $P\left(\underline{a}^{\mbox{f}},\bar{a}^{\mbox{f}},\underline{a}^{\mbox{b}},\bar{a}^{\mbox{b}}\right)$
with $\underline{a}^{\mbox{f}},\underline{a}^{\mbox{b}}\in\left[\underline{a},0\right)$
and $\bar{a}^{\mbox{f}},\bar{a}^{\mbox{b}}\in\left(0,\bar{a}\right]$,
they shall fall in the feasible trajectory vector set $\mathcal{P}$
defined in \eqref{fig:shooting_operation}. \end{prop}
\begin{proof}
Since the acceleration of each segment generated from the SH algorithm
is either explicitly specified within $[\underline{a},\bar{a}]$ (i.e.,
one of $\underline{a}^{\mbox{f}},\bar{a}^{\mbox{f}},\underline{a}^{\mbox{b}},\bar{a}^{\mbox{b}}$
and 0) or just following a shadow trajectory's acceleration that shall
fall in $[\underline{a},\bar{a}]$ as well. So the constraint with
respect to acceleration in \eqref{eq:set-kinematic-constraints} is
satisfied. 

Then, we will use mathematical induction to exam the remaining constraints
in \eqref{eq:set-kinematic-constraints}-\eqref{eq: set-safety-constraints}.
First for vehicle 1, FSP can generate $p_{1}$ with at maximum 2 segments,
which apparently falls in $\mathcal{T}_{1}^{-}$ (and thus both constraints
\eqref{eq:set-kinematic-constraints}-\eqref{eq: set-entry-boundary}.
are satisfied). If BSP is not needed, exit constraint \eqref{eq: set-exit-boundary}
is automatically satisfied and thus $p_{1}\in\mathcal{T}_{1}$. Otherwise,
the new segments generated from BSP below $L$ start from a speed
no greater than $\bar{v}$ and decelerate backwards to a value no
less than 0 and then increase the speed and merge into the forward
shooting trajectory at a speed no greater than $\bar{v}$. During
this process, the speed shall always stay within $[0,\bar{v}]$ and
therefore constraint \eqref{eq:set-kinematic-constraints} remains
valid. The auxiliary FSP is similar to the primary FSP and thus will
not violate constraint \eqref{eq:set-kinematic-constraints} as well.
Further, the BSP step SH3 does not affect the entry boundary condition
\eqref{eq: set-entry-boundary} and makes exit condition \eqref{eq: set-exit-boundary}
feasible in addition. Therefore we obtain $p_{1}\in\mathcal{T}_{1}$. 

Then we assume that $p_{n-1}\in\mathcal{T}_{n-1},\forall n=2,\cdots,N$,
and we will prove that $p_{n}\in\mathcal{T}_{n}(p_{n-1})$. If $p_{n}$
is not blocked by $p_{n-1}$ during the FSP, then the construction
of $p_{n}$ is similar to that of $p_{1}$ and thus $p_{n}$ should
automatically satisfy constraints \eqref{eq:set-kinematic-constraints}-\eqref{eq: set-exit-boundary}
and thus $p_{n}\in\mathcal{T}_{n}$. Further, $p_{n}$ shall be always
below $p_{n-1}^{\mbox{s}}$ and therefore $p_{n}$ shall satisfy safety
constraint \eqref{eq: set-safety-constraints}, i.e., $p_{n}\in\mathcal{T}_{n}(p_{n-1})$.
Otherwise, if $p_{n}$ is blocked by $p_{n-1}$, the construction
of $p_{n}$ would generate some more segments that merge the forward
trajectory into $p_{n-1}^{\mbox{s}}$ (e.g., as Figure \ref{fig:shooting_operation}(a)
illustrates) and then follow $p_{n-1}^{\mbox{s}}$, as compared with
the construction of $p_{1}$. Due to the induction assumption, the
segments following $p_{n-1}^{\mbox{s}}$ shall satisfy kinematic constraint
\eqref{eq:set-kinematic-constraints} the same as the corresponding
segments in $p_{n-1}$. For the merging segment, since it starts from
a forward shooting segment and ends at a shadow segment and therefore
its speed range should be bounded by $[0,\bar{v}]$. Therefore, we
obtain $p_{n}\in\mathcal{T}$. Again, the primary FSP ensures that
$p_{n}$ satisfies the entry boundary \eqref{eq: set-entry-boundary}
and BSP ensures that $p_{n}$ satisfy exit constraint \eqref{eq: set-exit-boundary}.
This yields $p_{n}\in\mathcal{T}_{n}$. Further we see that any segment
generated from FSP shall be either below or on $p_{n-1}^{\mbox{s}}$.
If the BSP generates new segments, they shall be strictly right to
the forward shooting trajectory. Therefore, $p_{n}$ shall fall in
$\mathcal{T}_{n}\left(p_{n-1}\right)$. This proves that $P\left(\underline{a}^{\mbox{f}},\bar{a}^{\mbox{f}},\underline{a}^{\mbox{b}},\bar{a}^{\mbox{b}}\right)\in\mathcal{P}$.
\end{proof}

\subsection{Shooting Heuristic for the Lead Vehicle Problem\label{sub:SH_LVP}
}

The above proposed SH algorithm can be easily adapted to solve LVP
\eqref{eq:P_feasible_platoon_LVP}. The only modifications are to
fix $p_{1}$ and to set $\mathcal{G}=[-\infty,\infty]$ (and therefore
no BSP is needed). We denote this adapted SH for LVP by SHL. More
importantly, we further find that SHL can be alliteratively implemented
in a parallel manner. Each trajectory can be calculated directly from
the input parameters without its preceding trajectory's information.
We denote this parallel alternative of SHL with PSHL. Apparently,
PSHL allows further improved computational efficiency with parallel
computing. This section describes PSHL and validates that PSHL indeed
solves LVP. We first introduce an \emph{extended forward shooting
operation} (EFSO) that merges two feasible trajectories.
\begin{description}
\item [{EFSO-1:}] Given two feasible trajectories $p:=\left[\mathbf{s}_{j}:=\left(l_{j},v_{j},a_{j},t_{j},t_{j+1}\right)\right]_{j=1,2,\cdots,J}$
and $p'=\mathbf{s}'_{k}:=\left[\left(l'_{k},v'_{k},a'_{k},\right.\right.$
$\left.\left.t'_{k},t'_{k+1}\right)\right]_{k=1,2,\cdots,K}\in\mathcal{T}$,
and deceleration rate $a^{-}<0$. Set iterators $j=1$ and $k=1$
. 
\item [{EFSO-2:}] Solve $t_{jk}^{\mbox{m}}:=t^{\mbox{mf}}\left(\mathbf{s}'_{k},\left(l_{j},v_{j},t_{j},a_{j}\right),a^{-}\right)$
and $t_{jk}^{+}=t^{\mbox{+f}}\left(\mathbf{s}'_{k},\left(l_{j},v_{j},t_{j},a_{j}\right),a^{-}\right)$.
If $t_{jk}^{\mbox{m}}=-\infty$, return $t^{\mbox{m}}(p',p,a^{-})=t^{+}(p',p,a^{-})=-\infty$.
If $t_{jk}^{\mbox{m}}<\infty$ and $t_{jk}^{\mbox{m}}\in\left[t_{j},t_{j+1}\right]$,
return $t^{\mbox{m}}(p',p,a^{-})=t_{jk}^{\mbox{m}}$, $t^{\mbox{+}}(p',p,a^{-})=t_{jk}^{+}$.
If $k<K$, set $k=k+1$ and repeat this step; otherwise, go to the
next step.
\item [{EFSO-2:}] If $j<J$, set $j=j+1$ and $k=1$ and go to Step EFSO-2.
Otherwise, return $t^{\mbox{m}}(p',p,a^{-})=t^{+}(p',p,a^{-})=\infty$.
\end{description}
Based on EFSO, we devise the following \emph{extended forward shooting
process} (EFSP) that generates a forward shooting trajectory constrained
by a series of upper bound trajectories.
\begin{description}
\item [{EFSP-1:}] Given a feasible state point $(l,v,t)$, and a set of
trajectories $\left\{ p_{m}\in\mathcal{T}\right\} _{m=1,\cdots,M}$.
Initiate $m=1$ and $p=p^{\mbox{f}}\left((l,v,t),\emptyset\right)$
with FSP.
\item [{EFSP-2:}] Call EFSO to solve $t^{\mbox{m}}:=t^{\mbox{m}}\left(p_{m},p,\underline{a}^{\mathbf{\mbox{f}}}\right)$
and $t^{\mbox{+}}:=t^{\mbox{+}}\left(p_{m},p,\underline{a}^{\mathbf{\mbox{f}}}\right)$.
If $t^{\mbox{m}}=-\infty$, return $p^{\mbox{f}}\left((l,v,t),\left\{ p_{m}\right\} _{m=1,\cdots,M}\right)=\emptyset$.
If $t^{\mbox{m}}<\infty$, revise $p:=\left[p\left(t^{-}\left(p\right):t^{\mbox{m}}\right),\left(p\left(t^{\mbox{m}}\right),\dot{p}\left(t^{\mbox{m}}\right),\underline{a}^{\mathbf{\mbox{f}}},t^{\mbox{m}},t^{+}\right),\right.$
$\left.p_{m}\left(t^{+}:\infty\right)\right]$. 
\item [{EFSP-3:}] If $m<M$, set $m=m+1$ and go to PSFP-2; otherwise return
$p^{\mbox{f}}\left((l,v,t),\left\{ p_{m}\right\} _{m=1,\cdots,M}\right)=p$.
\end{description}
Then we are ready to present the PSHL algorithm as follows.
\begin{description}
\item [{}]~
\item [{PSHL-1:}] Given lead trajectory $p_{1}\in\mathcal{T}$ and boundary
condition $\left[v_{n}^{-},t_{n}^{-}\right]_{n\in\mathcal{N}\backslash\{1\}}$.
Initialize acceleration parameters $\bar{a}^{\mbox{f}}$ and $\underline{a}^{\mbox{f}}$.
Set $\bar{p}_{1}:=p_{1},\bar{p}_{n}:=p^{\mbox{f}}\left((0,v_{n}^{-},t_{n}^{-}),\emptyset\right),\forall n\in\mathcal{N}\backslash\{1\}.$
Set initial trajectory platoon $P=[p_{1}]$,$n=2$.
\item [{PSHL-2:}] Solve $p_{n}:=p^{\mbox{f}}\left(\left(0,v_{n}^{-},t_{n}^{-}\right),\left\{ \bar{p}_{m}^{\mbox{s}^{n-m}}\right\} _{m=1,\cdots,n-1}\right)$
with EFSP. 
\item [{PSHL-3:}] If $p_{n}=\emptyset$, return $P^{\mbox{PSH}}\left(\underline{a}^{\mbox{f}},\bar{a}^{\mbox{f}}\right):=\emptyset$.
Otherwise, append $p_{n}$ to $P$. If $n=N$, return $P^{\mbox{PSH}}\left(\underline{a}^{\mbox{f}},\bar{a}^{\mbox{f}}\right):=P$;
otherwise, set $n=n+1$ and go to PSHL-2.\end{description}
\begin{prop}
$P^{\mbox{PSH}}\left(\underline{a}^{\mbox{f}},\bar{a}^{\mbox{f}}\right)$
obtained from PSHL \textup{is identical to $P^{\mbox{SH}}\left(\underline{a}^{\mbox{f}},\bar{a}^{\mbox{f}},\underline{a}^{\mbox{b}},\bar{a}^{\mbox{b}}\right)$}
when $p_{1}\in\mathcal{T}$ is fixed and $\mathcal{G}=\left(-\infty,\infty\right)$.\label{prop: PSHL=00003DSHL}\end{prop}
\begin{proof}
We first consider the cases that $P^{\mbox{PSH}}\left(\underline{a}^{\mbox{f}},\bar{a}^{\mbox{f}}\right)\neq\emptyset$.
We will prove this proposition via induction with the iterator being
vehicle index $n$. Let $\left[p_{1,}p_{2},\cdots,p_{N}\right]$ denote
the trajectories in $P^{\mbox{PSH}}\left(\underline{a}^{\mbox{f}},\bar{a}^{\mbox{f}}\right)$
and $\left[p'_{1,}p'_{2},\cdots,p'_{N}\right]$ denote the trajectories
in $P^{\mbox{SH}}\left(\underline{a}^{\mbox{f}},\bar{a}^{\mbox{f}},\underline{a}^{\mbox{b}},\bar{a}^{\mbox{b}}\right)$.
Apparently, when $n=1$,$p'_{1}=p_{1}$ since the lead vehicle's trajectory
is fixed. Assume that when $n=k$, $p'_{k}=p_{k}$. Based on the definition
of SH, $p'_{k+1}=\left[\bar{p}_{k+1}\left(t_{k+1}^{-}:t^{\mbox{m}}\right),\left(\bar{p}_{k+1}\left(t^{\mbox{m}}\right),\dot{\bar{p}}_{k+1}\left(t^{\mbox{m}}\right),\underline{a}^{\mathbf{\mbox{f}}},t^{\mbox{m}},t^{+}\right),p_{k}^{\mathbf{'\mbox{s}}}\left(t^{+}:\infty\right)\right]$
where $\bar{p}_{k+1}:=p^{\mbox{f}}\left((0,v_{k+1}^{-},t_{k+1}^{-}),\emptyset\right)$,
$p_{k}^{\mathbf{'\mbox{s}}}(t):=p'_{k}(t-\tau)-s$, $t^{\mbox{m}}:=t^{\mbox{m}}\left(p_{k}^{\mathbf{'\mbox{s}}},\bar{p}_{k+1},\underline{a}^{\mathbf{\mbox{f}}}\right)$,
$t^{\mbox{+}}:=t^{\mbox{+}}\left(p_{k}^{\mathbf{'\mbox{s}}},\bar{p}_{k+1},\underline{a}^{\mathbf{\mbox{f}}}\right)$.
Based on the induction assumption, $p_{k}^{\mathbf{'\mbox{s}}}=p_{k}^{\mathbf{\mbox{s}}}$
can be obtained by repeatedly calling EFSO to merge $\left\{ \bar{p}_{m}^{\mbox{s}^{k+1-m}}\right\} _{m=1,\cdots,k}$
as in PSHL-2. Therefore, $p'_{k+1}$ is obtained by merging $\left\{ \bar{p}_{m}^{\mbox{s}^{k+1-m}}\right\} _{m=1,\cdots,k+1}$
and thus $p'_{k+1}=p_{k+1}$. 

When $P^{\mbox{PSH}}\left(\underline{a}^{\mbox{f}},\bar{a}^{\mbox{f}}\right)=\emptyset$,
from the above discussion that shows the equivalence of generating
$p_{k}$ and $p'_{k}$, it is obvious that $P^{\mbox{SH}}\left(\underline{a}^{\mbox{f}},\bar{a}^{\mbox{f}}\right)=\emptyset$
as well. This completes the proof.\end{proof}
\begin{cor}
Given lead trajectory $p_{1}\in\mathcal{T}$ , PSHL yields $P^{\mbox{PSH}}\left(\underline{a}^{\mbox{f}},\bar{a}^{\mbox{f}}\right)\in\mathcal{P}^{\mbox{LVP}}\left(p_{1}\right)$
if $\underline{a}^{\mbox{f}},\bar{a}^{\mbox{f}}\in\left[\underline{a},\bar{a}\right]$. 
\end{cor}

\section{Theoretical Property Analysis \label{sec:Theoretical-Properties-of}}

This section analyze theoretical properties of the proposed SH algorithms,
including their solution feasibility and relationship with the classic
traffic flow theory. It is actually quite challenging to analyze such
properties because the original problem defined in Section \ref{sec:Problem-Statement}
involves infinite-dimensional trajectory variables and highly nonlinear
constraints. Fortunately, the concept of time geography \citep{miller2005measurement}
is found related to the bounds to feasible trajectory ranges. We generalize
this concept to enable the following theoretical analysis.

\subsection{Quadratic Time Geography }

We first generalize the time geography theory considering acceleration
range $[\bar{a},\underline{a}]$ in additional to speed range $[0,\bar{v}]$.
These generalized theory, which we call the \emph{quadratic time geography}
(QTG) theory, are illustrated in Figure \ref{fig:gen_time_geography}
and discussed in detail in this subsection.

\begin{figure}
\begin{centering}
\includegraphics[width=0.5\textwidth]{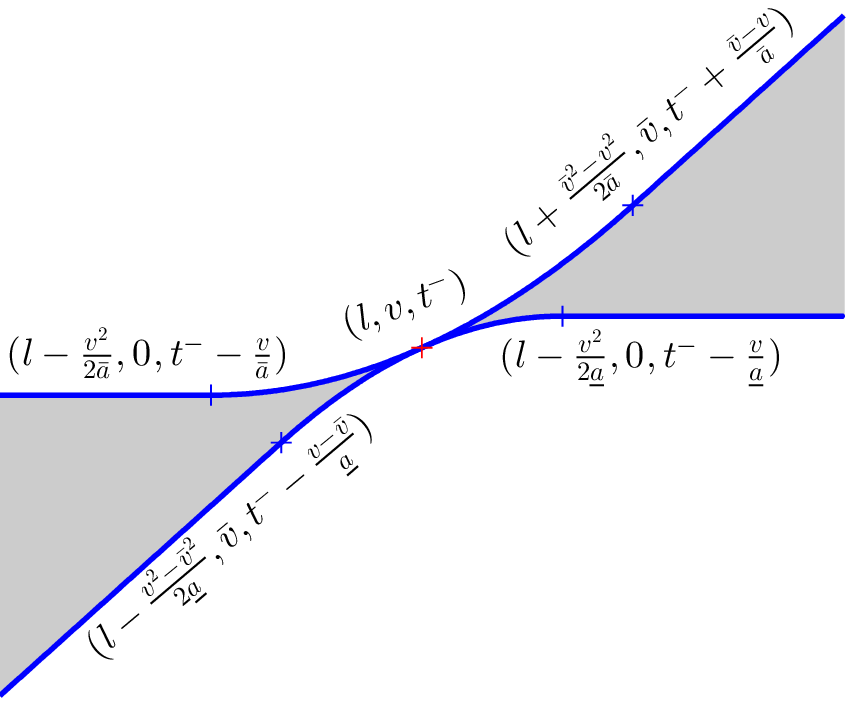}\includegraphics[width=0.5\textwidth]{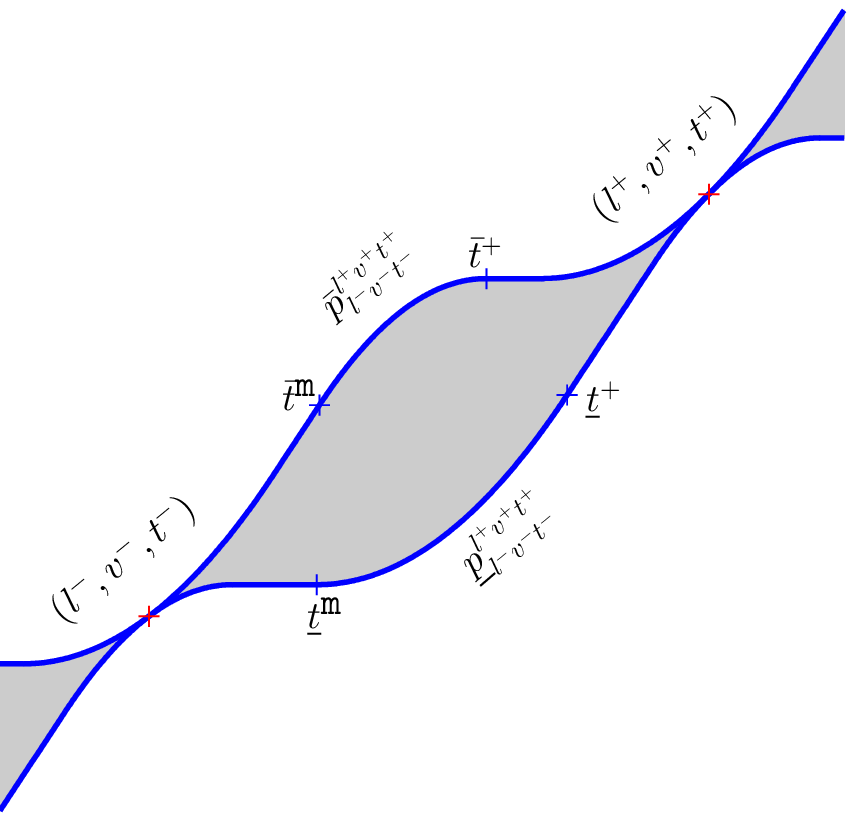}
\par\end{centering}

\begin{centering}
(a)\quad{}\quad{}\quad{}\quad{}\quad{}\quad{}\quad{}\quad{}\quad{}\quad{}\quad{}\quad{}\quad{}\quad{}\quad{}\quad{}(b)
\par\end{centering}

\protect\caption{Illustrations of the generalized time geography theory in : (a) quadratic
cone; and (b) quadratic prism. \label{fig:gen_time_geography}}
\end{figure}

\begin{defn}
We call the set of feasible trajectories (i.e., in $\mathcal{T}$)
passing a common feasible state point \emph{$\left(l,v,t^{-}\right)$
the quadratic cone} of $\left(l,v,t^{-}\right)$, denoted by $\mathcal{C}_{lvt^{-}}$,
illustrated as the shaded area in Figure \ref{fig:gen_time_geography}(a)
and formulated below: 
\[
\mathcal{C}_{lvt^{-}}=\left\{ p\left|p\in\mathcal{T},\,p(t^{-})=l,\,\dot{p}(t^{-})=v,\forall t\in\left(-\infty,\infty\right)\right.\right\} ,
\]
where\emph{ the upper bound} \emph{trajectory} $\bar{p}_{lvt^{-}}$
of $\left(l,v,t^{-}\right)$, illustrated as the top boundary of the
shade in Figure \ref{fig:gen_time_geography}(a), is formulated as

\[
\bar{p}_{lvt^{-}}(t):=\begin{cases}
l-\frac{v^{2}}{2\bar{a}}, & \mbox{if }t\in\left(-\infty,t^{-}-\frac{v}{\bar{a}}\right];\\
l+v(t-t^{-})+0.5\bar{a}(t-t^{-})^{2}, & \mbox{if }t\in\left[t^{-}-\frac{v}{\bar{a}},t^{-}+\frac{\bar{v}-v}{\bar{a}}\right];\\
l+\bar{v}(t-t^{-})-\frac{\left(\bar{v}-v\right)^{2}}{2\bar{a}}, & \mbox{if }t\in\left[t^{-}+\frac{\bar{v}-v}{\bar{a}},\infty\right),
\end{cases}
\]
and \emph{the lower bound} \emph{trajectory $\underline{p}_{lvt^{-}}$}of
$\left(l,v,t^{-}\right)$, illustrated as the bottom boundary of the
shade in Figure \ref{fig:gen_time_geography}(a), is formulated as

\[
\underline{p}_{lvt^{-}}(t)=\begin{cases}
l+\bar{v}(t-t^{-})-\frac{\left(\bar{v}-v\right)^{2}}{2\underline{a}}, & \mbox{if }t\in\left(-\infty,t^{-}-\frac{v-\bar{v}}{\underline{a}}\right];\\
l+v(t-t^{-})+0.5\underline{a}(t-t^{-})^{2}, & \mbox{if }t\in\left[t^{-}-\frac{v-\bar{v}}{\underline{a}},t^{-}+\frac{-v}{\underline{a}}\right];\\
l+\frac{-v^{2}}{2\underline{a}}, & \mbox{if }t\in\left[t^{-}+\frac{-v}{\underline{a}},\infty\right).
\end{cases}
\]
In other words, $\bar{p}_{lvt^{-}}$ is composed of quadratic segments
$(l-\frac{v^{2}}{2\bar{a}},0,0,t^{-}-\frac{v}{\bar{a}},-\infty)$,
$(l-\frac{v^{2}}{2\bar{a}},0,\bar{a},t^{-}-\frac{v}{\bar{a}},t^{-}+\frac{\bar{v}-v}{\bar{a}})$
and $(l+\frac{\bar{v}^{2}-v^{2}}{2\bar{a}},\bar{v},0,t^{-}+\frac{\bar{v}-v}{\bar{a}},\infty)$,
and $\bar{p}_{lvt^{-}}$ is comprised of quadratic segments $(l-\frac{v^{2}-\bar{v}^{2}}{2\underline{a}},\bar{v},0,t^{-}-\frac{v-\bar{v}}{\underline{a}},-\infty)$,
$(l-\frac{v^{2}-\bar{v}^{2}}{2\underline{a}},\bar{v},\bar{a},t^{-}-\frac{v-\bar{v}}{\underline{a}},t^{-}+\frac{-v}{\underline{a}})$
and $(l+\frac{-v^{2}}{2\underline{a}},0,0,t^{-}+\frac{-v}{\underline{a}},\infty)$.
Note that $\bar{p}_{lvt^{-}}=p^{\mbox{f}}\left(\left(l,v,t^{-}\right),\emptyset\right)$
from the FSP with $\bar{a}^{\mbox{f}}=\bar{a}$ and $\underline{a}^{\mbox{f}}=\underline{a}$,
$\bar{p}_{lvt^{-}}(t^{-})=\underline{p}_{lvt^{-}}(t^{-})=l$, and
$\bar{p}_{lvt^{-}}(t)\ge\underline{p}_{lvt^{-}}(t),\forall t\neq t^{-}$.
Further, $\mathcal{C}_{lvt^{-}}$ is always non-empty as long as $(l,v,t^{-})$
is feasible.
\end{defn}

\begin{defn}
We call the set of trajectories in $\mathcal{T}$ passing two feasible
state points $(l^{-},v^{-},t^{-})$ and $(l^{+},v^{+},t^{+})$ with
$t^{-}<t^{+}$, $l^{-}\le l^{+}$\emph{ a quadratic prism}, denoted
by $\mathcal{P}_{l^{-}v^{-}t^{-}}^{l^{+}v^{+}t^{+}}$, as illustrated
in Figure \ref{fig:gen_time_geography}(b) and formulated below: 

\[
\mathcal{P}_{l^{-}v^{-}t^{-}}^{l^{+}v^{+}t^{+}}:=\left\{ p\in\mathcal{T}\left|p(t^{-})=l^{-},\dot{p}(t^{-})=v^{-},p(t^{+})=l^{+},\dot{p}(t^{+})=v^{+}\right.\right\} =\mathcal{C}^{p^{-}v^{-}t^{-}}\bigcap\mathcal{C}^{p^{+}v^{+}t^{+}}.
\]
The upper bound of this $\mathcal{P}_{l^{-}v^{-}t^{-}}^{l^{+}v^{+}t^{+}}$
(as the top boundary of the shade in \ref{fig:gen_time_geography}(b)),
denoted by $\bar{p}_{l^{-}v^{-}t^{-}}^{l^{+}v^{+}t^{+}}$ , should
be composed by $\bar{p}_{l^{-}v^{-}t^{-}}\left(\infty:\bar{t}^{\mbox{m}}\right)$,
merging segment $\bar{\mathbf{s}}^{\mbox{m}}\left(\bar{p}_{l^{-}v^{-}t^{-}}\left(\bar{t}^{\mbox{m}}\right),\dot{\bar{p}}_{l^{-}v^{-}t^{-}}\left(\bar{t}^{\mbox{m}}\right),\bar{a},\bar{t}^{\mbox{m}},\bar{t}^{\mbox{+}}\right)$
and $\bar{p}_{l^{+}v^{+}t^{+}}\left(\bar{t}^{+}:\infty\right)$, where
we can actually apply FSP with $\bar{a}^{\mbox{f}}=\bar{a}$ and $\underline{a}^{\mbox{f}}=\underline{a}$
to obtain connection points $\bar{t}^{\mbox{m}}:=t^{\mbox{mf}}\left(\left(l^{-},v^{-},t^{-}\right),\bar{p}_{l^{+}v^{+}t^{+}}\right)$
and $\bar{t}^{\mbox{+}}:=t^{\mbox{+f}}\left(\left(l^{-},v^{-},t^{-}\right),\bar{p}_{l^{+}v^{+}t^{+}}\right)$.
The lower bound of this prism (as the bottom boundary of the shade
in \ref{fig:gen_time_geography}(b)), denoted by $\underline{p}_{l^{-}v^{-}t^{-}}^{l^{+}v^{+}t^{+}},$
is composed by section $\underline{p}_{l^{-}v^{-}t^{-}}(\infty:\underline{t}^{\mbox{m}})$,
merging segment $\bar{\mathbf{s}}^{\mbox{m}}\left(\underline{p}_{l^{-}v^{-}t^{-}}(\underline{t}^{\mbox{m}}),\dot{\underline{p}}_{l^{-}v^{-}t^{-}}(\underline{t}^{\mbox{m}}),\bar{a},\underline{t}^{\mbox{m}},\underline{t}^{\mbox{+}}\right)$
and $\underline{p}_{l^{+}v^{+}t^{+}}\left(\underline{t}^{+}:\infty\right)$,
where we can apply FSP in a transformed coordinate system to solve
$\underline{t}^{\mbox{m}}$ and $\underline{t}^{\mbox{+}}.$ We shift
the first shift the origin to $(l^{+},t^{+})$ and rotate the whole
coordinate system by 180 degree. Then state points $(l^{+},v^{+},t^{+})$
and $(l^{-},v^{-},t^{-})$ transfer into $(0,v^{+},0)$ and $(l^{+}-l^{-},v^{-},t^{+}-t^{-})$,
respectively, and $\underline{p}_{l^{-}v^{-}t^{-}}$ transfers into
$\hat{\underline{p}}_{l^{-}v^{-}t^{-}}$ defined as $\hat{\underline{p}}_{l^{-}v^{-}t^{-}}(t):=l^{+}-\underline{p}_{l^{-}v^{-}t^{-}}\left(2t^{+}-\left(t^{-}+t\right)\right)$.
Then we solve $\hat{\underline{t}}^{\mbox{m}}:=t^{\mbox{mf}}\left(\left(0,v^{+},0\right),\hat{\underline{p}}_{l^{-}v^{-}t^{-}}\right)$
and $\hat{\underline{t}}^{\mbox{+}}:=t^{\mbox{+f}}\left(\left(0,v^{+},0\right),\hat{\underline{p}}_{l^{-}v^{-}t^{-}}\right)$
with FSP with $a^{\mbox{f}}=-\underline{a}$ and $\underline{a}^{\mbox{f}}=-\bar{a}$.
Then we obtain $\underline{t}^{+}=t^{+}-\hat{\underline{t}}^{\mbox{m}}$
and $\underline{t}^{\mbox{m}}=t^{+}-\hat{\underline{t}}^{\mbox{+}}$. 
\end{defn}
Note that the feasibility of $\mathcal{P}_{l^{-}v^{-}t^{-}}^{l^{+}v^{+}t^{+}}$
depends on the values of $(l^{-},v^{-},t^{-})$ and $(l^{+},v^{+},t^{+})$,
as discussed in the following propositions.
\begin{prop}
\label{prop:prism_feasibility}Given two feasible state points $(l^{-},v^{-},t^{-})$
and $(l^{+},v^{+},t^{+})$, quadratic cone $\mathcal{P}_{l^{-}v^{-}t^{-}}^{l^{+}v^{+}t^{+}}$is
not empty if and only if $D\left(\bar{p}_{l^{+}v^{+}t^{+}}-\underline{p}_{l^{-}v^{-}t^{-}}\right)\ge0$
and $D\left(\bar{p}_{l^{-}v^{-}t^{-}}-\underline{p}_{l^{+}v^{+}t^{+}}\right)\ge0$
.\end{prop}
\begin{proof}
We first prove the necessity. If there exists a feasible trajectory
$p\in\mathcal{P}_{l^{-}v^{-}t^{-}}^{l^{+}v^{+}t^{+}}$, then we know
the $D\left(\bar{p}_{l^{+}v^{+}t^{+}}-p\right)\ge0$ and $D\left(p-\underline{p}_{l^{-}v^{-}t^{-}}\right)\ge0$,
which indicates $D\left(\bar{p}_{l^{+}v^{+}t^{+}}-\underline{p}_{l^{-}v^{-}t^{-}}\right)\ge0$.
Symmetrically, $D\left(\bar{p}_{l^{-}v^{-}t^{-}}-p\right)\ge0$ and
$D\left(p-\underline{p}_{l^{+}v^{+}t^{+}}\right)\ge0$ indicates $D\left(\bar{p}_{l^{-}v^{-}t^{-}}-\underline{p}_{l^{+}v^{+}t^{+}}\right)\ge0$. 

Then we investigate the sufficiency. Given $D\left(\bar{p}_{l^{+}v^{+}t^{+}}-\underline{p}_{l^{-}v^{-}t^{-}}\right)\ge0$
and $D\left(\bar{p}_{l^{-}v^{-}t^{-}}-\underline{p}_{l^{+}v^{+}t^{+}}\right)\ge0$,
we can first obtain that $l^{-}<l^{+}$ and $t^{-}<t^{+}$. Further
we know that there exists a point $\bar{t}^{*}\in[t^{-},t^{+}]$ such
that $\bar{p}_{l^{+}v^{+}t^{+}}(t)\ge\bar{p}_{l^{-}v^{-}t^{-}}(t),\forall t\in\left[-\infty,\bar{t}^{*}\right]$,
$\bar{p}_{l^{+}v^{+}t^{+}}\left(\bar{t}^{*}\right)=\bar{p}_{l^{-}v^{-}t^{-}}\left(\bar{t}^{*}\right)$
and $\bar{p}_{l^{+}v^{+}t^{+}}(t)\le\bar{p}_{l^{-}v^{-}t^{-}}(t),\forall t\in\left[\bar{t}^{*},\infty\right]$.
Then we can obtain a trajectory $p^{\mbox{m}}\in\mathcal{C}_{l^{-}v^{-}t^{-}}$
composed by $\bar{p}_{lvt^{-}}\left(t^{-}:\hat{t}^{\mbox{m}}\right)$,
$\mathbf{s}^{\mbox{m}}:=\left(\bar{p}_{lvt^{-}}\left(\hat{t}^{\mbox{m}}\right),\dot{\bar{p}}_{lvt^{-}}\left(\hat{t}^{\mbox{m}}\right),\underline{a},\hat{t}^{\mbox{m}},\bar{t}^{\mbox{m}}\right)$,
$\bar{p}_{l^{+}v^{+}t^{+}}\left(\bar{t}^{\mbox{m}}:\infty\right)$
satisfying $t^{-}\le\hat{t}^{\mbox{m}}\le\bar{t}^{\mbox{m}}<\infty$
and $D\left(\bar{p}_{l^{+}v^{+}t^{+}}-p^{\mbox{m}}\right)\ge0$. Next,
we will prove $\bar{t}^{\mbox{m}}\le t^{+}$ by contradiction. If
$\bar{t}^{\mbox{m}}>t^{+}$, then $\bar{p}_{l^{+}v^{+}t^{+}}\left(\bar{t}^{\mbox{m}}\right)>l^{+}$
and $\dot{\bar{p}}_{l^{+}v^{+}t^{+}}\left(\bar{t}^{\mbox{m}}\right)>v^{+}$.
Since segment $\mathbf{s}^{\mbox{m}}$ decelerates at $\underline{a}$,
then $D\left(\underline{p}_{l^{+}v^{+}t^{+}}-\mathbf{s}^{\mbox{m}}\right)>0$,
which however is contradictory to $D\left(\bar{p}_{l^{-}v^{-}t^{-}}-\underline{p}_{l^{+}v^{+}t^{+}}\right)\ge0$
because the start point of $\mathbf{s}^{\mbox{m}}$ is on $\bar{p}_{l^{-}v^{-}t^{-}}$.
This proves that $t^{-}\le\hat{t}^{\mbox{m}}\le\bar{t}^{\mbox{m}}\le t^{+}$.
Therefore, $p^{\mbox{m}}\in\mathcal{C}_{l^{-}v^{-}t^{-}}\cap\mathcal{C}_{l^{+}v^{+}t^{+}}=\mathcal{P}_{l^{-}v^{-}t^{-}}^{l^{+}v^{+}t^{+}}$.
This completes the proof.\end{proof}
\begin{prop}
\label{prop: shift_prism_feasibility}Given any $\delta\ge0$ and
two feasible state points $(l^{-},v^{-},t^{-})$ and $(l^{+},v^{+},t^{+})$
with $D(\bar{p}_{l^{-},v^{-},t^{-}}-\underline{p}_{l^{+},v^{+},t^{+}})\ge0$,
if quadratic prism $\mathcal{P}_{l^{-}v^{-}t^{-}}^{l^{+}v^{+}\left(t^{+}+\delta\right)}$
is not empty, then $\mathcal{P}_{l^{-}v^{-}t^{-}}^{l^{+}v^{+}t^{+}}$
is not empty and $D\left(\bar{p}_{l^{-}v^{-}t^{-}}^{l^{+}v^{+}t^{+}}-p\right)\ge0,\forall p\in\mathcal{P}_{l^{-}v^{-}t^{-}}^{l^{+}v^{+}\left(t^{+}+\delta\right)}$.\end{prop}
\begin{proof}
If $\mathcal{P}_{l^{-}v^{-}t^{-}}^{l^{+}v^{+}\left(t^{+}+\delta\right)}$
is not empty, Proposition \ref{prop:prism_feasibility} indicates
that $D\left(\bar{p}_{l^{+}v^{+}\left(t^{+}+\delta\right)}-\underline{p}_{l^{-}v^{-}t^{-}}\right)\ge0$.
Further, apparently $D\left(\bar{p}_{l^{+}v^{+}t^{+}}-\bar{p}_{l^{+}v^{+}\left(t^{+}+\delta\right)}\right)\ge0$
and thus due to the transitive property of function $D$, we obtain
$D\left(\bar{p}_{l^{+}v^{+}t^{+}}-\underline{p}_{l^{-}v^{-}t^{-}}\right)\ge0$,
which combined with the given condition $D\left(\bar{p}_{l^{-},v^{-},t^{-}}-\underline{p}_{l^{+},v^{+},t^{+}}\right)\ge0$
indicates that $\mathcal{P}_{l^{-}v^{-}t^{-}}^{l^{+}v^{+}t^{+}}$
is not empty based on Proposition \ref{prop:prism_feasibility}. 

Further, for any $p\in\mathcal{P}_{l^{-}v^{-}t^{-}}^{l^{+}v^{+}\left(t^{+}+\delta\right)}$,
we have $D\left(\bar{p}_{l^{-}v^{-}t^{-}}^{l^{+}v^{+}\left(t^{+}+\delta\right)}-p\right)>0$.
Since $D\left(\bar{p}_{l^{+}v^{+}t^{+}}-\bar{p}_{l^{+}v^{+}\left(t^{+}+\delta\right)}\right)\ge0$,
we also have $D\left(\bar{p}_{l^{+}v^{+}t^{+}}-p\right)\ge0$ due
to the transitive property of function $D$. This implies that $p\in\mathcal{P}_{l^{-}v^{-}t^{-}}^{l^{+}v^{+}t^{+}}$
and thus $D\left(\bar{p}_{l^{-}v^{-}t^{-}}^{l^{+}v^{+}t^{+}}-p\right)\ge0.$
This completes the proof.\end{proof}
\begin{rem}
Note that as $\bar{a}\rightarrow\infty$ and $\underline{a}\rightarrow-\infty$,
every smooth speed transition segment on the borders of a quartic
cone or prism reduces into a vertex, and the QTG concept converges
to the classic time geography \citep{miller2005measurement}. Besides,
when the spatiotemporal range of the studied problem is far greater
than that where acceleration $\bar{a}$ and deceleration $\underline{a}$
is discernible, neither is QTG much different from the classic time
geography.
\end{rem}

\subsection{Relationship Between QTG and SH }

As preparing for the investigation to the feasibility and optimality
of the SH solution, we now discuss the relationships between trajectories
generated from FSP and BSP and the borders of the corresponding quadratic
cone and prism. For uniformity, this subsection only considers FSP
and BSP with the extreme acceleration control variable values, i.e.,
$\left[\bar{a}^{\mbox{f}},\underline{a}^{\mbox{f}},\bar{a}^{\mbox{b}},\underline{a}^{\mbox{b}}\right]=\left[\bar{a},\underline{a},\bar{a},\underline{a}\right]$.

\begin{prop}
\label{prop: FSP-p_bar}The forward shooting trajectory $p^{\mbox{f}}\left(\left(l,v,t^{-}\right),\emptyset\right)$
generated from FSP overlaps\textup{ $\bar{p}_{lvt^{-}}(t^{-}:\infty)$. }
\end{prop}

\begin{prop}
\label{prop: BSP-p_bar}Given two feasible state points $\left(l^{-},v^{-},t^{-}\right)$
and $\left(l^{+},v^{+},t^{+}\right)$ with $l^{+}>l^{-}$ and $t^{+}>t^{-}$
such that quadratic prism $\mathcal{P}_{l^{-}v^{-}t^{-}}^{l^{+}v^{+}t^{+}}\neq\emptyset$,
the extended backward shooting trajectory $p^{\mbox{eb}}((l^{+},v^{+},t^{+}),p^{\mbox{f}}((l^{-},v^{-},t^{-}),\emptyset))$
generated from BSP overlaps \textup{$\bar{p}_{l^{-}v^{-}t^{-}}^{l^{+}v^{+}t^{+}}(t^{-}:\infty).$}
\end{prop}
These two propositions obviously hold based on the definitions of
FSP and BSP and thus we omit the proofs. These properties can be extended
to cases where the current trajectory is bounded by one or multiple
preceding trajectories from the top. 
\begin{defn}
Given a set of trajectories $\mathbf{q}=\left\{ q_{1},q_{2},\cdots,q_{M}\in\mathcal{T}\right\} $,
we define $u(\mathbf{q},t):=\min_{m\in\left\{ 1,\cdots,M\right\} }q_{m}(t),\forall t$
and we call function $u(\mathbf{q},\cdot)$ a \emph{quasi-trajectory}
and denote it with $u(\mathbf{q})$ for simplicity. Let $\mathcal{U}$
denote the set of all quasi-trajectories. Note that distance function
$D$ can be easily extended to $\mathcal{U}$, i.e., $D\left(u-u'\right):=\min_{t\in(\infty,\infty)}\left(u(t)-u'(t)\right),\forall u,u'\in\mathcal{U}$.
We can also denote EFSP result $p^{\mbox{f}}\left(\left(l,v,t\right),\mathbf{q}\right)$
as $p^{\mbox{f}}\left(\left(l,v,t\right),u(\mathbf{q})\right)$. Further,
let $p(\mathbf{q})$ denote the trajectory generated by merging all
trajectories in $\mathbf{q}$ with EFSO, and we can also denote BSP
result $p^{\mbox{b}}\left(\left(l,v,t\right),p(\mathbf{q})\right)$
and $p^{\mbox{eb}}\left(\left(l,v,t\right),p(\mathbf{q})\right)$
as $p^{\mbox{b}}\left(\left(l,v,t\right),u(\mathbf{q})\right)$ and
$p^{\mbox{eb}}\left(\left(l,v,t\right),u(\mathbf{q})\right)$, respectively,
for simplicity.
\end{defn}

\begin{defn}
Given a feasible state point $(l,v,t^{-})$ and a quasi-trajectory
$u\in\mathcal{U}$ , we define $\mathcal{C}_{lvt^{-}}^{u}:=\left\{ p\left|p\in\mathcal{C}_{lvt^{-}},D\left(u-p\right)\ge0\right.\right\} $,
which we call a \emph{bounded cone} of $(l,v,t^{-})$ with respect
to $u$. 
\end{defn}

\begin{defn}
Given two feasible state points $(l^{-},v^{-},t^{-})$, $(l^{+},v^{+},t^{+})$
and a quasi-trajectory $u\in\mathcal{U}$, we define $\mathcal{P}_{l^{-}v^{-}t^{-}}^{l^{+}v^{+}t^{+},u}:=\left\{ p\left|p\in\mathcal{P}_{l^{-}v^{-}t^{-}}^{l^{+}v^{+}t^{+}},D\left(u-p\right)\ge0\right.\right\} $,
which we call a \emph{bounded prism} of $(l^{-},v^{-},t^{-})$ and
$(l^{+},v^{+},t^{+})$ with respect to $u$.
\end{defn}
Apparently, bounded cones and prisms shall satisfy the following properties.
\begin{prop}
\label{prop:bounded_transitive}Given any two feasible state points
$(l^{-},v^{-},t^{-})$, $(l^{+},v^{+},t^{+})$ and two quasi-trajectories
$u,u'\in\mathcal{T}$ with $D(u'-u)\ge0$, then $\mathcal{C}_{l^{-}v^{-}t^{-}}^{u}\subseteq\mathcal{C}_{l^{-}v^{-}t^{-}}^{u'}$
and $\mathcal{P}_{l^{-}v^{-}t^{-}}^{l^{+}v^{+}t^{+},u}\subseteq\mathcal{P}_{l^{-}v^{-}t^{-}}^{l^{+}v^{+}t^{+},u'}$.
\end{prop}
Then we will prove less intuitive properties for bounded cones and
prisms. 
\begin{prop}
\label{prop: FSP-feasibility}Given a feasible state point $(l,v,t^{-})$
and a trajectory $p'\in\mathcal{T}$, then trajectory $p^{\mbox{f}}((l,v,t^{-}),p')$
obtained from FSP is not empty $\Leftrightarrow$ $\mathcal{C}_{lvt^{-}}^{p'}\neq\emptyset$
$\Leftrightarrow$ $D\left(p'-\underline{p}_{lvt^{-}}\right)\ge0$. \end{prop}
\begin{proof}
It is apparent that $p^{\mbox{f}}:=p^{\mbox{f}}((l,v,t^{-}),p')\neq\emptyset$
leads to $\mathcal{C}_{lvt^{-}}^{p'}\neq\emptyset$ because $p^{\mbox{f}}\in\mathcal{C}_{lvt^{-}}^{p'}$.
Further if we can find a feasible trajectory in $p\in\mathcal{C}_{lvt^{-}}^{p'}$,
we know that $D\left(p'-p\right)\ge0$ and $D\left(p-\underline{p}_{lvt^{-}}\right)\ge0$,
which yields $D\left(p'-\underline{p}_{lvt^{-}}\right)\ge0$ since
$D$ is transitive. Now we only need to prove that $D\left(p'-\underline{p}_{lvt^{-}}\right)\ge0$
leads to $p^{\mbox{f}}((l,v,t^{-}),p')\neq\emptyset$. Now we are
given $p'(t)\ge\underline{p}_{lvt^{-}}(t),\forall t\in[t^{-},\infty)$.
If $p'(t)\ge\bar{p}_{lvt^{-}}(t),\forall t\in[t^{-},\infty)$, then
Proposition \ref{prop: FSP-p_bar} indicates $p^{\mbox{f}}((l,v,t^{-}),p')=p^{\mbox{f}}((l,v,t^{-}),\emptyset)$
that overlaps $\bar{p}_{lvt^{-}}(t^{-}:\emptyset)$. Therefore $p^{\mbox{f}}((l,v,t^{-}),p')$
should be always non-empty. Otherwise, it should be that $p'(t)\ge\bar{p}_{lvt^{-}}(t),\forall t\in[t^{-},t^{*})$,
$p'(t^{*})=\bar{p}_{lvt^{-}}(t^{*})$ and $p'(t)<\bar{p}_{lvt^{-}}(t),\forall t\in\left(t^{*},\infty\right)$
for some $t^{*}\in[t^{-},\infty)$. We first define a continuous function
of time $\hat{t}\in[t^{-},\infty)$ as follows. We construct a trajectory
denoted by $\underline{p}_{\hat{t}}$ composed by maximally accelerating
section $\bar{p}_{lvt^{-}}(t^{-}:\hat{t})$ and a maximally decelerating
section $\underline{p}_{\underline{p}_{lvt^{-}}(\hat{t})\dot{\underline{p}}_{lvt^{-}}(\hat{t})t^{-}}(t^{-}:\infty)$.
Then we define function 
\[
d(\hat{t}):=D\left(p'-\underline{p}_{\hat{t}}\right)=\min_{t\in[t^{-},\infty)}p'(t)-\underline{p}_{\hat{t}}(t).
\]
Note that as $\hat{t}$ increases continuously, $\underline{p}_{\hat{t}}(t)$
increases continuously at every $t\in[t^{-},\infty)$. Then we can
see that function $d(\hat{t})$ shall continuously decrease with $\hat{t}$.
Note that $\underline{p}_{\hat{t}}$ is identical to $\underline{p}_{lvt^{-}}$
when $\hat{t}=0$. Then since $p'(t)\ge\underline{p}_{lvt^{-}}(t),\forall t\in[t^{-},\infty)$,
we obtain $d(0)\ge0$. Further, as $\hat{t}$ increases to $t^{*}$,
then $\underline{p}_{\hat{t}}(t)$ and $p'(t)$ shall intersection
at $t^{*}$, which indicates that $d(t^{*})\le0$. Due to Bolzano's
Theorem \citep{apostol1969calculus}, we can always find a $\hat{t}^{\mbox{m}}\in[t^{-},t^{*}]$
such that $d(\hat{t}^{\mbox{m}})=0$, which indicates that $p'$ and
$\underline{p}_{\hat{t}^{\mbox{m}}}$ get tangent at a time $\bar{t}{}^{\mbox{m}}\in[\hat{t}^{\mbox{m}},\infty).$
Then a trajectory $p^{\mbox{f}}$ can be obtained by concatenating
$\bar{p}_{lvt^{-}}\left(t^{-}:\hat{t}^{\mbox{m}}\right)$, $\left(\bar{p}_{lvt^{-}}\left(\hat{t}^{\mbox{m}}\right),\dot{\bar{p}}_{lvt^{-}}\left(\hat{t}^{\mbox{m}}\right),\underline{a},\hat{t}^{\mbox{m}},\bar{t}^{\mbox{m}}\right)$,
$p'\left(\bar{t}^{\mbox{m}}:\infty\right)$. Note that $p^{\mbox{f}}$
is exactly $p^{\mbox{f}}((l,v,t^{-}),p')$ and thus $p^{\mbox{f}}((l,v,t^{-}),p')$
is not empty. This completes the proof.
\end{proof}
We can further extend this result to a quasi-trajectory as a upper
bound.
\begin{cor}
\label{cor:FSP-Feasibility}Given a feasible state point $(l,v,t^{-})$
and a quasi-trajectory $u\in\mathcal{U}$, then trajectory $p^{\mbox{f}}((l,v,t^{-}),u)$
obtained from the EFSP is not empty $\Leftrightarrow$ $\mathcal{C}_{lvt^{-}}^{u}\neq\emptyset$
$\Leftrightarrow$ $D\left(u-\underline{p}_{lvt^{-}}\right)\ge0$. \end{cor}
\begin{prop}
\label{prop: FSP-shift-bound}Given a feasible state point $(l,v,t^{-})$,
a quasi-trajectory $u\in\mathcal{U}$ and scalars \textup{$\delta,\delta'\ge0$,}
if $\mathcal{C}_{lvt}^{u}\neq\emptyset$, then $\mathcal{C}_{(l-\delta)v\left(t^{-}+\delta'\right)}^{u}\neq\emptyset$
and we obtain $D\left(p^{\mbox{f}}((l,v,t^{-}),u)-p\right)\ge0,\forall p\in\mathcal{C}_{(l-\delta)v\left(t^{-}+\delta'\right)}^{u}$.\end{prop}
\begin{proof}
We first prove if $\mathcal{C}_{lvt}^{u}\neq\emptyset$, then $\mathcal{C}_{(l-\delta)v\left(t^{-}+\delta'\right)}^{u}\neq\emptyset$.
Due to Corollary \ref{cor:FSP-Feasibility}, we know $D\left(u-\underline{p}_{lvt^{-}}\right)\ge0$.
Since apparently $D\left(\underline{p}_{lvt^{-}}-\underline{p}_{(l-\delta)v\left(t^{-}+\delta'\right)}\right)\ge0$,
we obtain $D\left(u-\underline{p}_{(l-\delta)v\left(t^{-}+\delta'\right)}\right)\ge0$
due to the transitive property of operator $D$. This yields $\mathcal{C}_{(l-\delta)v\left(t^{-}+\delta'\right)}^{u}\neq\emptyset$
based on Corollary \ref{cor:FSP-Feasibility}.

The we prove the remaining part of this proposition. In case that
$D\left(u-\bar{p}_{lvt^{-}}\right)\ge0$, $p^{\mbox{f}}:=p^{\mbox{f}}((l,v,t^{-}),u)$
shall be identical to $\bar{p}_{lvt^{-}}$ and thus $D\left(p^{\mbox{f}}-p\right)\ge0$
obviously holds $\forall p\in\mathcal{C}_{lvt^{-}}^{u}=\mathcal{C}_{lvt^{-}}.$
Otherwise, then we know that $p^{\mbox{f}}=$$\left[\bar{p}_{lvt^{-}}\left(t^{-}:t^{\mbox{m}}\right),\mathbf{s}^{\mbox{m}}:=\left(\bar{p}_{lvt^{-}}\left(\hat{t}^{\mbox{m}}\right),\dot{\bar{p}}_{lvt^{-}}\left(\hat{t}^{\mbox{m}}\right),\underline{a},\hat{t}^{\mbox{m}},\bar{t}^{\mbox{m}}\right),p'\left(t^{\mbox{+}}:\infty\right)\right]$
for some $t^{\mbox{m}}$ and $t^{+}$ satisfying $t^{-}\le t^{\mbox{m}}\le t^{\mbox{+}}\le\infty$,
where $p'$ is the lower-bound trajectory merged by all elements in
$u$ with EFSO. If $\exists p\in\mathcal{C}_{lvt^{-}}^{u}$ such that
$D\left(p^{\mbox{f}}-p\right)<0$, then there much exist a $\tilde{t}\in(t^{\mbox{m}},t^{\mbox{+}})$
such that $p\left(\tilde{t}\right)$ is strictly above $\mathbf{s}^{\mbox{m}}$.
Since $D$$\left(\bar{p}_{lvt^{-}}-p\right)\ge0$ and $D$$\left(p-\underline{p}_{p\left(\tilde{t}\right)\dot{p}\left(\tilde{t}\right)\tilde{t}}\right)\ge0$,
thus we have $D$$\left(\bar{p}_{lvt^{-}}-\underline{p}_{p\left(\tilde{t}\right)\dot{p}\left(\tilde{t}\right)\tilde{t}}\right)\ge0$,
which indicates $\underline{p}_{p\left(\tilde{t}\right)\dot{p}\left(\tilde{t}\right)\tilde{t}}$
and $\mathbf{s}^{\mbox{m}}$ have to intersect at a time $t'\in\left[\hat{t}^{\mbox{m}},\tilde{t}\right)$.
However, sine $D\left(u-p\right)\ge0$, thus $\underline{p}_{p\left(\tilde{t}\right)\dot{p}\left(\tilde{t}\right)\tilde{t}}$
needs to intersect with $\mathbf{s}^{\mbox{m}}$ at another time $t"\in\left(\tilde{t},\bar{t}^{\mbox{m}}\right]$.
This is contradictory to the fact that $\mathbf{s}^{\mbox{m}}$ has
already decelerated at the extreme deceleration rate $\underline{a}$.
This contradiction proves that $D\left(p^{\mbox{f}}-p\right)\ge0,\forall p\in\mathcal{C}_{lvt^{-}}^{p'}$.
Further, given $\delta,\delta'\ge0$, for any $p\in\mathcal{C}_{(l-\delta)v\left(t^{-}+\delta'\right)}^{u},$we
can find a $p'\in\mathcal{C}_{lvt^{-}}^{p'}$satisfying $D(p'-p)\ge0$
with a similar argument. This proves that $D\left(p^{\mbox{f}}-p'\right)\ge0,\forall p'\in\mathcal{C}_{(l-\delta)v\left(t^{-}+\delta'\right)}^{u},\delta,\delta'\ge0.$
\end{proof}
Symmetrically, we can obtain similar properties for BSP as well.
\begin{cor}
\label{cor:BSP-Feasibility}Given a feasible state point $(l,v,t^{+})$
and a quasi-trajectory $u\in\mathcal{U}$,then trajectory $p^{\mbox{eb}}((l,v,t^{+}),u)$
obtained from BSP is not empty $\Leftrightarrow$ $\mathcal{C}_{lvt^{+}}^{u}\neq\emptyset$
$\Leftrightarrow$ $D\left(u-\underline{p}_{pvt^{+}}\right)\ge0$.
If $\mathcal{C}_{lvt^{+}}^{u}\neq\emptyset$, we obtain $\mathcal{C}_{(l-\delta)v\left(t^{+}+\delta'\right)}^{u}\neq\emptyset$
and $D\left(p^{\mbox{eb}}-p\right)\ge0,\forall p\in\mathcal{C}_{(l-\delta)v\left(t^{+}+\delta'\right)}^{u}.$
\end{cor}
Combining Corollaries \ref{cor:FSP-Feasibility}and \ref{cor:BSP-Feasibility}
leads to the following property with respect to a quadratic prism. 
\begin{cor}
\label{cor:BSP-Feasibility-Prism}Given two feasible state point $(l^{-},v^{-},t^{-})$,
$(l^{+},v^{+},t^{+})$ satisfying $D\left(\bar{p}_{l^{-}v^{-}t^{-}}-\underline{p}_{l^{+}v^{+}t^{+}}\right)\ge0$
and $D\left(\bar{p}_{l^{+}v^{+}t^{+}}-\underline{p}_{l^{-}v^{-}t^{-}}\right)\ge0$,
and a quasi-trajectory $u\in\mathcal{U}$, then $p^{\mbox{f}}((l^{-},v^{-},t^{-}),u)\neq\emptyset$
and $p^{\mbox{eb}}((l^{+},v^{+},t^{+}),u)\neq\emptyset$ $\Leftrightarrow$
$\mathcal{P}_{l^{-}v^{-}t^{-}}^{l^{+}v^{+}t^{+},u}\neq\emptyset$
$\Leftrightarrow$ $D(u-\underline{p}_{l^{-}v^{-}t^{-}})\ge0$ and
$D(u-\underline{p}_{l^{+}v^{+}t^{+}})\ge0$. Whenever $\mathcal{P}_{l^{-}v^{-}t^{-}}^{l^{+}v^{+}t^{+},u}\neq\emptyset$,
we obtain $D\left(p^{\mbox{eb}}-p\right)\ge0,$ \textup{$\forall p\in\mathcal{P}_{\left(l^{-}-\epsilon\right)v^{-}\left(t^{-}+\epsilon'\right)}^{\left(l^{+}-\delta\right)v^{+}\left(t^{+}+\delta'\right),u},\delta,\delta',\epsilon,\epsilon'\ge0$.}
\end{cor}

\subsection{Feasibility Properties\label{sub:Feasibility-Properties}}

Although the proposed heuristic algorithms are essentially heuristics
that may not explore the entire feasible region of the original problem,
it can be used as a touchstone for the feasibility of the original
problems under certain mild conditions. This section will discuss
the relationship between the feasibility of the SH solutions and that
of the original problems. Note that with different acceleration values,
the SH algorithms yields different solutions. For uniformity, this
section only investigates the representative SH algorithms with $\bar{a}^{\mbox{f}}=\bar{a}^{\mbox{b}}=\bar{a}$
and $\underline{a}^{\mbox{f}}=\underline{a}^{\mbox{b}}=\underline{a}$.

The following analysis investigates LVP, i.e., the feasibility of
$\mathcal{P}^{\mbox{LVP}}$ and that of $P^{\mbox{PSH}}\left(\underline{a},\bar{a}\right)$
.
\begin{defn}
Entry boundary condition $\left[v_{n}^{-},t_{n}^{-}\right]_{n\in\mathcal{N}}$
is \emph{proper} if $v_{n}^{-}\in\left[0,\bar{v}\right],\forall n\in\mathcal{N}$
and 
\[
D\left(\bar{p}_{(ms-ns)v_{m}^{-}(t_{m}^{-}+n\tau-m\tau)}-\underline{p}_{0v_{n}^{-}t_{n}^{-}}\right)\ge0,\forall m<n\in\mathcal{N}.
\]
\end{defn}
\begin{prop}
\label{prop:proper_condition}$\mathcal{P}\neq\emptyset$ $\Rightarrow$$\left[v_{n}^{-},t_{n}^{-}\right]_{n\in\mathcal{N}}$
is proper.\end{prop}
\begin{proof}
When $\mathcal{P}\neq\emptyset,$ let $P=\left[p_{n}\right]_{n\in\mathcal{N}}$
denote a generic feasible trajectory vector in $\mathcal{P}$. For
any given $m<n\in\mathcal{N}$, based on safety constraint \eqref{eq: set-safety-constraints},
we obtain $D\left(p_{m}^{s}-p_{m+1}\right)\ge0,$ $D\left(p_{m+1}^{s}-p_{m+2}\right)\ge0,\cdots,$
$D\left(p_{n-1}^{s}-p_{n}\right)\ge0$. This can be translated as
$D\left(p_{m}^{s^{n-m}}-p_{m+1}^{s^{n-m-1}}\right)\ge0,$ $D\left(p_{m+1}^{s^{n-m-1}}-p_{m+2}^{s^{n-m-2}}\right)\ge0,\cdots,$
$D\left(p_{n-1}^{s}-p_{n}\right)\ge0$, which indicates that $D\left(p_{m}^{s^{n-m}}-p_{n}\right)\ge0$
due to the transitive property of function $D\left(\cdot\right)$.
Further, since $D\left(\bar{p}_{(ms-ns)v_{m}^{-}\left[t_{m}^{-}+(n-m)\tau\right]}-p_{m}^{s^{n-m}}\right)\ge0$
and $D\left(p_{n}-\underline{p}_{0v_{n}^{-}t_{n}^{-}}\right)\ge0$
, we obtain $D\left(\bar{p}_{(ms-ns)v_{m}^{-}\left[t_{m}^{-}+(n-m)\tau\right]}-\underline{p}_{0v_{n}^{-}t_{n}^{-}}\right)\ge0,\forall m<n\in\mathcal{N}$.
This completes the proof.\end{proof}
\begin{thm}
\label{theo: feasibility_all_green} $\mathcal{P}^{\mbox{LVP}}\neq\emptyset$
$\Leftrightarrow$ \textup{$P^{\mbox{PSH}}\left(\underline{a},\bar{a}\right)\neq\emptyset$}
$\Leftrightarrow$ $\left[v_{n}^{-},t_{n}^{-}\right]_{n\in\mathcal{N}}$
is proper. \end{thm}
\begin{proof}
When $\left[v_{n}^{-},t_{n}^{-}\right]_{n\in\mathcal{N}}$ is proper,
based on Corollary \ref{cor:BSP-Feasibility}, we know that every
EFSP step will generate a feasible trajectory. Thus $P^{\mbox{PSH}}\left(\underline{a},\bar{a}\right)$
is feasible. Since $P^{\mbox{PSH}}\left(\underline{a},\bar{a}\right)\in\mathcal{P}^{\mbox{LVP}}$,
we see that $\mathcal{P}^{\mbox{LVP}}$ is feasible, too. Further,
Proposition \ref{prop:proper_condition} indicates that when $\mathcal{P}^{\mbox{LVP}}$
is feasible, which means $\mathcal{P}$ too is feasible, $\left[v_{n}^{-},t_{n}^{-}\right]_{n\in\mathcal{N}}$
is proper. This completes the proof.
\end{proof}
Now we add back the signal control and investigate the feasibility
of $\mathcal{P}$ under milder conditions with the SH solution. We
consider a special subset of $\mathcal{P}$ where every trajectory
has the maximum speed of $\bar{v}$ at the exit location $L$, i.e,
\[
\hat{\mathcal{P}}:=\left\{ \left[p_{n}\right]{}_{n\in\mathcal{N}}\in\mathcal{P}\left|\dot{p}_{n}\left(p_{n}^{-1}\left(L\right)\right)=\bar{v},\forall n\in\mathcal{N}\right.\right\} .
\]
This subset is not too restrictive, because in order to assure a high
traffic throughput rate, the exist speed of each vehicle should be
high. The following analysis investigates the relationship between
$\hat{\mathcal{P}}$ and $P^{\mbox{SH}}\left(\underline{a},\bar{a},\underline{a},\bar{a}\right)$.
\begin{prop}
\label{prop: exit_maximum_speed}When $L\ge\bar{v}^{2}/(2\bar{a})$,
if $P^{\mbox{SH}}\left(\underline{a},\bar{a},\underline{a},\bar{a}\right)=\left[p_{n}\right]{}_{n\in\mathcal{N}}$
is feasible, then $\dot{p}_{n}\left(t\right)=\bar{v},\forall t\ge p_{n}^{-1}(L),n\in\mathcal{N}$
.\end{prop}
\begin{proof}
We use induction to prove this proposition. The induction assumption
is $\dot{,p}_{n}\left(t\right)=\bar{v},\forall t\ge p_{n}^{-1}(L)$.
If the BSP is not needed, Proposition \ref{prop: FSP-p_bar} indicates
that $p_{1}=p_{1}^{\mbox{f}}=\bar{p}_{0v_{1}^{-}t_{1}^{-}}\left(t_{1}^{-}:\infty\right)$
and thus since $L\ge\bar{v}^{2}/(2\bar{a})$, we obtain $\dot{p}_{1}\left(t\right)=\dot{\bar{p}}_{0v_{1}^{-}t_{1}^{-}}\left(t\right)=\bar{v},\forall t\ge p_{1}^{-1}(L)=\bar{p}_{0v_{1}^{-}t_{1}^{-}}^{-1}(L)$.
Otherwise, if BSP is activated, it shoots backwards from the state
point $\left(L,\bar{p}_{0v_{1}^{-}t_{1}^{-}}\left(\bar{p}_{0v_{1}^{-}t_{1}^{-}}^{-1}(L)\right)=\bar{v},G\left(\bar{p}_{0v_{1}^{-}t_{1}^{-}}^{-1}(L)\right)\right)$,
and thus Proposition \ref{prop: BSP-p_bar} indicates that $p_{1}=\bar{p}_{0v_{1}^{-}t_{1}^{-}}^{L\bar{v}G\left(\bar{p}_{0v_{1}^{-}t_{1}^{-}}^{-1}(L)\right)}$.
Therefore, $\dot{p}_{1}\left(t\right)=\bar{v},\forall t\ge p_{1}^{-1}(L)$
holds again since $L\ge\bar{v}^{2}/(2\bar{a})$. Then assume that
this assumption holds for $n=k-1$, and we will investigate whether
it holds for $n=k,\forall k\in\mathcal{N}\backslash\{1\}.$ If the
primary FSP is not blocked by $p_{k-1}^{\mbox{s}}$, apparently $p_{k}=\bar{p}_{0v_{k}^{-}t_{k}^{-}}$
if the BSP is not needed, or $p_{k}=\bar{p}_{0v_{k}^{-}t_{k}^{-}}^{L\bar{v}G\left(\bar{p}_{0v_{k}^{-}t_{k}^{-}}^{-1}(L)\right)}$
if the BSP is activated. Either way $\dot{p}_{k}\left(t\right)=\bar{v},\forall t\ge p_{k}^{-1}(L)$
holds since $L\ge\bar{v}^{2}/(2\bar{a})$. Otherwise, $p_{k}^{\mbox{f}}$
merges into $p_{k-1}^{\mbox{s}}$ before reaching $L$, and thus based
on the induction assumption we obtain $\dot{p}_{k}^{\mbox{f}}\left(p_{k}^{\mbox{f}-1}(L)\right)=\dot{p}_{k-1}^{\mbox{s}}\left(p_{k-1}^{\mbox{s}-1}(L)\right)=\dot{p}_{k-1}\left(p_{k-1}^{-1}(L+s)\right)=\bar{v}.$
Note that again the BSP could only shift segments in parallel and
thus $\dot{p}_{k}\left(p_{k}^{-1}(L)\right)=\dot{p}_{k}^{\mbox{f}}\left(p_{k}^{\mbox{f}-1}(L)\right)=\bar{v}$.
The induction assumption also indicates that the segments of $p_{k-1}$
used in the auxiliary FSP for $p_{k}$ is at constant speed $\bar{v}$.
This means that the auxiliary FSP for $p_{k}$ will not be blocked
by $p_{k-1}$, and thus we obtain $\dot{p}_{k}\left(t\right)=\bar{v},\forall t\ge p_{k}^{-1}(L)$.
This completes the proof.\end{proof}
\begin{thm}
When $L\ge\bar{v}^{2}/(2\bar{a})$ , $\hat{\mathcal{P}}\neq\emptyset$
$\Leftrightarrow$ $P^{\mbox{SH}}\left(\underline{a},\bar{a},\underline{a},\bar{a}\right)\neq\emptyset$\textup{
}.\label{theo:P_hat_feasibility}\end{thm}
\begin{proof}
The proof of the sufficiency is simple. Again, we can write $P^{\mbox{SH}}\left(\underline{a},\bar{a},\underline{a},\bar{a}\right)$
as $\left[p_{n}\right]{}_{n\in\mathcal{N}}$. When $P^{\mbox{SH}}\left(\underline{a},\bar{a},\underline{a},\bar{a}\right)\neq\emptyset$,
Proposition \ref{prop: exit_maximum_speed} indicates that $P\in\hat{\mathcal{P}}$
and thus $\hat{\mathcal{P}}\neq\emptyset.$ 

Then we only need to prove the necessity. When there exists $P'=\left[p'_{n}\right]{}_{n\in\mathcal{N}}\in\hat{\mathcal{P}}$
that is not empty, we will show that $\left[p_{n}\right]{}_{n\in\mathcal{N}}$
too is not empty with the following induction. For the notation convenience,
we denote $t{}_{n}^{'+}=\left(p'_{n}\right)^{-1}(L)$ and $t{}_{n}^{+}:=p{}_{n}^{-1}(L)$.
The induction assumption is that $p_{n}$ exists and $D\left(p_{n}-p'_{n}\right)\ge0$.
When $n=1,$ if BSP is not activated in constructing $p_{1}$, then
$p_{1}=p^{\mbox{f}}\left(\left(0,v_{1}^{-},t_{1}^{-}\right),\emptyset\right)=\bar{p}_{0v_{1}^{-}t_{1}^{-}}(t^{-}:\infty)$
based on Proposition \ref{prop: FSP-p_bar}. Therefore, $p_{1}$ exists
and $D\left(p_{1}-p'_{1}\right)\ge0$. Otherwise, denote $\bar{t}_{1}^{+}=\bar{p}_{0v_{1}^{-}t_{1}^{-}}^{-1}(L)$
and $t_{1}^{+}:=G\left(\bar{t}_{1}^{+}\right)$. Note that since $L\ge\bar{v}^{2}/(2\bar{a})$,
$\dot{\bar{p}}_{0v_{1}^{-}t_{1}^{-}}\left(\bar{p}_{0v_{1}^{-}t_{1}^{-}}^{-1}(L)\right)=\bar{v}$.
Apparently, $D\left(\bar{p}_{0v_{1}^{-}t_{1}^{-}}-\underline{p}_{L\bar{v}\bar{t}_{1}^{+}}\right)\ge0$
and $D\left(\underline{p}_{L\bar{v}\bar{t}_{1}^{+}}-\underline{p}_{L\bar{v}t_{1}^{+}}\right)\ge0$,
which indicates $D\left(\bar{p}_{0v_{1}^{-}t_{1}^{-}}-\underline{p}_{L\bar{v}t_{1}^{+}}\right)\ge0$.
Further since $L\ge\bar{v}^{2}/(2\bar{a})$, we know $D\left(\bar{p}_{L\bar{v}t_{1}^{+}}-\underline{p}_{0v_{1}^{-}t_{1}^{-}}\right)\ge0$.
Then we obtain $p_{1}=\bar{p}_{0v_{1}^{-}t_{1}^{-}}^{L\bar{v}t_{1}^{+}}$
based on Proposition \ref{prop: BSP-p_bar}. Further, apparently $t{}_{1}^{'+}\ge\bar{t}_{1}^{+}$,
and then $t{}_{1}^{'+}=G\left(t{}_{1}^{'+}\right)\ge G\left(\bar{t}_{1}^{+}\right)=t_{1}^{+}$
since function $G(\cdot)$ is increasing. Then Proposition \ref{prop: shift_prism_feasibility}
indicates that $D\left(p_{1}-p'_{1}\right)\ge0$. 

Then assume that the induction assumption holds for $n=k-1,\forall k\in\mathcal{N}\backslash\{1\}.$
Then for $n=k$, Since $p'\in\mathcal{P}_{0v_{k}^{-}t_{k}^{-}}^{L\bar{v}t_{k}^{'+},p_{k-1}^{\mbox{'s}}}$,
then we obtain from Corollary \ref{cor:BSP-Feasibility-Prism} that
$D\left(p_{k-1}^{\mbox{'s}}-\underline{p}_{0v_{k}^{-}t_{k}^{-}}\right)\ge0$
where $p_{k-1}^{\mbox{'s}}$ is the shadow trajectory of $p'_{k-1}$.
And the induction assumption tells that $D\left(p_{k-1}^{\mbox{s}}-p_{k-1}^{\mbox{'s}}\right)\ge0$
where $p_{k-1}^{\mbox{s}}$ is the shadow trajectory of $p_{k-1}$,
which further indicates that $D\left(p_{k-1}^{\mbox{s}}-\underline{p}_{0v_{k}^{-}t_{k}^{-}}\right)\ge0$
based on the transitive property of $D$. Then Proposition \ref{prop: FSP-feasibility}
indicate that $p^{\mbox{f}}\left((0,v_{k}^{-},t_{k}^{-}),p_{k-1}^{\mbox{s}}\right)\neq\emptyset$.
For the simplicity of presentation, we denote $p^{\mbox{f}}\left((0,v_{k}^{-},t_{k}^{-}),p_{k-1}^{\mbox{s}}\right)$
by $p_{k}^{\mbox{f}}$. If BSP is not used in constructing $p_{k}$,
then $p_{k}=p_{k}^{\mbox{f}}$ exists and Proposition \ref{prop: FSP-shift-bound}
indicates $D\left(p_{k}-p'_{k}\right)\ge0$. Otherwise, let $\bar{t}_{k}^{+}=p_{k}^{\mbox{f}-1}\left(L\right)$
and $t_{k}^{+}=G\left(\bar{t}_{k}^{+}\right)$. Then if $p_{k}$ exists,
then $p_{k}^{-1}(L)=t_{k}^{+}$ and $\dot{p}_{k}\left(t_{k}^{+}\right)=\bar{v}$
since $L\ge\bar{v}^{2}/(2\bar{a})$. Again, it is easy to see that
$D\left(\bar{p}_{0v_{k}^{-}t_{k}^{-}}-\underline{p}_{L\bar{v}t_{k}^{+}}\right)\ge0$
and $D\left(\bar{p}_{L\bar{v}t_{k}^{+}}-\underline{p}_{0v_{k}^{-}t_{k}^{-}}\right)\ge0$,
and therefore Corollary \ref{cor:BSP-Feasibility-Prism} indicates
that $p_{k}$ exists. Also, apparently $t{}_{k}^{'+}\ge\bar{t}_{k}^{+}$,
and then $t{}_{k}^{'+}=G\left(t{}_{k}^{'+}\right)\ge G\left(\bar{t}_{k}^{+}\right)=t_{k}^{+}$
since function $G(\cdot)$ is increasing. Note that $p_{k}=p^{\mbox{eb}}\left((L,\bar{v},t_{k}^{+}),p_{k}^{\mbox{f}}\right)$
from BSP with $\bar{a}^{\mbox{f}}=\bar{a}^{\mbox{b}}=\bar{a}$ and
$\underline{a}^{\mbox{f}}=\underline{a}^{\mbox{b}}=\underline{a}$.
Then Corollary \ref{cor:BSP-Feasibility-Prism} further indicates
that $D\left(p_{k}-p'_{k}\right)\ge0$. This completes the proof.
\end{proof}
Further, we will show that when $L$ is sufficiently long, the feasibility
of $P^{\mbox{SH}}\left(\underline{a},\bar{a},\underline{a},\bar{a}\right)$
is equivalent to the feasibility of $\mathcal{P}$. 
\begin{thm}
When $L\ge\frac{\bar{v}^{2}}{2\bar{a}}+\frac{\bar{v}^{2}}{-2\underline{a}^{\mbox{b}}}+\frac{\bar{v}^{2}}{-2\underline{a}^{\mbox{b}}}+s(N-1)$
, $P^{\mbox{SH}}\left(\underline{a},\bar{a},\underline{a},\bar{a}\right)\neq\emptyset$
$\Leftrightarrow$ $\mathcal{P}\neq\emptyset$ $\Leftrightarrow$
$\left[v_{n}^{-},t_{n}^{-}\right]_{n\in\mathcal{N}}$ is proper.\label{theo:P_feasibility}\end{thm}
\begin{proof}
When $P^{\mbox{SH}}\left(\underline{a},\bar{a},\underline{a},\bar{a}\right)\neq\emptyset$,
it is trivial to see that $\mathcal{P}\neq\emptyset$ since $P^{\mbox{SH}}\left(\underline{a},\bar{a},\underline{a},\bar{a}\right)\in\mathcal{P}$.
Proposition \ref{prop:proper_condition} indicates that $\mathcal{P}\neq\emptyset$
leads to that $\left[v_{n}^{-},t_{n}^{-}\right]_{n\in\mathcal{N}}$
is proper. Then we only need to prove that when $\left[v_{n}^{-},t_{n}^{-}\right]_{n\in\mathcal{N}}$
is proper, $P^{\mbox{SH}}\left(\underline{a},\bar{a},\underline{a},\bar{a}\right)\neq\emptyset$.
We denote the trajectories in $P^{\mbox{SH}}\left(\underline{a},\bar{a},\underline{a},\bar{a}\right)$
with $\left[p_{n}\right]{}_{n\in\mathcal{N}}$. Note that for each
trajectory $p_{n}$, if BSP in SH3 is activated, then it shall be
always successfully completed within highway segment $\left[\hat{L}:=L-\bar{v}^{2}/\left(-2\bar{a}^{\mbox{b}}\right)-\bar{v}^{2}/\left(-2\bar{a}^{\mbox{b}}\right),L\right]$.
Therefore, BSP neither causes infeasibility nor affects the shape
of $p_{\mbox{n}}$ within $\left[0,\hat{L}_{\mbox{n}}:=L-\bar{v}^{2}/\left(-2\bar{a}^{\mbox{b}}\right)-\bar{v}^{2}/\left(-2\bar{a}^{\mbox{b}}\right)-s(n-1)\right]$.
Even if we consider the backward wave propagation caused by the previous
trajectories. This indicates $p_{n}\left(t_{n}^{-}:\dot{p}_{n}^{-1}\left(\hat{L}_{\mbox{n}}\right)\right)=p_{n}^{\mbox{f}}\left(t_{n}^{-}:\dot{p}_{n}^{-1}\left(\hat{L}_{\mbox{n}}\right)\right)$,
and if the feasibility check of $p_{n}$ fails in SH, it must happens
in this section. Note that Theorem \ref{theo: feasibility_all_green}
shows that if $\left[v_{n}^{-},t_{n}^{-}\right]_{n\in\mathcal{N}}$
is proper, $P^{\mbox{PSH}}\left(\underline{a},\bar{a}\right)$ is
not empty. We denote $P^{\mbox{PSH}}\left(\underline{a},\bar{a}\right)$
with $\left[\hat{p}_{n}\right]{}_{n\in\mathcal{N}}$. Then to prove
this theorem, we only need to show that $p_{n}\left(t_{n}^{-}:p_{n}^{-1}\left(\hat{L}_{n}\right)\right)=\hat{p}_{n}\left(t_{n}^{-}:p_{n}^{-1}\left(\hat{L}_{n}\right)\right),\forall n\in\mathcal{N}$
from the following induction. 

The induction assumption is $p_{n}\left(t_{n}^{-}:p_{n}^{-1}\left(\hat{L}_{n}\right)\right)=\hat{p}_{n}\left(t_{n}^{-}:p_{n}^{-1}\left(\hat{L}_{n}\right)\right)$
and $\dot{p}_{n}(t)=\bar{v},\forall t\in[p_{n}^{-1}(\hat{L}_{n}),p_{n}^{-1}(\hat{L})]$.
When $n=1$, the trajectory generated from FSP satisfies $p_{1}^{\mbox{f}}:=p_{1}^{\mbox{f}}\left((0,v_{1}^{-},t_{1}^{-}),\emptyset\right)=\hat{p}_{n}$.
Since BSP will not affect the shape of $p_{1}$ below $\hat{L}_{1}$,
thus $p_{1}\left(t_{1}^{-}:p_{1}^{-1}\left(\hat{L}_{1}\right)\right)=\hat{p}_{1}\left(t_{1}^{-}:p_{1}^{-1}\left(\hat{L}_{1}\right)\right)$.
Since $p_{1}^{\mbox{f}}$ shall finish accelerating to $\bar{v}$
before reaching location $\hat{L}$ based on the value of $\hat{L}$,
then $\dot{p}_{1}(t)=\dot{\hat{p}}_{1}(t)=\bar{v},\forall t\in[p_{n}^{-1}(\hat{L}_{n}),p_{n}^{-1}(\hat{L})]$.
Then we assume that the induction assumption holds for $n=k-1,\forall k\in\mathcal{N}\backslash\{1\}.$
When $n=k$, the forward shooting trajectory is $p_{k}^{\mbox{f}}:=p^{\mbox{f}}\left(\left(0,v_{k}^{-},t_{k}^{-}\right),p_{k-1}^{\mbox{s}}\right)$.
If $D\left(p_{k-1}^{\mbox{s}}-\bar{p}_{0v_{k}^{-}t_{k}^{-}}\right)\ge0$,
then apparently $p_{k}^{\mbox{f}}=\hat{p}_{k}=p^{\mbox{f}}\left(\left(0,v_{k}^{-},t_{k}^{-}\right),\emptyset\right),$
and thus $p_{k}(t_{k}^{-}:p_{k}^{-1}(\hat{L}_{k}))=p_{k}^{\mbox{f}}(t_{k}^{-}:p_{k}^{-1}(\hat{L}_{k}))=\hat{p}_{k}(t_{k}^{-}:p_{k}^{-1}(\hat{L}_{k}))$
and $\dot{p}_{k}(t)=\bar{v},\forall t\in[p_{k}^{-1}(\hat{L}_{k}),p_{k}^{-1}(\hat{L})]$.
Otherwise, note that $p_{k}^{\mbox{f}}=p^{\mbox{f}}((0,v_{k}^{-},t_{k}^{-}),p_{k-1}^{s})=p^{\mbox{f}}((0,v_{k}^{-},t_{k}^{-}),\hat{p}_{k-1}^{s})=\hat{p}_{k}$
and thus the existence of $\hat{p}_{k}$ indicates that $p_{k}^{\mbox{f}}$
exists as well. The induction assumption of $\dot{p}_{k-1}(t)=\bar{v},\forall t\in[p_{k-1}^{-1}(\hat{L}_{k-1}),p_{k-1}^{-1}(\hat{L})]$
indicates that $\dot{p}_{k-1}^{\mbox{s}}(t)=\bar{v},\forall t\in[p_{k-1}^{\mbox{s}-1}(\hat{L}_{k}=\hat{L}_{k-1}-s),p_{k-1}^{\mbox{s}-1}(\hat{L}-s)]$,
and therefore, $p_{k}^{\mbox{f}}$ shall merge with $p_{k-1}^{\mbox{s}}$
at a location before $\hat{L}$ because the merging speed has to be
less than $\bar{v}$, and therefore $\dot{p}_{k}(t)=\dot{p}_{k}^{\mbox{f}}(t)=\bar{v},\forall t\in[p_{k}^{-1}(\hat{L}_{k}),p_{k}^{-1}(\hat{L})]$.
Therefore, $p_{k}$ and $\hat{p}_{k}$ overlap over highway segment
$\left[0,\hat{L}_{n}\right]$, or $p_{k}\left(t_{k}^{-}:p_{k}^{-1}\left(\hat{L}_{k}\right)\right)=\hat{p}_{k}\left(t_{k}^{-}:p_{k}^{-1}\left(\hat{L}_{k}\right)\right)$
. This completes the proof.
\end{proof}

\subsection{Relationship to Classic Traffic Flow Models}

Evolution of highway traffic has been traditionally investigated with
various microscopic models (e.g., car following \citep{Brackstone1999181},
cellular automata \citep{Nagel1992}) and macroscopic kinematic models
(kinemetic models \citep{Lighthill1955,Richards1956} and cell transmission
\citep{Daganzo1994}). \citet{Daganzo2006} proves the equivalence
between the kinematic wave model with the triangular fundamental diagram
(KWT) \citep{newell1993simplified}, Newell's lower-order model \citep{Newell2002}
and the linear cellular automata model \citep{Nagel1992}. We will
just show the relevance of the the SHL solution to the KWT model,
and this relevance can be easily transferred to other models based
on their equivalence. Given the first vehicle's trajectory $q_{1}=p_{1}\in\mathcal{T}$,
the KWT model specifies a rule to construct vehicle $n$'s trajectory,
denoted by $q_{n}$, with vehicle $n's$ entry condition $(0,\cdot,t_{n}^{-})$
and preceding trajectory $q_{n-1}$, as formulated below 
\begin{equation}
q_{n}(t)=\min\left\{ \bar{v}(t-t_{n}^{-}),q_{n-1}(t-\tau)-s\right\} ,\forall n\in\mathcal{N}\backslash\{1\},t\in[t_{n}^{-},\infty).\label{eq:KWT-rule}
\end{equation}
For notation convenience, we denote this equation with $q_{n}=q^{\mbox{KWT}}\left((0,\cdot,t_{n}^{-}),q_{n-1}\right)$.
This section analyzes the relationship between the SHL solution with
the KWT solution for the same LVP setting. We denote the SHL trajectory
vector with $P^{\mbox{LVP}}\left(\underline{a}^{\mbox{f}},\bar{a}^{\mbox{f}}\right)=\left[p_{n}\right]_{n\in\mathcal{N}}$
and that from KWT by $Q=\left[q_{n}\right]_{n\in\mathcal{N}}$ . Without
loss of generality, we only investigate the case when $\left[v_{n}^{-},t_{n}^{-}\right]_{n\in\mathcal{N}}$
is proper or $P^{\mbox{LVP}}\left(\underline{a}^{\mbox{f}},\bar{a}^{\mbox{f}}\right)$
is feasible.

\citet{Daganzo2006} showed that KWT has a contraction property; i.e.,
the result of KWT is insensitive to small input errors, as stated
in the following proposition.
\begin{prop}
Given two trajectories $q,q'\in\bar{\mathcal{T}}$ satisfying $\max_{t\in(-\infty,\infty)}\left|q(t)-q'(t)\right|\le\epsilon$
for some $\epsilon>0$, then $\max_{t\in(-\infty,\infty)}\left|q^{\mbox{KWT}}\left((0,\cdot,t^{-}),q^{\mbox{s}}\right)-q^{\mbox{KWT}}\left((\delta,\cdot,t^{-}),q^{'\mbox{s}}\right)\right|\le\epsilon$\textup{
}for any $\left|\delta\right|<\epsilon$ and\textup{ $t^{-}$ }yielding
\textup{$q^{\mbox{KWT}}\left((0,\cdot,t^{-}),q^{\mbox{s}}\right)\neq\emptyset$}
and\textup{ $q^{\mbox{KWT}}\left((\delta,\cdot,t^{-}),q^{'\mbox{s}}\right)\neq\emptyset$.}
\end{prop}
We now show that $P^{\mbox{LVP}}\left(\underline{a}^{\mbox{f}},\bar{a}^{\mbox{f}}\right)$
has the same contraction property.
\begin{thm}
\label{Theo: contraction}Given two feasible trajectories $p,p'\in\mathcal{T}$
satisfying $\max_{t\in(-\infty,\infty)}\left|p(t)-p'(t)\right|\le\epsilon$,
then $\max_{t\in(-\infty,\infty)}\left|p^{\mbox{f}}((0,v^{-},t^{-}),p^{\mbox{s}})-p^{\mbox{f}}((\delta,v{}^{-},t^{-}),p{}^{'\mbox{s}})\right|\le\epsilon$
for any $\left|\delta\right|<\epsilon$ and any \textup{$(v^{-},t^{-})$
}yielding $p^{\mbox{f}}((0,v^{-},t^{-}),$ $p^{\mbox{s}})\neq\emptyset$
and $p^{\mbox{f}}((\delta,v^{-},t^{-}),p{}^{'\mbox{s}})\neq\emptyset$.\end{thm}
\begin{proof}
Since $\max_{t\in(-\infty,\infty)}\left|p(t)-p'(t)\right|\le\epsilon$,
we obtain $\max_{t\in(-\infty,\infty)}\left|p^{\mbox{s}}(t)-p^{'\mbox{s}}(t)\right|\le\epsilon.$
Further, define $p^{\mbox{s}+\epsilon}(t):=p^{\mbox{s}}(t)+\epsilon,$
$p^{\mbox{s}-\epsilon}(t):=p^{\mbox{s}}(t)-\epsilon$, $p^{\mbox{f}+\epsilon}:=p^{\mbox{f}}\left(\left(\epsilon,v^{-},t^{-}\right),p^{\mbox{s}+\epsilon}\right)$
and $p^{\mbox{f}-\epsilon}:=p^{\mbox{f}}\left(\left(-\epsilon,v^{-},t^{-}\right),p^{\mbox{s}-\epsilon}\right)$.
Note that $p(t)=p^{+\epsilon}(t)-\epsilon=p^{-\epsilon}(t)+\epsilon,\forall t\in(-\infty,\infty)$.
Since $D\left(p^{s+\epsilon}-p^{\mbox{s}}\right)=\epsilon$ and $\max_{t\in(-\infty,\infty)}\left|p(t)-p'(t)\right|\le\epsilon$,
we obtain $D\left(p^{s+\epsilon}-p^{\mbox{'s}}\right)\ge0$. Then
Proposition \ref{prop:bounded_transitive} indicates that $\mathcal{C}_{\delta vt^{-}}^{p^{\mbox{'s}}}\subseteq\mathcal{C}_{\delta vt^{-}}^{p^{\mbox{s}+\epsilon}}$.
Further, Proposition \ref{prop: FSP-shift-bound} indicates $D\left(p^{\mbox{f}+\epsilon}-\hat{p}\right)\ge0,\forall\hat{p}\in\mathcal{C}_{\delta vt^{-}}^{p^{\mbox{s}+\epsilon}}$.
We denote $p^{\mbox{f}}\left((0,v^{-},t^{-}),p^{\mbox{s}}\right)$
and $p^{\mbox{f}}\left((\delta,v{}^{-},t^{-}),p{}^{'\mbox{s}}\right)$
as $p^{\mbox{f}}$ and $p^{'\mbox{f}}$ for short, respectively. Since
$p^{'\mbox{f}}\in\mathcal{C}_{\delta vt^{-}}^{p^{\mbox{'s}}}\subseteq\mathcal{C}_{\delta vt^{-}}^{p^{\mbox{s}+\epsilon}}$,
we obtain $D\left(p^{\mbox{f}+\epsilon}-p^{'\mbox{f}}\right)\ge0$.
On the other hand, we obtain $D\left(p^{\mbox{'s}}-p^{s-\epsilon}\right)\ge0$.
Then Proposition \ref{prop:bounded_transitive} indicates $\mathcal{C}_{-\epsilon vt^{-}}^{p^{\mbox{s}-\epsilon}}\subseteq\mathcal{C}_{-\epsilon vt^{-}}^{p^{\mbox{'s}}}$.
Further, Proposition \ref{prop: FSP-shift-bound} indicates $D\left(p'-\hat{p}\right)\ge0,\forall\hat{p}\subseteq\mathcal{C}_{-\epsilon vt^{-}}^{p^{\mbox{'s}}}$.
Since $p^{\mbox{f}-\epsilon}\in\mathcal{C}_{-\epsilon vt^{-}}^{p^{\mbox{s}-\epsilon}}\subseteq\mathcal{C}_{-\epsilon vt^{-}}^{p^{\mbox{'s}}}$,
we obtain $D\left(p^{'\mbox{f}}-p^{\mbox{f}-\epsilon}\right)\ge0$.
Combining $D\left(p^{\mbox{f}+\epsilon}-p^{'\mbox{f}}\right)\ge0$
and $D\left(p^{'\mbox{f}}-p^{\mbox{f}-\epsilon}\right)\ge0$ yields
$\max_{t\in(-\infty,\infty)}\left|p^{\mbox{f}}(t)-p^{'\mbox{f}}(t)\right|\le\epsilon$
.
\end{proof}
Theorem \ref{Theo: contraction} impliesthat $P^{\mbox{LVP}}\left(\underline{a}^{\mbox{f}},\bar{a}^{\mbox{f}}\right)$
is not sensitive to small input errors as well. However, we shall
note that the solution to the SH algorithm for a signalized section
may be sensitive to small errors because the exit time of a trajectory,
if close to the start of a red phase, could be pushed back to the
next green phase due to a small input perturbation. Nonetheless, this
kind of ``jump'' only affects a limited number of trajectories that
are close to a red phase, and most other trajectories will not be
much affected. 

The following analysis investigates the difference between $P^{\mbox{LVP}}\left(\underline{a}^{\mbox{f}},\bar{a}^{\mbox{f}}\right)$
and $Q$.
\begin{prop}
\label{prop:Parallel_KWT}$Q=\left\{ q_{n}\right\} _{n\in\mathcal{N}}$
formulated in \eqref{eq:KWT-rule} can be solved as 
\begin{equation}
q_{n}(t)=\min\left\{ p_{1}^{\mbox{s}^{n-1}}(t),\bar{v}\left(t-t_{n}^{-}\right)\right\} ,\forall t\in[t_{n}^{-},\infty),n\in\mathcal{N}.\label{eq:Parallel-KWT}
\end{equation}
where $p_{1}^{\mbox{s}^{k}}(t):=p_{1}\left(t-k\tau\right)-ks,\forall k.$\end{prop}
\begin{proof}
We prove this proposition with induction. Apparently, equation \eqref{eq:Parallel-KWT}
holds for $n=1$ since $p_{1}^{\mbox{s}^{0}}=p_{1}\in\mathcal{T}$
and thus $\bar{v}\left(t-t_{1}^{-}\right)\ge p_{1}^{\mbox{s}^{0}}(t),\forall t\in[t_{1}^{-},\infty)$.
Assume that equation \eqref{eq:Parallel-KWT} holds for $n=k-1$.
When $n=k$, equation \eqref{eq:KWT-rule} indicates $q_{k}(t)=\min\{\bar{v}(t-t_{k}^{-}),q_{k-1}(t-\tau)-s\}$.
Then plug the induction assumption into the above equation and we
obtain $q_{k}(t)=\min\{\bar{v}(t-t_{k}^{-}),p_{1}^{\mbox{s}^{k-2}}(t-\tau)-s,\bar{v}\left(t-t_{k-1}^{-}-\tau\right)-s\}$$=\min\{p_{1}^{\mbox{s}^{k-1}}(t),\bar{v}(t-\max\{t_{k}^{-},t_{k-1}^{-}+\tau+\bar{v}/s\})\}.$
Since $\left[v_{n}^{-},t_{n}^{-}\right]_{n\in\mathcal{N}}$ is proper,
we obtain $\max\left\{ t_{k}^{-},t_{k-1}^{-}+\tau+\bar{v}/s\right\} =t_{k}^{-}$.
This completes the proof.\end{proof}
\begin{thm}
\label{theo:KWT_SHL_bounds} If $P^{\mbox{LVP}}\left(\underline{a}^{\mbox{f}},\bar{a}^{\mbox{f}}\right)\neq\emptyset$
and $Q\neq\emptyset$, then $D\left(q_{n}-p_{n}\right)=0$ and 
\begin{equation}
D\left(p_{n}-q_{n}\right)\ge\min\left\{ -0.5\bar{v}^{2}/\bar{a}^{\mbox{f}},0.5\bar{v}^{2}/\underline{a}^{\mbox{f}}\right\} ,\forall n\in\mathcal{N}\backslash\left\{ 1\right\} .\label{eq:SH-KWT lower bound}
\end{equation}
\end{thm}
\begin{proof}
We compare parallel formulations from $P^{\mbox{LVP}}\left(\underline{a}^{\mbox{f}},\bar{a}^{\mbox{f}}\right)$
from PSHL in Section \ref{sub:SH_LVP} with equation \eqref{eq:Parallel-KWT}
for $Q$. We only need to prove this theorem for a generic $n\in\mathcal{N}\backslash\left\{ 1\right\} $.
From PSHL, we see that $p_{n}$ is essentially obtained by smoothing
quasi-trajectory $u_{n}:=u\left(\left\{ \bar{p}_{m}^{\mbox{s}^{n-m}}\right\} _{m=1,\cdots,n}\right)$
with tangent segments at constant decelerating rate $\underline{a}$.
Define $\bar{q}_{n}(t):=\bar{v}\left(t-t_{n}^{-}\right),\forall t\in\left[t_{n}^{-},+\infty\right),\underline{q}_{n}:=p^{\mbox{f}}\left(\left(0,0,\underline{t}_{n}\right),\emptyset\right)$.
Note that $D\left(\underline{q}_{n}-\bar{q}_{n}\right)=-0.5\bar{v}^{2}/\bar{a}^{\mbox{f}}$,
and $q_{n}(t)=\min\left\{ \bar{q}_{n}(t),p_{1}^{\mbox{s}^{n-1}}(t)\right\} .$
Note that $D\left(q_{n}-u_{n}\right)\ge0$ since $\bar{q}_{n}\ge\bar{p}_{n}$.
Also, $D\left(u_{n}-p_{n}\right)\ge0$ since the merging segments
generated from EFSO-2 are all below $u_{n}$. This indicates $D\left(q_{n}-p_{n}\right)\ge0$.
Further since $q_{n}$ and $p_{n}$ always meet at the initial point
$(0,t_{n}^{-})$, we obtain $D\left(q_{n}-p_{n}\right)=0$.

Then we will prove $D\left(p_{n}-q_{n}\right)\ge\min\left\{ -0.5\bar{v}^{2}/\bar{a}^{\mbox{f}},0.5\bar{v}^{2}/\underline{a}^{\mbox{f}}\right\} $.
We first examine $p'_{n}:=p^{\mbox{f}}((0,v_{n}^{-},t_{n}^{-}),$
$\{\bar{p}_{m}^{\mbox{s}^{n-m}}\}_{m=2,\cdots,n-1}),\forall n\in\mathcal{N}\backslash\{1\}$,
and claim $D\left(p'_{n}-\underline{q}_{n}\right)\ge0$. Basically,
$p'_{n}$ is composed by some shooting segments from $\left\{ \bar{p}_{m}^{\mbox{s}^{n-m}}\right\} _{m=2,\cdots,n-1}$and
intermediate tangent merging segments. Apparently, $D\left(\bar{p}_{m}^{\mbox{s}^{n-m}}-\underline{q}_{n}\right)\ge0$.
Note that every merging segment shall be above $\underline{p}_{n}$
or it will not catch up with the next shooting segment that accelerates
at the maximum rate $\bar{a}^{\mbox{f}}$ before reaching speed $\bar{v}$.
Thus this claim holds and thus $D\left(p'_{n}-\bar{q}_{n}\right)\ge-0.5\bar{v}^{2}/\bar{a}^{\mbox{f}}$.
Therefore, $D\left(p'_{n}-q_{n}\right)\ge-0.5\bar{v}^{2}/\bar{a}^{\mbox{f}}$.
Note that $p_{n}$ is obtained by merging $p'_{n}$ and $\bar{p}_{1}^{\mbox{s}^{n-1}}$with
a tangent segment denoted by $\hat{s}_{n}$. Apparently, $D\left(\bar{p}_{1}^{\mbox{s}^{n-1}}-q_{n}\right)\ge0$
since $q_{n}(t)$ is defined as $\min\left\{ \bar{q}_{n}(t),p_{1}^{\mbox{s}^{n-1}}(t)\right\} $.
Further note that $D\left(\hat{s}_{n}-\bar{p}_{1}^{\mbox{s}^{n-1}}\right)\ge0.5\bar{v}^{2}/\underline{a}^{\mbox{f}}$,
and since $D\left(\bar{p}_{1}^{\mbox{s}^{n-1}}-q_{n}\right)\ge0$,
we obtain $D\left(\hat{s}_{n}-q_{n}\right)\ge0.5\bar{v}^{2}/\underline{a}^{\mbox{f}}$.
Therefore $D(p_{n}-q_{n})$$\ge\min\left\{ D\left(p'_{n}-q_{n}\right),D\left(\bar{p}_{1}^{\mbox{s}^{n-1}}-q_{n}\right),D\left(\hat{s}_{n}-q_{n}\right)\right\} \ge\min\left\{ -0.5\bar{v}^{2}/\bar{a}^{\mbox{f}},0.5\bar{v}^{2}/\underline{a}^{\mbox{f}}\right\} $.
This completes the proof.
\end{proof}
The above theorem reveals that a trajectory generated from$P^{\mbox{LVP}}\left(\underline{a}^{\mbox{f}},\bar{a}^{\mbox{f}}\right)$
should be always below the counterpart trajectory in $Q$ and the
difference is attributed to smoothed accelerations instead of speed
jumps. One elegant finding from this theorem is that this difference
does not accumulate much across vehicles but is bounded by a constant.
Further, the lower bound to this difference is actually tight and
we can find instances where $D\left(p_{n}-q_{n}\right)$ for every
vehicle $n\in\mathcal{N}\backslash\left\{ 1\right\} $ is exactly
identical to $\min\left\{ -0.5\bar{v}^{2}/\bar{a}^{\mbox{f}},0.5\bar{v}^{2}/\underline{a}^{\mbox{f}}\right\} $.
One such instance can be specified by setting $p_{1}(t)=L,\forall t$
for a sufficiently long $L$, and$v_{n}^{-}=0,\forall n\in\mathcal{N}\backslash\left\{ 1\right\} $.
This theorem also leads to the following asymptotic relationship.
\begin{cor}
If $\bar{a}^{\mbox{f}}\rightarrow\infty$ and $\underline{a}^{\mbox{f}}\rightarrow-\infty$,
then we have $P^{\mbox{LVP}}\left(\underline{a}^{\mbox{f}},\bar{a}^{\mbox{f}}\right)\rightarrow Q.$\label{cor_KWT_SH_asymptotic}
\end{cor}
This relationship indicates that SHL can be viewed as a generalization
of KWT as well as other equivalent models, including Newell's lower
order model and the linear cellular automata model. Essentially, the
SH solution can be viewed as a smoothed version of these classic models
that circumvents a common issue of these classic models, i.e., infinite
acceleration/deceleration or ``speed jumps''. Such speed jumps would
cause unrealistic evaluation of traffic performance measures, particularly
those on safety and environmental impacts. Further, SHL inherits the
simple structure of these models and thus can be solved very efficiently.
One commonality between SHL and KWT is that at stationary states,
i.e., when each trajectory move at a constant speed, both the SHL
solution and the KWT solution are consistent with a triangular fundamental
diagram \citealt{newell1993simplified}, as stated below. 
\begin{thm}
\label{theo: fundamental} If solutions $P^{\mbox{LVP}}\left(\underline{a}^{\mbox{f}},\bar{a}^{\mbox{f}}\right)$
and $Q$ are both feasible and each $\dot{p}_{n}(t)$ or $\dot{q}_{n}(t)$
remains constant $\forall t\ge t_{n}^{-}$, then there exists some
$V\in[0,\bar{v}]$ such that

\begin{equation}
p_{n}(t)=q_{n}(t)=V\left(t-t_{n}^{-}\right),\dot{p}_{n}(t)=\dot{q}_{n}(t)=V,\forall t\in\left[t_{n}^{-},\infty\right),n\in\mathcal{N}.\label{eq:p_q_equal}
\end{equation}
Further, if $V=\bar{v}$,
\begin{equation}
t_{n}^{-}-t_{n-1}^{-}\ge\tau+V/s,\forall n\in\mathcal{N},\label{eq:t_n_minus_sep_less}
\end{equation}
If $V<\bar{v},$ then 

\begin{equation}
t_{n}^{-}=(n-1)(\tau+V/s)-t_{1}^{-},\forall n\in\mathcal{N}\backslash\left\{ 1\right\} .\label{eq:t_n_minus_sep_equal}
\end{equation}
Further, define the traffic density as $K:=\frac{N-1}{\sum_{n=2}^{N}\left(p_{n-1}(t)-p_{n}(t)\right)}=\frac{N-1}{\sum_{n=2}^{N}\left(q_{n-1}(t)-q_{n}(t)\right)}$
for any $t\ge t_{n}^{-}$ and traffic volume as $O:=\frac{N-1}{p_{N}^{-1}(l)-p_{1}^{-1}(l)}=\frac{N-1}{q_{N}^{-1}(l)-q_{1}^{-1}(l)}$
for any $l\ge0$, then we will always have 
\begin{equation}
\begin{cases}
K\le1/(s+V\tau), & \mbox{if }V=\bar{v};\\
K=1/(s+V\tau) & \mbox{if }V<\bar{v}.
\end{cases}\label{eq:K-V}
\end{equation}

\begin{equation}
O=KV=\min\left\{ K\bar{v},(1-sK)/\tau\right\} \in\left[0,\bar{v}/(s+\bar{v}\tau)\right].\label{eq:O-K}
\end{equation}
\end{thm}
\begin{proof}
We fist investigate the case when every $\dot{p}_{n}(t)$ is constant.
If there exist two consecutive trajectories not parallel with each
other, then they will either intersect or depart from each other to
an infinite spacing since they both are straight lines. The former
is impossible because of safety constraints \eqref{eq: set-safety-constraints}.
Neither is the latter possible since the following trajectory has
to accelerate when their spacing exceeds the shadow spacing and runs
at a speed lower than the preceding vehicle. Therefore $\dot{p}_{n}(t)$
values are identical to a $V\in[0,\bar{v}]$ across $n\in\mathcal{N}$.
Therefore, $p_{n}(t)=V\left(t-t_{n}^{-}\right),\dot{p}_{n}(t)=V,\forall t\in\left[t_{n}^{-},\infty\right),n\in\mathcal{N}.$

Since $P^{\mbox{LVP}}\left(\underline{a}^{\mbox{f}},\bar{a}^{\mbox{f}}\right)$
is feasible, then we know from Theorem \ref{theo: feasibility_all_green}
that $\left[v_{n}^{-},t_{n}^{-}\right]_{n\in\mathcal{N}}$ is proper,
which indicates that equation \eqref{eq:t_n_minus_sep_less} holds
either $V=\bar{v}$ or $V<\bar{v}$. Further, when $V<\bar{v}$, equation
\eqref{eq:t_n_minus_sep_less} becomes a strict equality (or otherwise
the following trajectory needs to accelerate). Thus equation \eqref{eq:t_n_minus_sep_equal}
holds in this case. 

Then we will show that $q_{n}=p_{n}$. Since all trajectories in $P^{\mbox{LVP}}\left(\underline{a}^{\mbox{f}},\bar{a}^{\mbox{f}}\right)$
are at speed $V$, we know that $v_{n}^{-}=V,\forall n\in\mathcal{N}$
and the lead trajectory is given as $p_{1}(t)=V\left(t-t_{1}^{-}\right)$.
If $V=\bar{v}$, apparently no trajectory blocks its following trajectory,
and thus $q_{n}(t)=V\left(t-t_{n}^{-}\right)=p_{n}(t),\forall t\in\left[t_{n}^{-},\infty\right),n\in\mathcal{N}.$
Otherwise if $V<\bar{v}$, due to equation \eqref{eq:t_n_minus_sep_equal},
$q_{n}$ has no room to accelerate and thus has to stay at speed $V$.
Therefore, equation \eqref{eq:p_q_equal} holds. Then we will prove
equations \eqref{eq:K-V} and \eqref{eq:O-K}, which quantify the
triangular fundamental diagram. First, since $\left[v_{n}^{-},t_{n}^{-}\right]_{n\in\mathcal{N}}$
is proper, then we obtain $p_{n-1}\left(t_{n}^{-}\right)-p_{n}\left(t_{n}^{-}\right)\ge V\tau+s,\forall n\in\mathcal{N}\backslash\left\{ 1\right\} $,
and thus $K\ge1/(V\tau+s)$ holds. When $V<\bar{v}$, due to \eqref{eq:t_n_minus_sep_equal},
then $p_{n-1}\left(t_{n}^{-}\right)-p_{n}\left(t_{n}^{-}\right)=V\tau+s,\forall n\in\mathcal{N}\backslash\left\{ 1\right\} $,
and thus $K=1/(V\tau+s)$ holds. This proves equation \eqref{eq:K-V}.
Further note that $p_{n}^{-1}(0)-p_{n-1}^{-1}(0)=\left(p_{n-1}\left(t_{n}^{-}\right)-p_{n}\left(t_{n}^{-}\right)\right)/V,\forall n\in\mathcal{N}\backslash\left\{ 1\right\} $,
then $O=\frac{N-1}{\sum_{n=2}^{N}\left(p_{n}^{-1}(0)-p_{n-1}^{-1}(0)\right)}=\frac{N-1}{\sum_{n=2}^{N}\left(p_{n-1}\left(t_{n}^{-}\right)-p_{n}\left(t_{n}^{-}\right)\right)/V}=KV$.
When $V=\bar{v}$, plugging inequality $K\le1/(s+V\tau)$ in equation
\eqref{eq:K-V} into the previous equation, we obtain $KV=K\bar{v}\le\bar{v}/(s+\bar{v}\tau)$,
and $K\le1/(s+V\tau)$ can be rearranged into to $(1-sK)/\tau\ge KV$.
Therefore, equation \eqref{eq:O-K} holds in this case. When $V<\bar{v}$,
apparently, $KV<K\bar{v}$. Further, equality $K=1/(s+V\tau)$ in
equation \eqref{eq:K-V} can be rearranged as $KV=(1-sK)/\tau$. Note
that $K=1/(s+V\tau)>1/(s+\bar{v}\tau)$, and plugging this inequality
into $KV=(1-sK)/\tau$ yields $KV=(1-sK)/\tau<(1-s/(s+\bar{v}\tau))/\tau=\bar{v}/(s+\bar{v}\tau)$.
This indicates that equation \eqref{eq:O-K} holds in the case too.
This completes the proof. 
\end{proof}

\section{Illustrative Examples \label{sec:Numerical-Examples}}

This section presents a few illustrative examples that help visualize
results of the algorithms and related theoretical properties. In the
following examples, we set the parameters to their default values
unless explicitly stated otherwise. These default parameter values
are $L=1000$ m, $N=50$, $\underline{a}=-5\mbox{ m/s}^{2}$, $\overline{a}=2\mbox{ m/s}^{2}$,
$\bar{v}=25$ m/s, $G=R=25\mbox{ s}$, $s=7$ m, and the default boundary
condition is generated as follows. We first set the arrival times
as 
\begin{equation}
t_{1}^{-}=0,\,t_{n}^{-}=t_{n-1}^{-}+(\tau+s/\bar{v})\left[1+\left(\frac{C}{Gf^{\mbox{s}}}-1\right)\left(1-\alpha+\alpha\epsilon_{n}\right)\right],\forall n\in\mathcal{N}\backslash\{1\},\label{eq:t_dist}
\end{equation}
where $f^{\mbox{s}}\in\left(0,C/G\right]$ is the traffic saturation
rate (or the ratio of the arrival traffic volume to the intersection's
maximum capacity), $\alpha\in[0,1]$ is the dispersion factor of the
headway distribution, and $\left\{ \epsilon_{n}\right\} _{n\in\mathcal{N}\backslash\{1\}}$
are uniformly distributed non-negative random numbers that satisfy
$\sum_{n\in\mathcal{N}\backslash\{1\}}\epsilon_{n}=N-1$. The default
values of $f^{\mbox{s}}$ and $\alpha$ are both set to 1. We set
the arrival times in this way such that the average time headway equals
$(\tau+s/\bar{v})C/\left(Gf^{\mbox{s}}\right)$ yet the individual
arrival times can be stochastic. The stochasticity of the arrival
times increases with headway dispersion factor $\alpha$. Note that
we always maintain every headway no less than $(\tau+s/\bar{v})$
because the boundary condition would be infeasible otherwise. Next,
for each $n\in\mathcal{N}$, the initial speed $v_{n}^{-}$ is consecutively
drawn as a random number uniformly distributed over the corresponding
lower bound and upper bound. These two bounds assure that $v_{n}^{-}$
satisfies $v_{n}^{-}\in[0,\bar{v}]$, $D\left(\bar{p}_{(ms-ns)v_{m}^{-}\left[t_{m}^{-}+(n-m)\tau\right]}-\underline{p}_{0v_{n}^{-}t_{n}^{-}}\right)\ge0,\forall m<n\in\mathcal{N}$
and $D\left(\bar{p}_{(ns-ms)v_{n}^{-}\left[t_{n}^{-}+(m-n)\tau\right]}-\underline{p}_{0v_{m}^{-}t_{m}^{-}}\right)\ge0,\exists v_{m}^{-}\in[0,\bar{v}],\forall m>n\in\mathcal{N}$
with the maximum acceleration and the minimum deceleration for QTG
downscaled to $\bar{a}/3$ and $\underline{a}/3$, respectively. The
downscale of the acceleration limits assures that the boundary condition
is feasible for any forward acceleration $\bar{a}^{\mbox{f}}\ge\bar{a}/3$
and any forward deceleration $\underline{a}^{\mbox{f}}\ge\underline{a}/3$,
which allows to test a range of $\bar{a}^{\mbox{f}}$ and $\underline{a}^{\mbox{f}}$
values. If a proper boundary condition cannot be found, we will regenerate
random parameters $\left\{ \epsilon_{n}\right\} _{n\in\mathcal{N}\backslash\{1\}}$
with a different random seed and repeat this process until finding
a proper boundary condition.

Section \ref{sub:Manual-v.s.-Automated} compares the SH solutions
with a benchmark instance that simulates the manually-driven traffic
counterpart. This comparison aims to qualitatively show the advantaged
of the proposed CAV control strategies over manual driving. Section
\ref{sub:Leading-Vehicle-Problem} compares SHL and PSHL results and
shows they produce the identical results for the same LVP input. Section
\ref{sub:Feasibility-Tests} tests the feasibility of SH and SHL solutions
with different boundary conditions to verify some theory predictions
and reveal insights into how parameter changes affect the solution
feasibility. Section \ref{sub:Comparison-with-Classic} compares SHL
and LWK results and measures their difference to check the theoretical
error bounds.

\subsection{Manual v.s. Automated Trajectories\label{sub:Manual-v.s.-Automated}}

To illustrated the advantage of results, we construct a benchmark
instance that simulates the manually-driven traffic counterpart. We
adapt the Intelligent Driver model \citep{Treiber2000b} as the manual-driving
rule for every vehicle:

\begin{equation}
\ddot{p}_{n}(t)=\max\left\{ \min\left\{ \bar{a}\left(1-\frac{s^{*}}{f{}_{n-1}(t-\tau)-p_{n}(t-\tau)-l^{0}}\right),\bar{a}'\right\} ,\underline{a}'\right\} ,\label{eq:IDM}
\end{equation}
where vehicle length $l^{0}=5$m, acceleration bounds $\bar{a}'=\begin{cases}
0, & \mbox{if }\dot{p}_{n}(t)\ge\bar{v};\\
\bar{a}, & \mbox{otherwise},
\end{cases}$ and $\underline{a}'=\begin{cases}
0, & \mbox{if }\dot{p}_{n}(t)\le0;\\
\underline{a}, & \mbox{otherwise},
\end{cases}$ , comfort deceleration $b=1.67$m/s$^{2}$, and the desired spacing
$s^{*}=(s-l^{0})+\dot{p}_{n}(t-\tau)\tau+\dot{p}_{n}(t-\tau)\frac{\dot{p}_{n}(t-\tau)-\dot{p}'_{n-1}(t-\tau)}{\sqrt{\bar{a}b}}$.
The effect of traffic lights is emulated with a bounding frontier
$f{}_{n-1}$ that is $p{}_{n-1}$ if $p_{n}$ is not blocked by a
red light or a virtual vehicle parked at $L+s$ otherwise. We define
the yellow time prior to the beginning of a red phase as $y=3s$ and
$f{}_{n-1}$ can be formulated as 
\[
f{}_{n-1}(t)=\begin{cases}
L+s, & \mbox{if }L-p_{n}(t-\tau)<\bar{v}^{2}/(2b),t\in[mC-y,mC+G]\mbox{ and }\bar{p}_{p_{n}(t-\tau)\dot{p}_{n}(t-\tau)(t-\tau)}^{-1}(L)\notin\mathcal{G};\\
p{}_{n-1}(t), & \mbox{otherwise},
\end{cases},
\]
To accommodate the lead vehicle that does not have a preceding trajectory,
without loss of generality, we define a virtual preceding vehicle
$p_{0}(t)=\infty$ and $\dot{p}_{0}(t)=\bar{v},\forall t$. There
are two reasons to select this model. First, the trajectories produced
from this model appear to be consistent with our driving experience
at a signalized intersection: we tend to slow down and make a stop
only when we get close the intersection at a yellow or red light.
Secondly, it is easy to verify that at a stationary state, the same
macroscopic relationship between density and flow volume defined in
Theorem \ref{theo: fundamental} holds. This way, the comparison between
the SH solution and this benchmark will only focus on the ``trajectory
smoothing'' effect rather than improvement of stationary traffic
characteristics, which has been investigated in other studies (e.g.,
\citet{Shladover2009}). 

Figure \ref{fig:traj_results} compares the benchmark result with
the SH output. Figure \ref{fig:traj_results}(a) plots the benchmark
manual-driving trajectories generated with car-following model \eqref{eq:IDM}.
We see that due to abrupt accelerations and decelerations in the vicinity
of the traffic lights, a number of consecutive stop-and-go waves are
formed and propagated backwards from the intersection. These stop-and-go
waves slow down the passing speed of the vehicles at the intersections,
and thus decrease the traffic throughput and increase the travel delay.
As a result, the total travel time (i.e., the time duration between
the first vehicle's entry at location 0 and the last vehicle's exist
at location $L$) is over 300 seconds for the benchmark case. Further,
it is intuitive that these stop-and-go waves adversely impact fuel
consumptions and emissions and amplify collision risks. Figure \ref{fig:traj_results}(b)
plots the SH result with $\left(\bar{a}^{\mbox{f}},\underline{a}^{\mbox{f}},\bar{a}^{\mbox{b}},\underline{a}^{\mbox{b}}\right)$
identical to their bounding values $\left(\bar{a},\underline{a},\bar{a},\underline{a}\right)$.
We see that despite some sharp accelerations and decelerations, all
vehicles can pass the intersection at the maximum speed and thus the
traffic throughput gets maximized. The total travel time now is only
around 170 seconds. Figure \ref{fig:traj_results}(c) plots the SH
result with $\left(\bar{a}^{\mbox{f}},\underline{a}^{\mbox{f}},\bar{a}^{\mbox{b}},\underline{a}^{\mbox{b}}\right)$
downscaled to $\left(\bar{a}/3,\underline{a}/3,\bar{a}/3,\underline{a}/9\right)$.
We see that with the acceleration/deceleration magnitudes reduces,
trajectories become much smoother while the total travel time keeps
low around 170 seconds. This will further reduce the traffic's environmental
impacts and enhance its safety.

\begin{figure}
\begin{centering}
\includegraphics{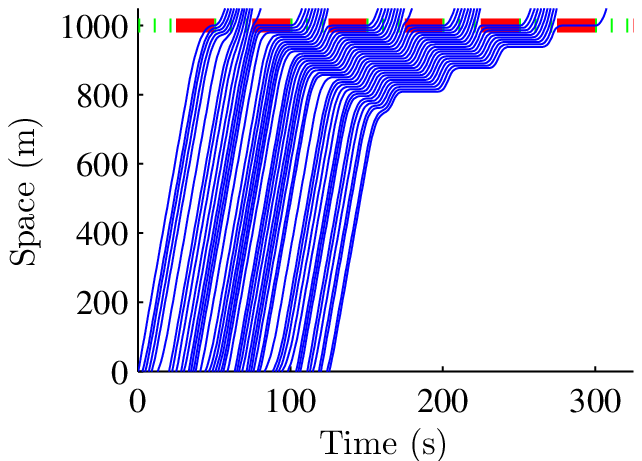}\includegraphics{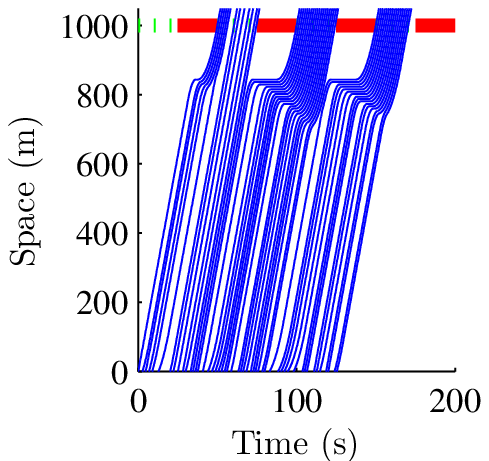}\includegraphics{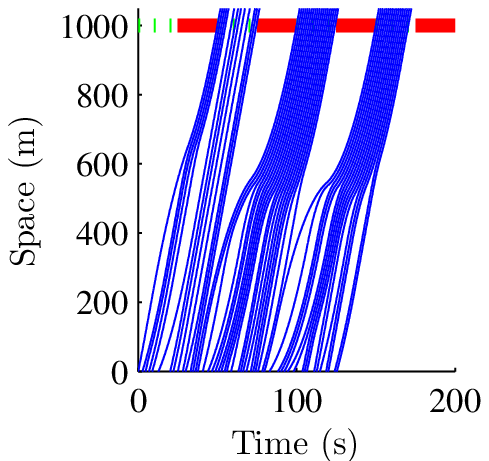}
\par\end{centering}

\begin{centering}
\qquad{}\qquad{}(a)\qquad{}\qquad{}\qquad{}\qquad{}\qquad{}\qquad{}\qquad{}(b)\qquad{}\qquad{}\qquad{}\qquad{}\qquad{}\qquad{}\qquad{}(c)
\par\end{centering}

\protect\caption{(a) Benchmark manual-driving trajectories, (b) SH result $P\left(\bar{a},\underline{a},\bar{a},\underline{a}\right),$
and (c) SH result $P\left(\bar{a}/3,\underline{a}/3,\bar{a}/3,\underline{a}/9\right).$
\label{fig:traj_results}}
\end{figure}

To inspect the differences between the benchmark and the SH result
from a macroscopic point of view, we measure the macroscopic traffic
characteristics, including density and flow volumes for the trajectory
sets in Figure \ref{fig:traj_results} with the measuring method proposed
by \citet{Laval2011}. Basically, we roll a parallelogram with a length
of 100m and a time interval of 5s along the shock wave direction (at
a speed of $-s/\tau$) across the trajectories in every plot in Figure
\ref{fig:traj_results} by a 100m$\times$5s step size. We measure
the flow volume and density at each parallelogram and plot the measurements
as circles in the corresponding diagrams in Figure \ref{fig:fund_measurements},
where the solid curves represent the stationary flow-density relationships
specified in equation \ref{eq:O-K}. We see that in Figure \ref{fig:fund_measurements}(a)
for the bench mark trajectories, many measurements are distributed
on the congested side of this diagram and most of them are below the
stationary curve, which explains why the performance of the benchmark
case is the worst. This is probably because traffic the stop and go
waves (or traffic oscillations) result in a lower traffic throughput
even at the same density, which is known as the capacity drop phenomenon
\citep{Cassidy1999,Ma2015}. In Figure \ref{fig:fund_measurements}(b)
for the SH result $P\left(\bar{a},\underline{a},\bar{a},\underline{a}\right)$,
there are much fewer measurements falling in the congested branch,
and these measurements are closer to the stationary curve. In Figure
\ref{fig:fund_measurements}(c) for the smoothed SH result $P\left(\bar{a}/3,\underline{a}/3,\bar{a}/3,\underline{a}/9\right)$,
even more measurements lie in the free-flow branch, and these measurements
become consistent with the stationary curve. This suggests that the
proposed SH algorithm with proper parameter values can counteract
the capacity drop phenomenon and bring macroscopic traffic characteristics
toward the free-flow branch of the stationary curve. 

\begin{figure}
\begin{centering}
\includegraphics[width=0.33\textwidth]{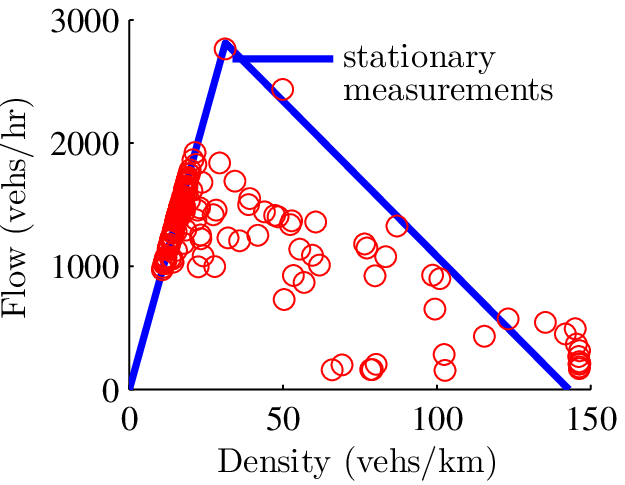}\includegraphics[width=0.33\textwidth]{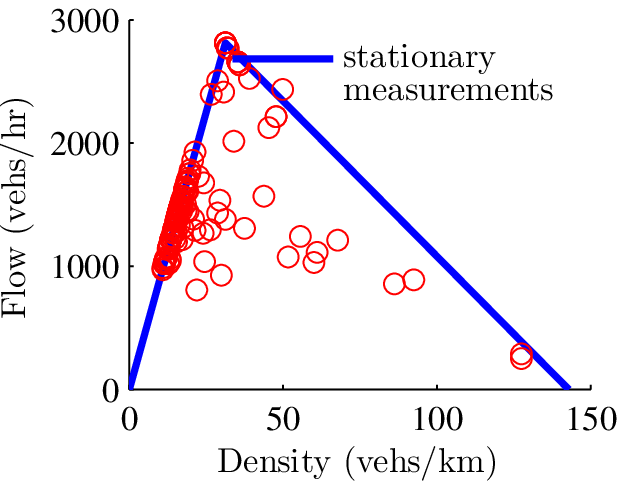}\includegraphics[width=0.33\textwidth]{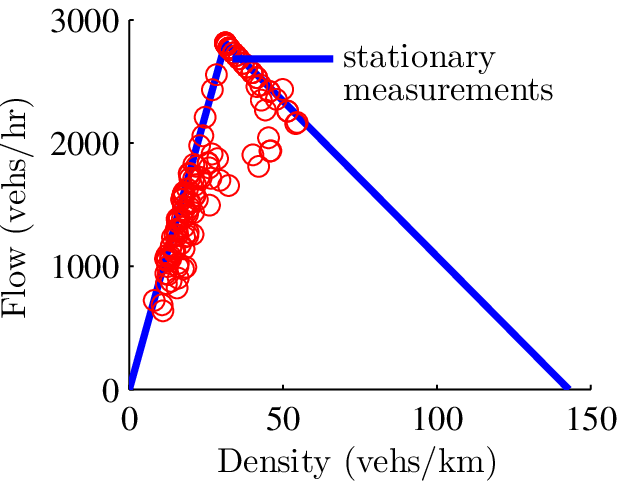}
\par\end{centering}

\begin{centering}
\qquad{}\qquad{}(a)\qquad{}\qquad{}\qquad{}\qquad{}\qquad{}\qquad{}\qquad{}(b)\qquad{}\qquad{}\qquad{}\qquad{}\qquad{}\qquad{}\qquad{}(c)
\par\end{centering}

\protect\caption{Macroscopic characteristics for (a) the benchmark trajectories, (b)
$P\left(\bar{a},\underline{a},\bar{a},\underline{a}\right)$, and
(c) $P\left(\bar{a}/3,\underline{a}/3,\bar{a}/3,\underline{a}/9\right).$
\label{fig:fund_measurements}}
\end{figure}

Overall, the results in this Subsection show that the proposed SH
algorithm can much improve the highway traffic performance in mobility,
environment and safety. To realize the full utility of the SH algorithm,
quantitative optimization needs to be conducted, which will be detailed
in Part II of this study.

\subsection{Lead Vehicle Problem \label{sub:Leading-Vehicle-Problem}}

This subsection presents LVP results from manual driving law \eqref{eq:IDM}
and the proposed CAV driving algorithms. In the LVP, we set $G=\infty$
and $R=0$, and we update saturation rate to $f_{s}=0.5$ in generating
the boundary condition (so that the average headway remains $0.5(\tau+s/\bar{v})$).
The lead trajectory is set to initially cruise at speed $\bar{v}$
for 20 seconds, then deceleration to the zero speed with a decelerating
rate of $\underline{a}/3$, then keep stopped for 20 seconds, then
accelerate to $\bar{v}$ with a rate of $\bar{a}/3$, and finally
keep cruising at this speed, i.e., 
\begin{eqnarray}
p_{1}: & = & \left[\left(0,\bar{v},0,0,20\right),\left(20\bar{v},\bar{v},\frac{\underline{a}}{3},20,20-\frac{3\bar{v}}{\underline{a}}\right),\left(20\bar{v}-\frac{3\bar{v}^{2}}{2\underline{a}},0,0,20-\frac{3\bar{v}}{\underline{a}},40-\frac{3\bar{v}}{\underline{a}}\right),\right.\nonumber \\
 &  & \left.\left(20\bar{v}-\frac{3\bar{v}^{2}}{2\underline{a}},0,\frac{\underline{a}}{3},40-\frac{3\bar{v}}{\underline{a}},40-\frac{3\bar{v}}{\underline{a}}+\frac{3\bar{v}}{\bar{a}}\right),\left(20\bar{v}-\frac{3\bar{v}^{2}}{2\underline{a}}+\frac{3\bar{v}^{2}}{2\bar{a}},\bar{v},0,40-\frac{3\bar{v}}{\underline{a}}+\frac{3\bar{v}}{\bar{a}},\infty\right)\right].\label{eq:leading_traj_gen}
\end{eqnarray}
This way, $p_{1}$ triggers a stopping wave and we can examine its
propagation under different driving conditions. Figure \ref{fig:traj_results-LVP}
shows the trajectory comparison results. We see that first, in Figures
\ref{fig:traj_results-LVP} (b) and (c), the PSHL results (solid lines)
exactly overlap with the SHL results (crosses), which verifies Proposition
\ref{prop: PSHL=00003DSHL} that states the equivalence between PSHL
and SHL. Compared with the manual-driving case in Figure \ref{fig:traj_results-LVP}(a),
the automated-driving case in Figure \ref{fig:traj_results-LVP}(b),
even with the extreme acceleration and deceleration rates $\left(\bar{a},\underline{a}\right)$,
relatively better absorbs the backward stopping wave within a fewer
number of vehicles. Reducing the acceleration and deceleration magnitudes
to $\left(\bar{a}/3,\underline{a}/3\right)$ in Figure \ref{fig:traj_results-LVP}(c)
can further smooth vehicle trajectories and dampen the impact from
the stopping wave. These results imply that proper CAV controls can
effectively smooth stop-and-go traffic and reduce backward shock wave
propagation on a freeway.

\begin{figure}
\begin{centering}
\includegraphics[width=0.33\linewidth]{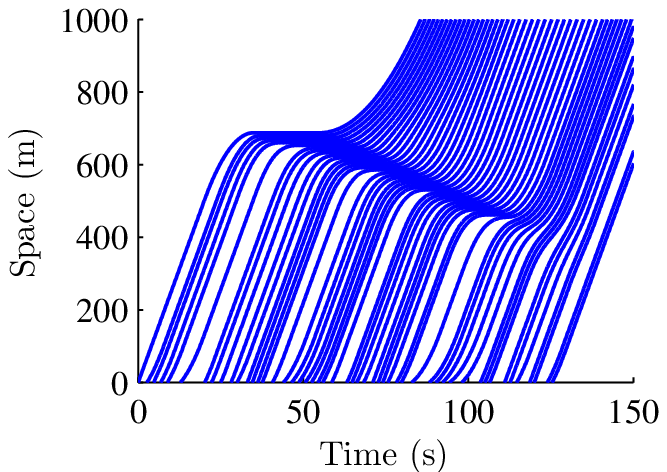}\includegraphics[width=0.33\textwidth]{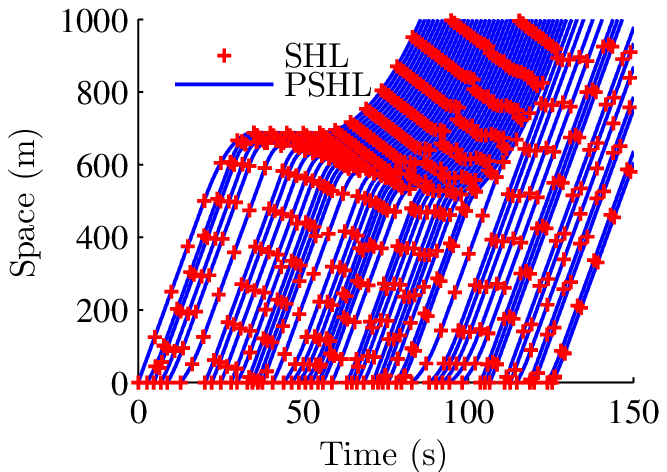}\includegraphics[width=0.33\textwidth]{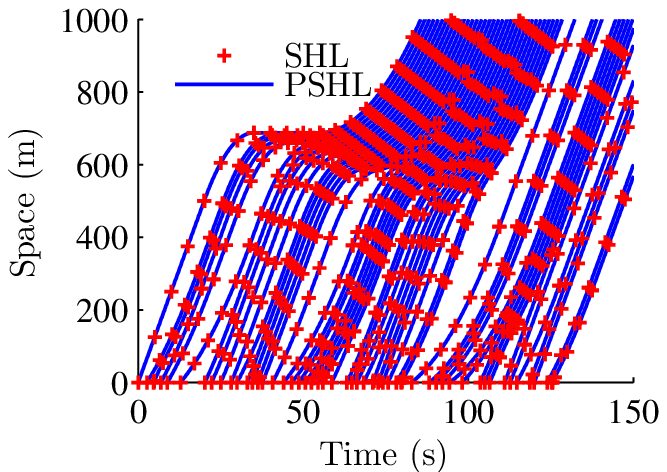}
\par\end{centering}

\begin{centering}
\qquad{}\qquad{}(a)\qquad{}\qquad{}\qquad{}\qquad{}\qquad{}\qquad{}\qquad{}(b)\qquad{}\qquad{}\qquad{}\qquad{}\qquad{}\qquad{}\qquad{}(c)
\par\end{centering}

\protect\caption{(a) Benchmark manual-driving trajectories, (b) SHL and PSHL results
with $\left(\bar{a}^{\mbox{f}},\underline{a}^{\mbox{f}}\right)=\left(\bar{a},\underline{a}\right),$
and (c) SHL and PSHL results with $\left(\bar{a}^{\mbox{f}},\underline{a}^{\mbox{f}}\right)=\left(\bar{a}/3,\underline{a}/3\right).$
\label{fig:traj_results-LVP}}
\end{figure}

\subsection{Feasibility Tests\label{sub:Feasibility-Tests}}

This section conducts some numerical tests to test the feasibility
of the proposed algorithms with different input settings. We first
investigate SHL (or PSHL) for LVP with $\left(\bar{a}^{\mbox{f}},\underline{a}^{\mbox{f}}\right)=\left(\bar{a},\underline{a}\right).$
Theorem \ref{theo: feasibility_all_green} proves that SHL is feasible
if and only if boundary condition $\left[v_{n}^{-},t_{n}^{-}\right]_{n\in\mathcal{N}}$
is proper. Thus we investigate how the feasibility of SHL changes
with the distribution of $v_{n}^{-}$ and $t_{n}^{-}$ . Since the
default boundary condition generation method always assures that $\left[v_{n}^{-},t_{n}^{-}\right]_{n\in\mathcal{N}}$
is proper, this subsection uses a different generation method to allow
$\left[v_{n}^{-},t_{n}^{-}\right]_{n\in\mathcal{N}}$ to be non-proper.
We still use equation \eqref{eq:t_dist} to generate $t_{n}^{-}$,
and thus the dispersion of time headway is controlled by $\alpha$.
In the next step, each $v_{n}^{-}$ is instead randomly pulled along
a uniform interval $\left[(1-\beta)\bar{v},\bar{v}\right]$ where
$\beta\in[0,1]$ is the speed dispersion factor and increases with
the dispersion of the $v_{n}^{-}$ distribution. Note that the final
values of $t_{n}^{-}$ and $v_{n}^{-}$ are randomly generated. We
generate 20 boundary condition instances with the same $\alpha$ and
$\beta$ values yet different random seeds. Then we feed each boundary
condition instance to the SHL algorithm and record the feasibility
of the result. We call the percentage of feasible solutions over all
20 boundary conditions the \emph{feasibility rate} with regard to
this specific parameter setting. Figure \ref{fig:feasibility_PSHL}
plots hot maps for the feasibility rate over $\alpha\times\beta\in[0,1]\times[0,1]$
with different $f^{\mbox{s}}$ values. We can see that overall, as
$\alpha$ and $\beta$ increase, the feasibility rate decreases, and
more instances are infeasible as $f^{\mbox{s}}$ increases. This is
because higher dispersion of $t_{n}^{-}$ and $v_{n}^{-}$ is more
likely to cause conflicts between trajectories that cannot be reconciled
under safety constraint \eqref{eq: set-safety-constraints}, and such
conflicts may increase as traffic gets denser. Note that in all maps
in Figure \ref{fig:feasibility_PSHL}, the transition band between
100\% feasibility rate (the white color) and 0\% feasibility rate
(the black color) is very narrow. This indicates that SHL's feasibility
is dichotomous. With this observation, the parameters could be partitioned
into only two phases (feasible and infeasible) to facilitate relevant
analysis. 

\begin{figure}
\begin{centering}
\includegraphics[height=0.3\textwidth]{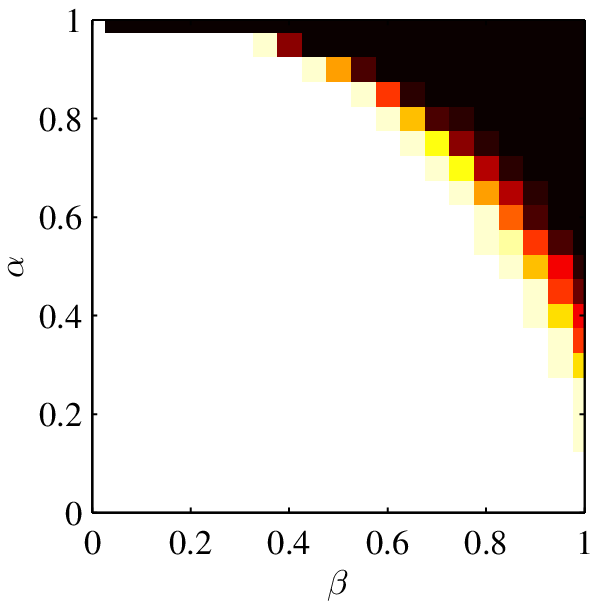}\includegraphics[height=0.3\textwidth]{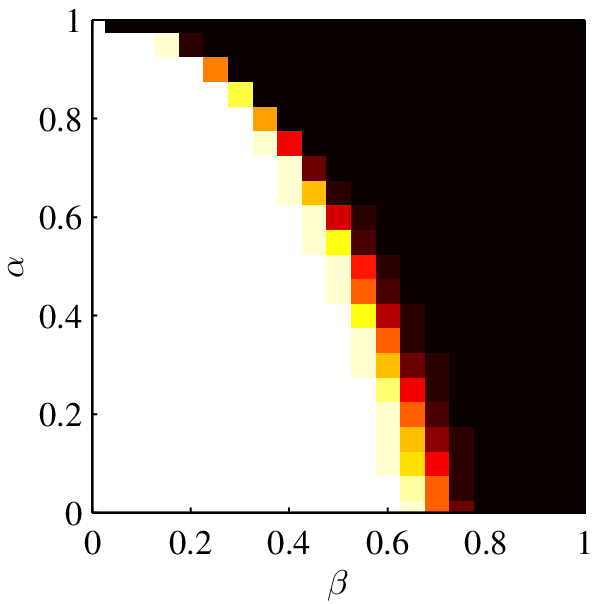}\includegraphics[height=0.3\textwidth]{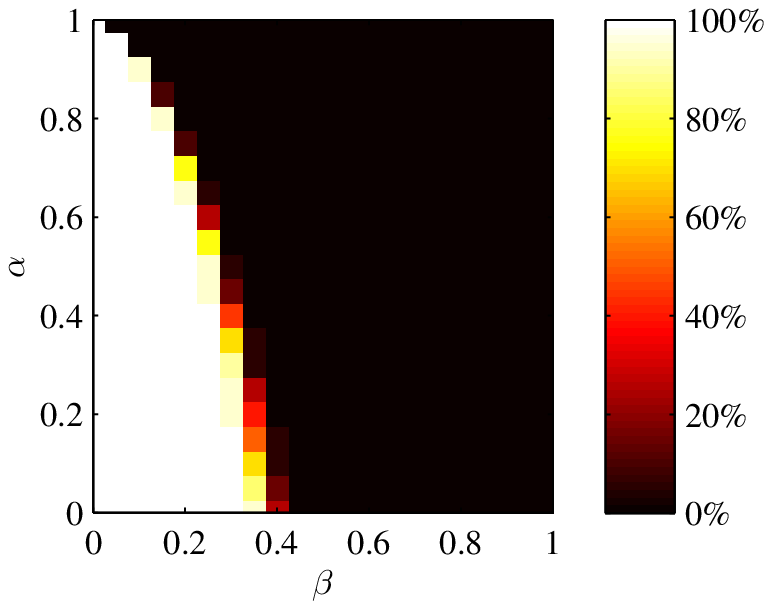}
\par\end{centering}

\begin{centering}
\qquad{}\qquad{}(a)\qquad{}\qquad{}\qquad{}\qquad{}\qquad{}\qquad{}\qquad{}(b)\qquad{}\qquad{}\qquad{}\qquad{}\qquad{}\qquad{}\qquad{}(c)
\par\end{centering}

\protect\caption{Feasibility rate of PSHL with (a) $f^{\mbox{s}}=0.2$, (b) $f^{\mbox{s}}=0.5$,
and (c) $f^{\mbox{s}}=0.8.$ \label{fig:feasibility_PSHL}}
\end{figure}

Next we investigate the SH algorithm considering traffic lights with
$\left(\bar{a}^{\mbox{f}},\underline{a}^{\mbox{f}},\bar{a}^{\mbox{b}},\underline{a}^{\mbox{b}}\right)=\left(\bar{a},\underline{a},\bar{a},\underline{a}\right)$.
We conduct similar experiments as those for Figure \ref{fig:feasibility_PSHL}
with the same adapted boundary condition generation method. Theorems
\ref{theo:P_hat_feasibility} and \ref{theo:P_feasibility} suggest
that the feasibility of SH is related to segment length $L$. Traffic
congestion $f^{\mbox{s}}$ and platoon size $N$ shall also affect
SH's feasibility. This time, we fix $\alpha=\beta=0.5$ and analyze
how the feasibility rate varies with $L$ and $f^{\mbox{s}}$ over
different $N$ values, and the results are shown in Figure \ref{fig:feasibility_SH}.
We see that again the increase of $f^{\mbox{s}}$ raises the chance
of infeasibility. Further, as $L$ decreases, the likelihood of feasibility
diminishes. This is because a short section may not be sufficient
to store enough stopping or slowly moving vehicles to both comply
with the signal phases and allow the following vehicles to enter the
section at their due times. Also, more instances are infeasible as
$N$ increases. This is because more vehicles shall bare a higher
chance of producing an irreconcilable conflict between two consecutive
vehicles against safety constraint \eqref{eq: set-safety-constraints}.
Similarly, the SH feasibility is dichotomous and a two-phase representation
might be applicable. 

\begin{figure}
\begin{centering}
\includegraphics[height=0.3\textwidth]{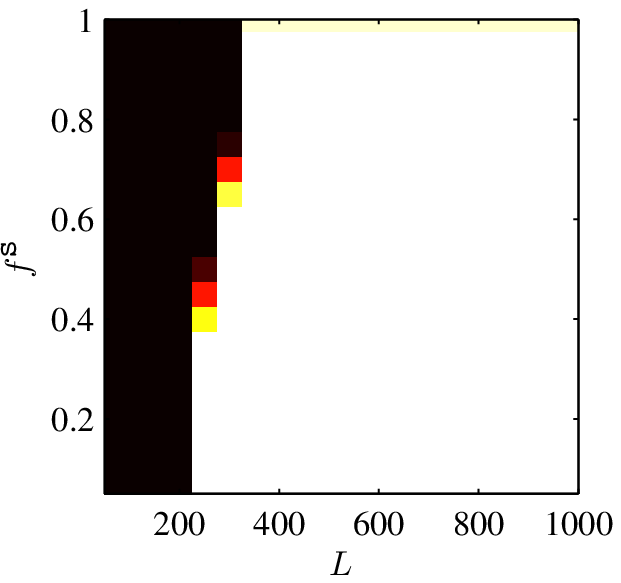}\includegraphics[height=0.3\textwidth]{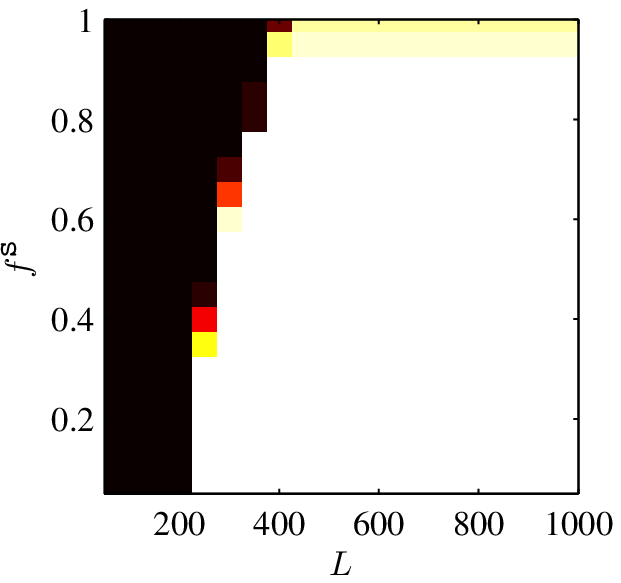}\includegraphics[height=0.3\textwidth]{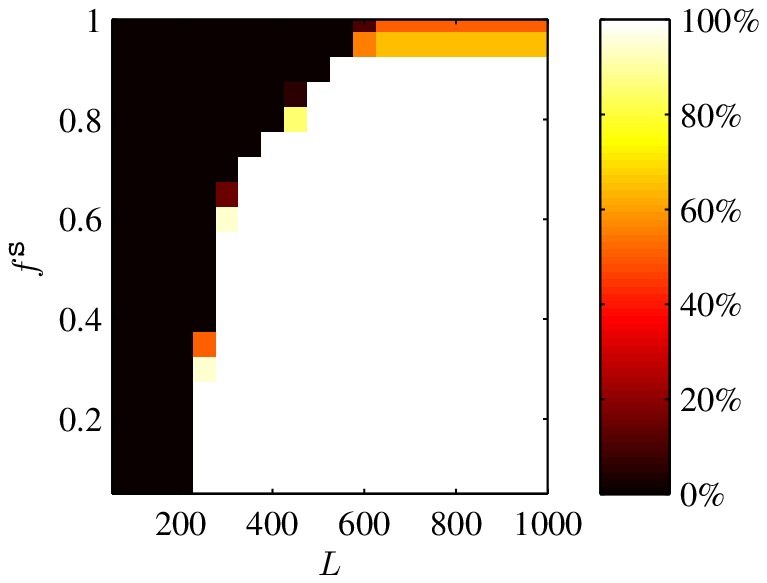}
\par\end{centering}

\begin{centering}
\qquad{}\qquad{}(a)\qquad{}\qquad{}\qquad{}\qquad{}\qquad{}\qquad{}\qquad{}(b)\qquad{}\qquad{}\qquad{}\qquad{}\qquad{}\qquad{}\qquad{}(c)
\par\end{centering}

\protect\caption{Feasibility rate of PSHL with (a) $N=25$; (b) $N=50$; and (c) $N=100.$
\label{fig:feasibility_SH}}
\end{figure}

\subsection{Comparison with Classic Traffic Flow Models\label{sub:Comparison-with-Classic}}

This section compares KWT solution $Q$ and SHL solution $P^{\mbox{LVP}}\left(\underline{a}^{\mbox{f}},\bar{a}^{\mbox{f}}\right)$
for LVP. Again, we set $G=\infty$ and $R=0$, $f^{\mbox{s}}=0.5$.
The lead trajectory is generated with equation \eqref{eq:leading_traj_gen},
and the boundary condition is generated with the default method associated
with equation \eqref{eq:t_dist} to assure the feasibility. Figure
\ref{fig:KWT_SHL_Traj} compares these trajectories with different
$\left(\bar{a}^{\mbox{f}},\underline{a}^{\mbox{f}}\right)$ values.
In this section we allow $\left(\bar{a}^{\mbox{f}},\underline{a}^{\mbox{f}}\right)$
go beyond $\left(\bar{a},\underline{a}\right)$ so as to investigate
the asymptotic properties of SHL. We see that in general, trajectories
in $Q$ have abrupt turns while those in $P^{\mbox{LVP}}\left(\underline{a}^{\mbox{f}},\bar{a}^{\mbox{f}}\right)$
are relatively smooth. All trajectories in $P^{\mbox{LVP}}\left(\underline{a}^{\mbox{f}},\bar{a}^{\mbox{f}}\right)$
is below those in $P^{\mbox{LVP}}\left(\underline{a}^{\mbox{f}},\bar{a}^{\mbox{f}}\right)$,
which is consistent with the upper bound property stated in Theorem
\ref{theo:KWT_SHL_bounds}. As $\left(\bar{a}^{\mbox{f}},\underline{a}^{\mbox{f}}\right)$
amplifies from $\left(\bar{a}/3,\underline{a}/3\right)$ to $\left(3\bar{a},3\underline{a}\right)$,
we see that accelerations and decelerations in $P^{\mbox{LVP}}\left(\underline{a}^{\mbox{f}},\bar{a}^{\mbox{f}}\right)$
become sharper and trajectories in $P^{\mbox{LVP}}\left(\underline{a}^{\mbox{f}},\bar{a}^{\mbox{f}}\right)$
get closer to those in $Q$, which is consistent with the asymptotic
property stated in Corollary \ref{cor_KWT_SH_asymptotic}. 

\begin{figure}
\begin{centering}
\includegraphics[width=0.33\textwidth]{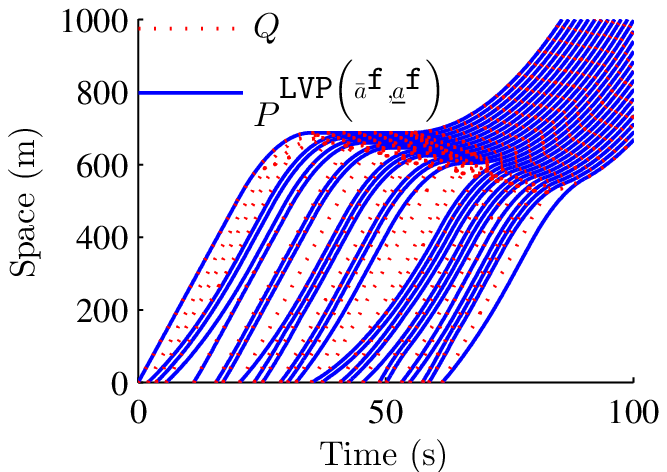}\includegraphics[width=0.33\linewidth]{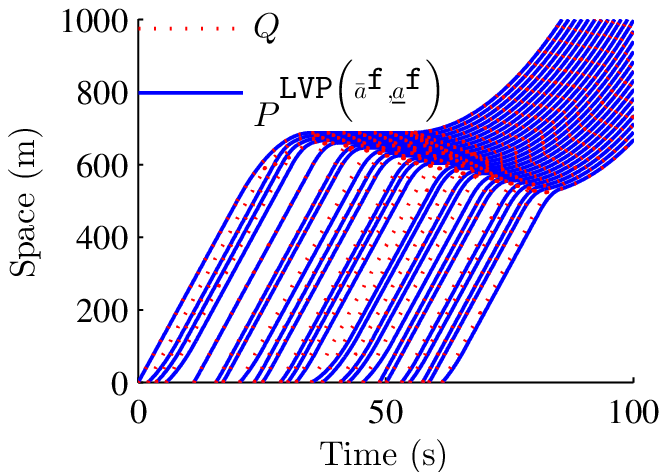}\includegraphics[width=0.33\linewidth]{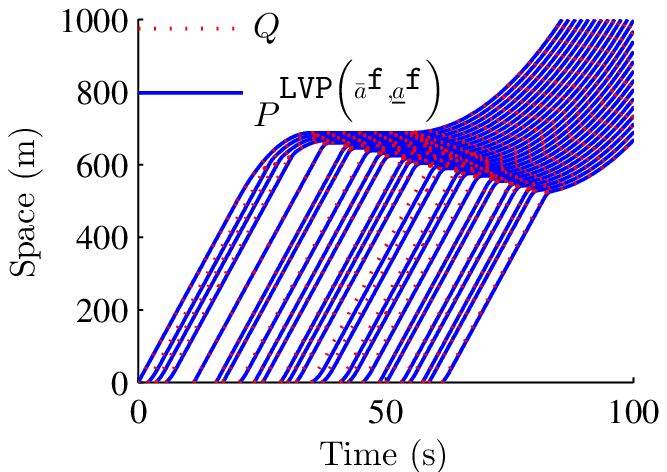}
\par\end{centering}

\begin{centering}
\qquad{}\qquad{}(a)\qquad{}\qquad{}\qquad{}\qquad{}\qquad{}\qquad{}\qquad{}(b)\qquad{}\qquad{}\qquad{}\qquad{}\qquad{}\qquad{}\qquad{}(c)
\par\end{centering}

\protect\caption{Comparison between KWT solution $Q$ and SHL solution $P^{\mbox{LVP}}(\bar{a}^{\mbox{f}},\underline{a}^{\mbox{f}})$
with (a) $\left(\bar{a}^{\mbox{f}},\underline{a}^{\mbox{f}}\right)=\left(\bar{a}/3,\underline{a}/3\right)$,
(b) $\left(\bar{a}^{\mbox{f}},\underline{a}^{\mbox{f}}\right)=\left(\bar{a},\underline{a}\right)$,
and (c) $\left(\bar{a}^{\mbox{f}},\underline{a}^{\mbox{f}}\right)=\left(3\bar{a},3\underline{a}\right)$.\label{fig:KWT_SHL_Traj}.}
\end{figure}

Figure \ref{fig:KWT_SHL_Error} plots the errors between $Q$ and
$P^{\mbox{LVP}}\left(\gamma\bar{a},\gamma\underline{a}\right)$ and
their bounds (defined in Theorem \ref{theo:KWT_SHL_bounds}), where
acceleration factor $\gamma$ increases from $1/3$ to $3$. We see
that $D\left(q_{n}-p_{n}\right)$ is always identical to 0 (or the
upper bound) and $D\left(p_{n}-q_{n}\right)$ is always above the
lower bound \eqref{eq:SH-KWT lower bound} for all $\gamma$ values.
As $\gamma$ increases, the errors and their bounds all converge to
0, which again confirms Corollary \ref{cor_KWT_SH_asymptotic}. In
summary, these experiments show that SHL can be viewed as a smoothed
version of KWT that replaces speed jumps in KWT with smooth accelerations
and decelerations. 

\begin{figure}
\begin{centering}
\includegraphics{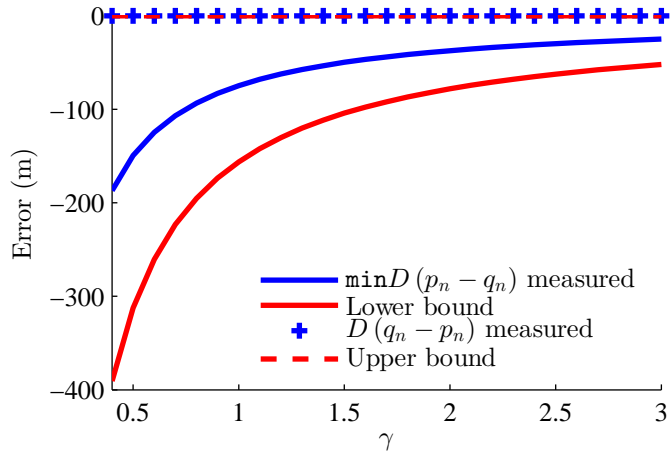}
\par\end{centering}

\protect\caption{Measured errors between KWT solution $Q$ and SHL solution $P^{\mbox{LVP}}\left(\gamma\bar{a},\gamma\underline{a}\right)$
v.s. their bounds \label{fig:KWT_SHL_Error}.}
\end{figure}

\section{Conclusion\label{sec:Conclusion}}

This paper investigates the problem of controlling multiple vehicle
trajectories on a highway with CAV technologies. We propose a shooting
heuristic to efficiently construct vehicle trajectories that follow
one another under a number of constraints, including the boundary
condition, physical limits, following safety, and traffic signals.
With slight adaptation, this heuristic is applicable to not only highway
arterials with interrupted traffic but also uninterrupted freeway
traffic. We generalize the time geography theory to consider finite
accelerations. This allows us to study the behavior of the proposed
algorithms. We find that the proposed algorithms can always find a
feasible solution to the original complex multi-trajectory control
problem under certain mild conditions. We further point out that the
shooting heuristic solution to the lead vehicle problem can be viewed
as a smoothed version of the classic kinematic theory's result. We
find that the kinematic wave theory is essentially a special case
of the proposed shooting heuristic with infinite accelerations. Further,
the difference between the shooting heuristic solution and the kinematic
wave solution is found to be limited within two theoretical bounds
independent of the size of the vehicles in the studied traffic stream.
Numerical experiments are conducted to illustrate some theoretical
results and draw additional insights into how the proposed algorithms
can improve highway traffic. 

This paper provides a methodological and theoretical foundation for
management of future CAV traffic. The following part II paper \citet{Ma2015}
of this study will apply the proposed constructive heuristic with
given acceleration rates to a trajectory optimization framework that
optimizes the overall performance of CAV traffic (e.g., in terms of
mobility, environment impacts and safety) by finding the best acceleration
rates. Theoretical results and numerical examples on this optimization
framework will be presented. This study overall expects to provide
both theoretical principles and application guidance for upgrading
the existing highway traffic management systems with emerging CAV
technologies. It can be extended in a number of directions to address
practical challenges and emerging opportunities in deploying CAV technologies,
such as calibration with field experiments, more complex geometries,
heterogeneous vehicles, and mixed manual and automated traffic.

\section*{Acknowledgments}

This research is supported in part by the U.S. National Science Foundation
through Grants CMMI CAREER\#1453949, CMMI \#1234936 and CMMI \#1541130
and by the U.S. Federal Highway Administration through Grant DTFH61-12-D-00020. 

\bibliographystyle{plainnat}
\bibliography{Literature}

\end{document}